\documentclass[11pt]{article}

\usepackage{amsmath, amssymb, mathtools}
\usepackage{amsthm} 
\usepackage{enumerate}
\usepackage[hyphens]{url}
\usepackage{hyperref}
\usepackage{xcolor}
\hypersetup{colorlinks,
linkcolor={blue!50!black},
citecolor={red!50!black},
urlcolor={blue!70!black}}
\usepackage{graphicx}
\usepackage{graphbox}
\usepackage{microtype}
\usepackage{siunitx}
\usepackage{booktabs}
\usepackage{cleveref}
\usepackage{todonotes}
\usepackage{verbatim}
\usepackage{enumerate}
\usepackage{float}
\usepackage{caption}
\usepackage{subcaption}
\usepackage{authblk}
\usepackage[top=1.3in, 
bottom=1.3in, 
left=1.1in, 
right=1.1in]{geometry}
\usepackage{shuffle}
\usepackage{quiver}
\usepackage{upgreek}
\theoremstyle{plain}
\newtheorem{theorem}{Theorem}[section]
\newtheorem{proposition}[theorem]{Proposition}
\newtheorem{corollary}[theorem]{Corollary}
\newtheorem{lemma}[theorem]{Lemma}
\theoremstyle{definition} 

\newtheorem{example}[theorem]{Example}
\theoremstyle{remark}
\newtheorem{remark}[theorem]{Remark}

\newcommand{\cA}{\mathcal{A}}
\newcommand{\cB}{\mathcal{B}}

\newcommand{\cD}{\mathcal{D}}

\newcommand{\cK}{\mathcal{K}}

\newcommand{\cT}{\mathcal{T}}

\newcommand{\cW}{\mathcal{W}}

\newcommand{\CC}{\mathbb{C}}

\newcommand{\NN}{\mathbb{N}}

\newcommand{\RR}{\mathbb{R}}

\newcommand{\ba}{\mathbf{a}}

\newcommand{\bx}{\mathbf{x}}

\newcommand{\by}{\mathbf{y}}

\newcommand{\bell}{\pmb{\ell}}

\newcommand{\dd}{\mathop{}\!\mathrm{d}}
\newcommand{\subhol}[2][T]{{#2\text{-Höl};[0,#1]}}
\newcommand{\suphol}[1]{{#1\text{-Höl}}}
\newcommand{\subvar}[2][T]{{#2\text{-var};[0,#1]}}

\title{Dynamic Universal Approximation via Signature \\ Controlled Differential Equations}
\author[1]{Tomás Carrondo}
\author[1]{Christa Cuchiero}
\author[1]{Paul P. Hager \thanks{\today. Corresponding author: \textit{paul.peter.hager@univie.ac.at} \\
This research was funded in whole or in part by the Austrian Science Fund (FWF) Y 1235. FH gratefully acknowledges support from SURE-AI research center funded by the Research Council of Norway, grant 357482. }}
\author[2]{Fabian N. Harang}
\date{}
\affil[1]{University of Vienna, Department of Statistics and Operations Research}
\affil[2]{BI Norwegian Business School, Department of Data Science and Analytics}

\begin{document}

\maketitle
\vspace{-1cm}
\begin{abstract}
We study signature controlled differential equations (Sig-CDEs), that is, path-dependent controlled differential equations (CDEs) whose vector fields factor through the signature map. Working on spaces of stopped Hölder paths, we develop an existence, uniqueness, and stability theory for general path-dependent CDEs, and translate these pathwise well-posedness criteria into conditions on the corresponding signature functionals. We then prove dynamic universality:  simply parametrized Sig-CDEs approximate the solution path of any well-posed path-dependent CDE arbitrarily well, uniformly over bounded sets of controls and initial histories, with global variants obtained using weighted spaces. Within this framework, entire maps of group-like elements provide a specific class of Sig-CDEs. Using a new class of limiting tensor spaces, we recast Sig-CDEs as infinite-dimensional classical CDEs and prove their well-posedness via a gauge-type scaling argument, thereby establishing a principled way to lift generic path-dependent dynamics. Lastly, we study truncated Sig-CDEs as finite-dimensional differential equations on step$-N$ Lie groups under intrinsic conditions, that is, with well-posedness formulated in terms of the underlying group metric.

\end{abstract}

\tableofcontents

\section{Introduction}
The classical theory of ordinary differential equations is built on the premise that the instantaneous rate of change of a system depends only on its current state. Although this assumption underlies a highly successful modeling paradigm in mathematics and the sciences, it is often an idealization that overlooks the role of memory and past behavior. In many complex systems, such as biological population dynamics \cite{hutchinson1948circular,kermack1927contribution}, viscoelastic materials \cite{volterra1909sulle,boltzmann1874}, and feedback-controlled mechanisms \cite{minorsky1922directional}, the evolution depends not only on the present state but also on the system's history. Such systems are commonly referred to as \emph{hereditary systems} or \emph{systems with memory} \cite{volterra1928theorie}.

A natural mathematical formulation is therefore to allow the vector field to depend on the history of the solution. In the controlled setting considered here, this leads to \emph{path-dependent controlled differential equations} (path-dependent CDEs), which we write as
\begin{equation}
\label{eq_path_CDE_intro}
\dd y_t = f(y|_{[0,t]})\, \dd x_t,\quad t\in [T_0,T_1], \qquad
y|_{[0,T_0]} = w,
\end{equation}
where $w:[0,T_0]\to\mathbb R^m$ is a prescribed initial history, $x:[0,T_1]\to\mathbb R^d$ is a \emph{control path}, and $f$ is a non-anticipative functional with values in $L(\mathbb R^d,\mathbb R^m)$, the space of linear maps from $\mathbb R^d$ to $\mathbb R^m$. Thus, at time $t$, the dynamics may depend on the entire stopped trajectory $y|_{[0,t]}$, rather than only on the current state $y_t$. Throughout the paper, $x$ is assumed to be Lipschitz so the corresponding integral formulation is understood in the Riemann-Stieltjes sense. Such systems have been studied from various viewpoints with numerous variants and extensions, e.g., with  rougher and/or stochastic control paths. We refer to \Cref{sec:rel_lit} for a detailed account of the literature.

Even though these inherently infinite-dimensional systems are quite well understood from a theoretical perspective, there is still no canonical and universally applicable framework for systematically reducing them to finite-dimensional equations while preserving certain important learning-theoretic criteria as specified below. Hence, our main motivation, phrased in general terms, is to show that, assuming only well-posedness conditions on $f$ (in stronger topologies than the ones available in the literature), any system of the form \eqref{eq_path_CDE_intro} can be approximated by a finite-dimensional state equation augmented through linear feedback. That is, by systems of the form
\begin{equation}
\label{eq:n_dim_approx_intro}
\begin{cases}
    \dd y^n_t = f_n(z^n_t) \dd x_t, \qquad  &t \in [T_0,T_1] \\[5pt]
    \dd z^n_t = A_n(z^n_t) \dd y^n_t, \qquad  &t \in [0,T_1]
\end{cases}, 
\qquad y^n|_{[0,T_0]} = w,
\end{equation}
with linear operators $A_n \colon \mathcal{X}_n \to L(\RR^m,\mathcal{X}_n)$ and elementary readout functions $f_n \colon \mathcal{X}_n \to L(\mathbb{R}^d,\mathbb{R}^m)$ on finite-dimensional state spaces $\mathcal{X}_n$.
Due to the evident gain in tractability, the study of such finite-dimensional approximations is of eminent practical importance and has therefore been treated in the literature from several angles; see Section \ref{sec:rel_lit} for a review. Within this broader literature, our approach is then distinguished by its universality and learning-theoretic emphasis:
\begin{enumerate}
	\item \emph{No structural or model assumptions} on the path-functional $f$ beyond local Lipschitz continuity and linear growth on H\"older path spaces (in contrast to the standard weaker uniform topology).
	\item \emph{Globally well-posed approximating systems} of the form \eqref{eq:n_dim_approx_intro} that are differentiable with respect to the \emph{learnable parameters} determining the maps $f_n$ and $A_n$.
	\item A \emph{dynamic linear feedback} representation of the latent state $z^n$, namely, as a CDE driven by the observed trajectory $y^n$.
    \item An interpretation of $z^n$ as a \emph{well-suited feature map} that depends continuously on $y^n$ with respect to Hölder norms.
	\item \emph{Robustness,} in the sense of (i) \emph{uniform approximation} over bounded sets of potentially  \emph{irregular input paths}
    $w$ in Hölder norms rather than derivative-based norms, 
    and (ii) \emph{global approximation at the control path level} with respect to a weighted norm.
\end{enumerate}
The first two points reflect the central objectives of \emph{universal approximation theory}: to approximate broad classes of systems under minimal structural assumptions while retaining a parametrization that can be learned from data by gradient-based methods. The dynamic linear feedback representation in the third point, in which $z^n$ evolves as a CDE driven by the observed trajectory $y^n$, is closely aligned with the role of latent states in state-space models and recurrent architectures \cite{kalman1960new,baum1966statistical,jordan1986serial,elman1990finding}. Such a representation is particularly relevant in applications such as filtering and reinforcement learning \cite{kaelbling1998planning,sutton2018reinforcement}. The fourth and fifth points are linked by the choice of Hölder topology: viewing $z^n$ as a continuous feature map on Hölder path spaces provides stability with respect to peturbations of the observed trajectory and, consequently, robustness at the level of the input paths. This is especially important for oscillatory, noisy, high-frequency, or irregularly sampled data, for which derivative-based notions of approximation may be too restrictive.

The concrete approach pursued here is based on path signatures in the sense of Chen \cite{chen1954iterated,chen1957integration,chen1977iterated} and Lyons \cite{lyons1998differential, Lyons2002,lyons2007differential}, and provides a canonical way to construct finite-dimensional approximations \eqref{eq:n_dim_approx_intro} by truncating a graded path representation. Specifically, the signature $S(y)_{0,\cdot}$ of an $\alpha-$Hölder path $y$ with $1/2 <\alpha \leq 1$ is the solution of the fundamental linear CDE
\begin{equation}
\label{eq:signature_fundamental_cde}
    \dd S(y)_{0,t} = S(y)_{0,t} \otimes \dd y_t, 
    \qquad S_{0,0} = \mathbf{1} \in T((\RR^m)),
\end{equation}
taking values in the extended tensor algebra $T((\RR^m)) = \prod_{n=0}^\infty (\RR^m)^{\otimes n}$. Solving this equation level by level, one obtains the representation in terms of iterated Young integrals (see \Cref{subsec_path_signatures_tensor_algebras} for a more detailed introduction)
\[
S(y)_{0,T} =
\left(
1,
\int_0^T \dd y^1_t,\ldots,\int_0^T \dd y^m_t,
\int_{0<t_1<t_2<T} \dd y^1_{t_1}\dd y^1_{t_2},
\int_{0<t_1<t_2<T} \dd y^1_{t_1}\dd y^2_{t_2},
\ldots
\right).
\]
Denoting by $\pi_{0,N_n} : T((\RR^m)) \to T^{N_n}(\RR^m)$ the projection onto the first $N_n$ tensor levels, the truncated signature $z^{n} = \pi_{0,N_n}(S(y))$ defines a finite-dimensional latent state that itself solves a finite-dimensional CDE on $T^{N_n}(\RR^m)$ with linear vector fields $A_n(z)v = \pi_{0,N_n}(z \otimes v)$. This means that  we then deal with \emph{signature CDEs} of the form
\begin{equation}
\label{eq:approx_sig_cde}
    \dd y^n_t = f_n\big( \pi_{0,N_n}( S(y^n)_{0,t})\big) \dd x_t.
\end{equation}
As part of our main results — see Theorems \ref{prop_dynamic_universality_activations} and \ref{thm_global_approx_control_level} — we show that finite-dimensional signature systems remain \emph{dynamic universal approximators} for \eqref{eq_path_CDE_intro}, even when their readouts are restricted to the elementary form 
\[f^{ij}_n(z) = c_0 + c_1\sigma( \langle \bell_n^{ij}, z \rangle),\]
where $\bell^{ij}_n \in T^{N_n}(\RR^m)$, $c_0, c_1 \in \mathbb R$, and $\sigma:\RR\to\RR$ is a sufficiently regular activation function. More precisely, we prove there exists a sequence of coefficients $(\bell_n)_{n\ge1}$ such that the corresponding solutions of \eqref{eq:approx_sig_cde} converge in Hölder spaces to the solution $y$ of \eqref{eq_path_CDE_intro}. This convergence is uniform over bounded sets of input pairs $(x,w)$ in Hölder norm and, when measured in suitable weighted norms, holds globally with respect to the control path $x$.

Returning to the original equation \eqref{eq_path_CDE_intro}, suppose that the path-dependent vector field $f$ admits a decomposition $f = F \circ S$, where $F: \mathcal{X} \to L(\RR^d,\RR^m)$ and $\mathcal{X} \subset T((\RR^m))$ is some suitable topological space.\footnote{Note that $F$ can always be defined as $f \circ S^{-1}$ on the image of the signature map.} We then obtain the signature CDE
\begin{align}
\label{eq:signature_cde_intro}
\dd y_t= F(S(y)_{0,t}) \dd x_t, \qquad t \in [T_0,T_1],
\end{align}
which may now depend on \emph{infinitely} many signature components. Building on our well-posedness result for path-dependent CDEs on Hölder spaces, Theorem \ref{prop_EU_global_sol_pd_cde}, we establish global existence and uniqueness of solution for \eqref{eq:signature_cde_intro} just in terms of Lipschitz continuity and \emph{homogeneous linear growth} conditions on $F$; see Proposition \ref{prop_well_posedness_SigCDEs}. In particular, this result applies to the finite-dimensional dynamic universal approximators introduced above and guarantees that the corresponding systems are well-posed. 

Proceeding one step further leads naturally to an infinite-dimensional state-space lift,  which is obtained by combining \eqref{eq:signature_cde_intro} with
\eqref{eq:signature_fundamental_cde}, so that
\begin{equation}
\label{eq:full_lift_intro}
    \dd S_{t} = S_{t} \otimes F(S_t)\dd x_t.
\end{equation}
An immediate follow-up question is therefore to 
identify the ``best'' state-space $\mathcal{X}$, together with a suitable topology, and to 
establish well-posedness of 
\eqref{eq:full_lift_intro} on this space. Such a theory is not only of intrinsic interest, but may also provide a natural starting point for a semigroup-theoretic treatment of signature-lifted evolution equations on Banach spaces, in analogy with the classical state-space approach to delay equations, infinite-dimensional control systems, and Volterra equations; see, among many others,
\cite{hale2013introduction,BensoussanDaPratoDelfourMitter2007,miller1974linear}.

Choosing the space $\mathcal X$, and hence a topology on the set of signatures, is nevertheless a delicate point. Indeed, requiring both the signature map $S$ and the vector field $F$ to be locally Lipschitz into and out of $\mathcal X$, respectively, imposes regularity conditions on the original path-dependent functional $f = F \circ S$, which we would like to keep as close as possible to our minimal well-posedness assumptions. At the same time, we want $\mathcal X$ to have a strong enough structure to control the infinite-dimensional dynamics \eqref{eq:full_lift_intro}, while its topology — or, more precisely, its metric — should be weak enough to preserve the local Lipschitz continuity of the signature lift. One possible choice for $\mathcal X$ is equipping $T((\RR^m))$ with the usual $\ell^1-$topology. However, as discussed above, in order to match our well-posedness conditions on path space, we aim instead for Banach spaces with stronger topologies, which facilitate the Lipschitz continuity of $F$ and allow global existence for \eqref{eq:full_lift_intro} to be established by means of gauge-type scaling arguments; see Lemma \ref{lem_aux_scaling_arg_2}.

All these considerations lead naturally to the introduction of projective limits of tensor spaces adapted to signature decay rates, as defined in \eqref{eq:limitTp}, and pushed further to a second projective limit in~\eqref{eq_second_projective_limit}. We believe that these limiting spaces provide the most suitable setup, in the sense described above, namely \eqref{eq:limitTp} as domain of $F$
and \eqref{eq_second_projective_limit}
as the state space for which \eqref{eq:full_lift_intro} has a continuous flow. The well-posedness result for \eqref{eq:full_lift_intro}, see Theorem \ref{thm_scaling_argument}, crucially relies on this choice of spaces. Summarizing on a methodological level, our program comprises the following contributions:
\begin{enumerate}[(i)]
	\item an analytic framework for path-dependent CDEs \eqref{eq_path_CDE_intro} on Hölder path spaces, including local and global well-posedness as well as continuity of the solution map,
	\item the construction of Banach spaces $\mathcal{X} \subset T((\mathbb{R}^m))$ via projective limits of weighted $\ell^1$-spaces, which dissect sets of signatures according to the regularity of the underlying paths,
	\item global well-posedness for the resulting signature controlled differential equations (Sig-CDEs), both at path level \eqref{eq:signature_cde_intro} and at the lifted level \eqref{eq:full_lift_intro}, under explicit tensor-level linear growth and Lipschitz continuity conditions,
	\item a dynamic universal approximation result for path-dependent systems by Sig-CDEs, with global approximation variants through the use of weighted spaces, 
	\item extensions of the theory to refined settings, including entire signature functionals and formulations on step-$N$ groups.
\end{enumerate}

\subsection{Related Literature}\label{sec:rel_lit}
The content of this paper sits at the intersection of several topics; below we organize the related literature according to the themes most relevant to our analysis.

\paragraph{Functional differential equations and path-dependent CDEs.} The modern theory of \emph{functional differential equations} (FDEs) may be viewed as a consolidation of Volterra's work on integro-differential equations for hereditary phenomena \cite{volterra1913leccons}; see \cite{hale2013introduction} for a classical account. This framework encompasses ordinary differential equations, differential-difference equations, and integro-differential equations, and includes many variants such as non-autonomous equations, equations with extended or reduced memory, and systems with delays in the derivative; see, for example, \cite{corduneanu1966certaines,jones1967hereditary,hale1970existence,neustadt1970solutions,imaz1966generalized}. Our well-posedness theory for path-dependent CDEs in Section~\ref{sec_path_dependent_cdes} follows the exposition in \cite{hale2013introduction}, where it is emphasized that the fundamental mechanisms underlying FDE well-posedness closely parallel those of classical ODE theory. The main point of departure is the topology on path space. Much of the classical FDE literature works with path segments endowed with the uniform topology, whereas we work on spaces of stopped H\"older paths. This is technically necessary for the present paper, since signatures are not defined for arbitrary continuous paths, but it is also of independent interest: because H\"older topologies are stronger than the uniform topology, assumptions formulated in such topologies may be weaker than the corresponding assumptions imposed uniformly over path segments. Moreover, it automatically leads to robustness properties with respect to irregular input paths and continuous feature maps, such as the signature.

The closest setting to ours is \cite{andretto2019path}, which studies the same class of path-dependent CDEs in the Young regime, assuming an $\alpha-$Hölder control $x$ with $\alpha>1/2$, and proves global existence directly. By contrast, we restrict attention to Lipschitz controls and follow a more classical strategy: we first establish local existence and uniqueness under minimal assumptions and then obtain global solutions via an extension argument. In addition, we derive explicit stability estimates for the solution map, a feature which is absent from both \cite{hale2013introduction} and \cite{andretto2019path} but is crucial for transferring approximation results from vector fields to solutions. We note that our estimates and proof techniques do not extend to the Young setting, since Young's inequality does not produce the Grönwall-type bounds our proof needs.

\paragraph{Stochastic and rough path-dependent dynamics.}

A closely related stochastic counterpart is the theory of \emph{stochastic functional differential equations}, also known as path-dependent SDEs. In the classical finite-memory setting, where the coefficients depend on a segment of the past trajectory, the theory was developed systematically in \cite{mohammed1094}; see also \cite{maostochastic2007} for a broad account of stochastic delay equations. A more general non-anticipative formulation, in which the coefficients are predictable functionals of the entire past trajectory, was studied in \cite{rogers2000diffusions2}. We also refer to the monograph \cite{protter2004stochastic} for a treatment in a general semimartingale setting.  These viewpoints are closer to the stopped-path formulation used here and connect naturally with functional Itô calculus, where non-anticipative functionals are treated as maps on spaces of stopped paths \cite{cont2013functional,dupire2009functional}. As in the deterministic setting most formulations are with respect to the uniform topology, but recently also Hölder topologies have been considered; see, for example, \cite{cont2019support}.

FDEs have also been studied through the lens of rough path theory \cite{lyons2007differential,friz2014course}, which provides a pathwise framework for treating irregular driving signals. In this setting, \emph{rough functional differential equations} correspond to path-dependent CDEs driven by rough paths of regularity $\alpha \leq 1/2$. Such equations have been analyzed, for instance, in \cite{kwossek2024functional,ananova2023rough,aida2024rough} and the references therein.
 The increased irregularity of the driver requires substantially stronger assumptions on the vector field, such as $C^3-$regularity on path space with bounded derivatives, and therefore falls outside the scope of the present work. Nevertheless, many of the techniques, notation, and tools we use are native to rough path theory, most notably the signature.

\paragraph{Signature-based models} 

In the deterministic and controlled setting, signature-based ideas have entered differential equation modeling through neural CDEs for irregular time series \cite{kidger2020neural}, neural rough differential equations (RDEs) driven by log-signature components of the input path \cite{morrill2020neural}, path-dependent neural jump ODEs based on the signature transform \cite{krach2022optimal}, signature-encoded continuous-time models for non-Markovian dynamics \cite{pradeleix2025learning},  neural RDE methods for continuous-time non-Markovian stochastic control \cite{hoglund2023neural}, and log-neural CDE architectures that exploit the Log-ODE method and the Lie-bracket structure to train CDE models more efficiently \cite{walker2024log}.

In probabilistic settings, \emph{signature stochastic differential equations} (Sig-SDEs) were introduced in \cite{cuchiero2023signature}, building on earlier signature-based financial models \cite{arribas2020sig}. Their tractability and approximation properties have since led to applications in path-dependent financial modeling; see, for example, \cite{cuchiero2023signatureSIAM,cuchiero2025joint, cuchiero2025universal, abijaber2025martingale,abi2025signature}. Parameter estimation for Sig-SDEs was studied in \cite{semnani2025path} through an analogue of moment matching known as \emph{expected signature matching} \cite{anastasia2011}. While Sig-SDEs are often presented as universal within suitable classes of stochastic processes, a dynamic universal approximation theorem analogous to the one proved here for path-dependent CDEs has not yet been established; this is the subject of an accompanying work. The present paper may be viewed as the deterministic counterpart.

\paragraph{Control theory and Lie-group formulations.}
The celebrated Chen--Fliess series form a classical bridge between iterated integrals and nonlinear control theory, with origins in the 1970s \cite{JurdjevicSussmann1972,Brockett1976,Fliess1981,Sussmann1983}. In this setting, the signature appears naturally as the collection of variables in the expansion; see \cite{duffaut2026chen} for a recent comparison. Despite this connection, and to the best of our knowledge, signatures have not previously been used to lift path-dependent controlled systems in the sense considered here. The classical control-theoretic approach to retarded controlled differential equations is instead to lift $(y_t,y|_{[t-r,t]})$ to a functional state space $\mathbb R^m\times H([0,r],\mathbb R^m)$, where $H$ is typically an $L^p$ or $W^{1,p}$ Sobolev space; see \cite[Section~II.4]{BensoussanDaPratoDelfourMitter2007}. This perspective has been used to study finite-dimensional approximations and their qualitative properties, including convergence of approximating semigroups \cite{KappelSalamon1987,ItoKappel1991}, preservation of stability and controllability \cite{KappelSalamon1989}, and convergence of  feedback laws \cite{KappelSalamon1987}.

The perspective taken here, however, is different. Rather than lifting the past trajectory to a function space, we encode it through its signature. At finite level $N$, this places the memory state in the free step-$N$ nilpotent Lie group $G^N(\mathbb R^m)$, or alternatively in the tensor group $\mathbf 1+\mathfrak t^N(\mathbb R^m)$. Thus, although our starting point is a path-dependent controlled system, the truncated signature lift naturally leads to a controlled differential equation on a finite-dimensional Lie group. This connects the paper to geometric control theory, where controlled dynamics on manifolds and Lie groups play a central role; see, for instance, \cite{agrachev2004}. Additionally, in our setting, the relevant Lie groups carry a further homogeneous structure inherited from tensor dilations; see Section \ref{subsec_path_signatures_tensor_algebras}. It is therefore natural to ask whether well-posedness of the lifted equations can be formulated intrinsically, in terms of homogeneous gauges or group metrics, rather than through an Euclidean norm on coordinates. This question is closely related to \cite{magnani2018lipschitz}, where Cauchy problems on homogeneous groups are studied and it is shown, in particular, that Lipschitz regularity with respect to a homogeneous group metric is not sufficient in general to guarantee uniqueness. This is in contrast to global existence which can be obtained under intrinsic growth assumptions, as we show in Section \ref{subsec_not_so_esoteric}. Although we do not pursue the related control-theoretic questions in this paper, these connections suggest that signature lifts may provide a useful alternative viewpoint on approximation theory for controlled path-dependent systems.

\section{Preliminaries}
\label{sec_preliminaries}

We begin by reviewing the background material that will be used throughout the paper. Most of the content is standard, with references provided where appropriate. Section \ref{subsec_path_spaces} introduces general notation together with the relevant path spaces. Section \ref{subsec_path_signatures_tensor_algebras} summarizes the main aspects of path signatures and their algebraic framework. Section \ref{subsec_weighted_spaces_global_approximations} presents the key tools from the theory of weighted spaces and global approximations. And lastly, Section \ref{subsec_limits_of_tensor_spaces} develops a class of limit tensor spaces tailored to the analysis of Sig-CDEs.

\subsection{Hölder Path Spaces}
\label{subsec_path_spaces}

Let $(\RR^d, |\cdot|)$ be the standard Euclidean space and fix $T > 0$. We write $C^0([0,T], \RR^d)$ for the space of continuous $\RR^d$-valued paths on $[0,T]$. For $0 \leq s < t \leq T$ and $x \in C^0([0,T], \RR^d)$, we set $x_t \coloneqq x(t)$ and $x_{s,t} \coloneqq x_t - x_s$. The space $C^0([0,T],\RR^d)$ is a Banach space when equipped with the usual supremum norm 
$$\| x\|_{\infty;[0,T]} \coloneqq \sup_{t\in [0,T]} |x_t|.$$

A central object in our analysis is the space of \emph{$\alpha-$Hölder continuous paths}. For $\alpha \in (0,1]$, we define $C^{\alpha\text{-Höl}}([0,T],\RR^d)$ as the set of continuous paths $x$ with finite semi-norm 
$$\| x \|_{\alpha\text{-Höl};[0,T]} \coloneqq \sup_{0 \leq s < t \leq T} \frac{|x_t - x_s|}{|t-s|^\alpha}.$$
Unless stated otherwise, $C^{\alpha\text{-Höl}}([0,T],\RR^d)$ will be equipped with the norm $|x_0| +  \|x\|_{\alpha\text{-Höl};[0,T]}$, under which it is a Banach space. The corresponding metric is given by
$$d_{\alpha\text{-Höl};[0,T]}(x,y) \coloneqq |x_0 - y_0| + \| x - y \|_{\alpha\text{-Höl};[0,T]}.$$

Although less frequently used, we will also refer to $p-$variation spaces. For $p \in [1,\infty)$, let $C^{p\text{-var}}([0,T],\RR^d)$ denote the set of continuous paths with finite $p-$variation, that is, continuous paths $x$ satisfying
$$\| x \|_{p\text{-var};[0,T]} \coloneqq \left( \sup_{(t_i) \in \mathcal{P}([0,T])} \sum_i |x_{t_{i+1}} - x_{t_i}|^p \right)^{1/p} < \infty,$$
where $\mathcal{P}([0,T])$ is the set of partitions $(t_i) = \{0 = t_0 < t_1 < \cdots < t_{n} = T\}$ of $[0,T]$. Endowed with the norm $|x_0| + \|x\|_{p\text{-var};[0,T]}$, this space is also Banach.

To capture the evolution of a path together with its past, we consider spaces of stopped paths as in e.g., \cite{dupire2009functional,cont2013functional,fournie2010functional}. For $t \in [0,T]$ and $x \in C^0([0,T],\RR^d)$, let $x_{\cdot \land t} : [0,T] \to \RR^d$ denote the path $x$ stopped at time $t$, i.e., $x_{u \land t} = x_u$ for $u \in [0,t]$ and $x_{u \land t} = x_t$ for $u \in (t,T]$. We then define
\begin{align}
\label{eq_stopped_paths_quotient}
    \Lambda^{\alpha\text{-Höl}}([0,T],\RR^d) &\coloneqq \Big\{ (t,x_{\cdot \land t}) : t \in [0,T],\ x \in C^{\alpha\text{-Höl}}([0,T], \RR^d) \Big\} \nonumber \\ 
    &\cong [0,T] \times C^{\alpha\text{-Höl}}([0,T], \RR^d)/ \sim, 
\end{align}
where $(t,x) \sim (s,y)$ if and only if $s = t$ and $x_{\cdot \land t} = y_{\cdot \land s}$. We refer to $\Lambda^{\alpha\text{-Höl}}([0,T],\RR^d)$ as the space of \emph{stopped $\alpha-$Hölder continuous paths}, and equip it with the metric
$$d^\Lambda_{\alpha\text{-Höl};[0,T]}\big( (t, x_{\cdot \land t}), (s, y_{\cdot \land s}) \big) \coloneqq |t-s| + |x_0 - y_0| + \| x_{\cdot \land t} - y_{\cdot \land s} \|_{\alpha\text{-Höl};[0,T]}.$$ 
It is a straightforward exercise to verify that $d^\Lambda_{\alpha\text{-Höl};[0,T]}$ induces the quotient topology, thereby justifying the homeomorphism in \eqref{eq_stopped_paths_quotient}; see \cite[Lemma A.1]{bayer2023optimal} for related arguments. 

Alternatively, $\Lambda^{\alpha\text{-Höl}}([0,T],\RR^d)$ may be identified with the vector bundle
\begin{align}
    \label{eq_stopped_paths_vector_bundle}
    \bigcup_{t\in [0,T]} C^{\alpha\text{-Höl}}([0,t],\RR^d), \ \text{with} \ C^{\alpha\text{-Höl}}(\{0\},\RR^d) \cong \RR^d,
\end{align}
equipped with the metric
$$d_{\alpha\text{-Höl}}^\Lambda(x|_{[0,t]}, y|_{[0,s]}) \coloneqq |t-s| + d_{\alpha\text{-Höl};[0,t]}(x|_{[0,t]}, \tilde{y}|_{[0,t]}), \quad s \leq t,$$ 
where $\tilde{y}|_{[0,t]}$ denotes the stopped path $y_{\cdot \land s}$ restricted to $[0,t]$. The representations in \eqref{eq_stopped_paths_quotient} and \eqref{eq_stopped_paths_vector_bundle} are trivially homeomorphic, and we write $\Lambda^{\alpha\text{-Höl}}([0,T],\RR^d)$ for either. The notations $(t,x_{\cdot \land t})$ and $x|_{[0,t]}$ will be used interchangeably. When working in the vector bundle notation, we abbreviate $\| x|_{[0,t]} \|_{\subhol[t]{\alpha}}$ to $\|x|_{[0,t]}\|_{\alpha\text{-Höl}}$, since the time interval of the semi-norm is implicit in the domain of $x|_{[0,t]}$. 

Lastly, we denote the closed ball of radius $r>0$ in $C^\suphol{\alpha}([0,t],\RR^d)$, centered at $x$, by $B^\alpha_t(x,r)$. If the subscript $t$ is omitted, it is understood that $t = T$. Likewise, if $x$ is omitted, then $x=0$.

\subsection{Path Signatures and Tensor Algebras}
\label{subsec_path_signatures_tensor_algebras}

The signature of a path $x : [0,T] \to \RR^d$ is the graded collection of all iterated integrals of its components $x^j$ against one another. Since these integrals are non-commutative, tensor algebras provide the natural framework to encode their multi-index structure. We briefly recall the relevant definitions, mainly to fix notation; for full expositions we refer to standard references.

For $d\in\mathbb{N}$,  let $\{e_1,\dots,e_d\}$ denote the canonical basis of $\mathbb{R}^d$. For each $n \geq 0$ set
$$
    (\mathbb{R}^d)^{\otimes n} \coloneqq \underbrace{\mathbb{R}^d\otimes\cdots\otimes\mathbb{R}^d}_{n\text{ times}},
$$
with the convention $(\mathbb{R}^d)^{\otimes 0}\cong\mathbb{R}$. Elements of $(\mathbb{R}^d)^{\otimes n}$ are linear combinations of elementary tensors $e_{i_1}\otimes\cdots\otimes e_{i_n}$, and for $n=0$ the unit $1\in\mathbb{R}$ denotes the degree-zero basis. Each level $(\RR^d)^{\otimes n}$ is equipped with the Euclidean norm (or any admissible norm \cite[Definition 1.25]{lyons2007differential}), and $|\cdot|$ always refers to this choice.

The \emph{extended tensor algebra} over $\RR^d$ is the direct product 
\[
    T((\RR^d)) \coloneqq \prod_{n=0}^\infty (\RR^d)^{\otimes n} \equiv \Big\{ \mathbf{a} \coloneqq \big(\ba^{(0)},\ba^{(1)},\ba^{(2)},\hdots \big) : \ba^{(n)} \in (\RR^d)^{\otimes n} \Big\}.
\]
Addition and scalar multiplication are defined componentwise, and the algebra product $\otimes$ is the concatenation tensor product, extended bilinearly so that degrees add under multiplication — see \cite[Proposition 7.4]{Friz2010}. Similarly, the \emph{tensor algebra} over $\RR^d$ is the direct sum 
\[
    T(\RR^d) \coloneqq \bigoplus_{n=0}^\infty (\RR^d)^{\otimes n} \equiv \Big\{ \big(\ba^{(n)}\big)_{n\geq 0} : \ba^{(n)} \in (\RR^d)^{\otimes n} \ \text{and} \ \ba^{(n)} \neq 0  \ \text{for finitely many} \ n\in \NN \Big\},
\]
a subalgebra of $T((\RR^d))$. For $N \in \NN$, the \emph{truncated tensor algebra of level $N$} is
\[
    T^N(\RR^d) \coloneqq \bigoplus_{n=0}^N (\RR^d)^{\otimes n}.
\]
We denote by $\pi_n : T((\RR^d)) \to (\RR^d)^{\otimes n}$ the canonical projection onto level $n$, and by $\pi_{0,N}$ the truncation map to $T^N(\RR^d)$ — see \cite{beda2025introduction} for a modern treatment of tensor algebras.

Let $\mathcal{A}_d=\{1,2,\dots,d\}$ be the alphabet indexed by the basis vectors, and write $\mathcal{W}(\mathcal{A}_d)$ (or simply $\mathcal{W}$) for the free monoid of finite words over $\mathcal{A}_d$. Denote by $\mathcal{W}_n = \{w \in \cW : |w|= n\}$ the set of words of length $n$, with the empty word $\varnothing$ as the unique element of length $0$. For a word $w=i_1\cdots i_n\in\mathcal{W}_n$ we set $e_w \coloneqq e_{i_1}\otimes\cdots\otimes e_{i_n}\in(\mathbb{R}^d)^{\otimes n},
$ so that $\{e_w:w\in\mathcal{W}_n\}$ forms the canonical basis of $(\mathbb{R}^d)^{\otimes n}$.

Every element $\mathbf{a}\in T((\RR^d))$ then admits a unique word-based expansion
\[
    \mathbf{a} \,=\, \sum_{n=0}^\infty \sum_{w\in\mathcal{W}_n} a_w \, e_w,
    \quad a_w\in\mathbb{R},
\]
where $a_\varnothing = a_0$ is the degree-zero component. In this notation, the tensor product $\mathbf{a}\otimes \mathbf{b}$ corresponds to the usual Cauchy product of series \cite{chevyrev2016primer}.

There is a natural duality pairing
\[
    \langle \ \cdot \, , \cdot \
    \rangle :
     T(\mathbb{R}^d) \times T((\mathbb{R}^d)) \longrightarrow \mathbb{R}
\]
defined by coordinatewise evaluation on the word basis: if $\mathbf{a}=\sum_{w\in\mathcal{W}} a_w e_w\in T(\mathbb{R}^d)$ and $\mathbf{b}=\sum_{w\in\mathcal{W}} b_w e_w\in T((\mathbb{R}^d))$ (note $a_w \neq 0$ for only finitely many $w$), then
\[
    \langle\mathbf{a},\mathbf{b} \rangle = \sum_{w\in\mathcal{W}} a_w\, b_w. 
\]
This identifies $T(\RR^d)$ with a subspace of the algebraic dual of $T((\RR^d))$. In particular, embedding $e_w \in T(\RR^d)$ in the natural way yields $\langle e_w, \mathbf{b} \rangle = b_w$, for any word $w \in \cW$. 

Consider $\frac{1}{2} < \alpha \leq 1$ and $x \in C^{\suphol{\alpha}}([0,T],\RR^d)$. The \emph{signature of $x$} is the tensor     
$$         
    S(x)_{0,T} \coloneqq \Big(1, S^{(1)}(x)_{0,T}, S^{(2)}(x)_{0,T}, \hdots \Big) \in T((\RR^d)),
$$     
where for each $n \in \NN$,      
$$
    S^{(n)}(x)_{0,T}  \coloneqq \sum_{i_1, \hdots, i_n =1}^d \left( \int_{0 < t_1 < \cdots < t_n < T} \dd x^{i_1}_{t_1} \cdots \dd x^{i_n}_{t_n}\right) (e_{i_1} \otimes \cdots \otimes e_{i_n}) \in (\RR^d)^{\otimes n}.
$$      
For $N \in \NN$, the truncated signature of level $N$ is defined by $S^N(x)_{0,T} \coloneqq \pi_{0,N}(S(x)_{0,T})$. The map $S: C^{\suphol{\alpha}}([0,T],\RR^d) \to T((\RR^d))$, $x \mapsto S(x)_{0,T}$ is called the \emph{signature map}.

\begin{remark}     
All integrals in $S(x)_{0,T}$ are understood in the Young/Riemann-Stieltjes sense; see \cite[Chapter 6]{Friz2010} or \cite{dudley1998introduction} for background on Young integration. It is also common to write
$$
    S^{(n)}(x)_{0,T} = \int_{0 < t_1 < \cdots < t_n < T} \dd x_{t_1} \otimes \cdots \otimes \dd x_{t_n}. %
$$ 
By varying the upper integration limit, the signature can be viewed as a $T((\RR^d))$-valued path $[0,T] \ni t\mapsto S(x)_{0,t}$. Moreover, the signature map extends naturally to stopped paths: for $x_{\cdot \land t} \in \Lambda^{\suphol{\alpha}}([0,T],\RR^d)$, we have $S(x_{\cdot \land t})_{0,T} \equiv S(x)_{0,t}$. The map $x|_{[0,t]} \mapsto S(x)_{0,t}$ is also referred to as the signature map.
\end{remark}

We conclude this section by recalling basic properties of signatures and identifying natural subsets of $T((\RR^d))$ containing them.

For $\frac{1}{2} < \alpha \leq 1$ and $x \in C^{\alpha\text{-Höl}}([0,T],\RR^d)$, one has (\cite[Theorem 3.1.3]{Lyons2002}, \cite[Exercise 4.6]{friz2014course})
\begin{align}
    \label{eq_factorial_decay_signatures}
     \left| \int_{0 < t_1 < \cdots < t_n < T} \dd x_{t_1} \otimes \cdots \otimes \dd x_{t_n} \right| \leq C_{\alpha,T} \frac{\|x\|^n_{\alpha\text{-Höl};[0,T]}}{(\alpha n)!},
\end{align}
where $C_{\alpha,T}>0$ depends only on $\alpha$ and $T$, and $(\alpha n)! \coloneqq \Gamma(1 + \alpha n)$ is defined via the Gamma function. This factorial decay ensures that signatures embed naturally into \emph{$\ell^p$–sum tensor spaces}, and for this reason, for each $1\leq p<\infty$, define 
$$
    T_p(\RR^d) \coloneqq \Big\{ \mathbf{a} \in T((\RR^d)) : \|\mathbf{a}\|_p = \Big( \sum_{n\geq 0} |\ba^{(n)}|^p \Big)^{\frac{1}{p}} <\infty \Big\},
$$
and also
$$
    T_\infty(\RR^d) \coloneqq \Big\{ \mathbf{a} \in T((\RR^d)) : \|\mathbf{a}\|_\infty = \sup_{n\geq 0} |\ba^{(n)}| <\infty \Big\}. 
$$
Each $T_p(\RR^d)$ is a Banach space, with inclusions $T_1(\RR^d) \subset T_2(\RR^d) \subset \cdots \subset T_\infty(\RR^d)$, and progressively weaker topologies. By \eqref{eq_factorial_decay_signatures}, every signature lies in $T_1(\RR^d)$. Moreover, if $1<p,q<\infty$ satisfy $\frac{1}{p}+\frac{1}{q}=1$, then the duality relations
\begin{align*} 
(T_p(\RR^d))' \cong T_q(\RR^d), \quad \text{and} \quad  (T_1(\RR^d))' \cong T_\infty(\RR^d), 
\end{align*}
hold under the canonical levelwise pairing induced by the Hilbert–Schmidt inner product — see \cite[Theorem 16.49]{aliprantis2006infinite} for more details.

The signature further satisfies the \emph{shuffle property} \cite{ree1958lie}. Given two words $w=i_1\cdots i_n$ and $v=j_1\cdots j_m$, their shuffle product is the commutative, associative product on basis tensors defined recursively by 
$$
    e_w \shuffle e_\varnothing = e_\varnothing \shuffle e_w \coloneqq e_w, \quad e_w \shuffle e_v \coloneqq (e_{w'} \shuffle e_v) \otimes e_{i_n} + (e_w \shuffle e_{v'}) \otimes e_{j_m},
$$
where $w'=i_1\cdots i_{n-1}$ and $v'=j_1\cdots j_{m-1}$, and extended linearly to $T(\RR^d)$. Thus, for tensors $\mathbf{a},\mathbf{b}\in T(\RR^d)$,
$$
    \mathbf{a} \shuffle \mathbf{b} = \sum_{w,v \in \cW} a_w b_v (e_w \shuffle e_v). 
$$
Partitioning the integration domain in iterated integrals shows that signatures satisfy:
\begin{align}
\label{eq_shuffle_property}
    \langle e_w , S(x)_{0,T} \rangle \langle e_v , S(x)_{0,T} \rangle = \langle e_w \shuffle e_v , S(x)_{0,T} \rangle.
\end{align}
for all words $w,v\in\cW$ (\cite[Theorem 1.14]{chevyrev2016primer}). In particular, any polynomial in linear functionals of the signature is again a linear functional of the signature; equivalently, products of iterated integrals decompose into linear combinations of higher-order iterated integrals.

Motivated by this property, we define the set of \emph{group-like elements}:
\begin{align}
\label{eq_group_like_elements}
    G(\RR^d) \coloneqq \Big\{ \mathbf{g} \in T((\RR^d)) \setminus \{\mathbf{0}\} : \langle \mathbf{b}_1 , \mathbf{g} \rangle \langle \mathbf{b}_2 , \mathbf{g} \rangle = \langle \mathbf{b}_1 \shuffle \mathbf{b}_2 , \mathbf{g} \rangle, \ \text{for all} \ \mathbf{b}_1,\mathbf{b}_2 \in T(\RR^d)\Big\}.
\end{align}
That is, $G(\RR^d)$ consists of all nonzero tensors on which the shuffle property holds. Several algebraically equivalent characterizations of $G(\RR^d)$ can be found in \cite[Theorem 3.2]{reutenauer1993free}. From \eqref{eq_shuffle_property} it follows that every signature belongs to $G(\RR^d)$, and indeed one can show that this inclusion is strict \cite{cass2024topologies}. Equipped with the tensor product $\otimes$, $G(\RR^d)$ forms a group with identity $\mathbf{1}$. Projecting to truncated tensors yields a family of finite-dimensional groups:
$$
    G^N(\RR^d) \coloneqq \pi_{0,N}(G(\RR^d)),
$$
the well-known \emph{free step-$N$ nilpotent Lie groups} over $\RR^d$. With the truncated product $\otimes$ on $T^N(\RR^d)$, each $G^N(\RR^d)$ is a group, which coincides with the exponential of Lie polynomials \cite{chen1957integration}. More precisely, we have that $G^N(\RR^d) = \exp_{\otimes N}\big(\mathfrak{g}^N(\RR^d)\big)$ \cite[Theorem 7.30]{Friz2010}, where $\exp_{\otimes_N}(\, \cdot \,)$ denotes the (truncated) tensor exponential map, and 
\[\mathfrak{g}^N(\RR^d) \coloneqq \RR^d \oplus [\RR^d,\RR^d] \oplus \cdots \oplus \underbrace{[\RR^d,[\hdots, [\RR^d,\RR^d]\hdots]]}_{(N-1) \ \text{brackets}}\]
corresponds to the \emph{free step$-N$ nilpotent Lie algebra} with Lie bracket given by the commutator.

We recall that $\mathfrak{g}^N(\RR^d)$ is a sub-Lie algebra of
\[
    \mathfrak{t}^N(\RR^d) \coloneqq \Big\{\mathbf{a} \in T^N(\RR^d) : \pi_0(\mathbf{a}) = 0 \Big\},
\]
while $G^N(\RR^d)$ is a closed sub-Lie group of 
\[
    \mathbf{1} + \mathfrak{t}^N(\RR^d) \coloneqq \Big\{ \mathbf{a} \in T^N(\RR^d) : \pi_0(\mathbf{a}) = 1 \Big\}.
\]
We refer to \cite[Section 7.3]{Friz2010} for details. Importantly, just like with $G^N(\RR^d)$ and $\mathfrak{g}^N(\RR^d)$, the exponential map $\exp_{\otimes N} : \mathfrak{t}^N(\RR^d) \to \mathbf{1} + \mathfrak{t}^N(\RR^d)$ is a global diffeomorphism. 

Finally, both $G^N(\mathbb{R}^d)$ and $\mathbf{1}+\mathfrak{t}^N(\mathbb{R}^d)$ are examples of \emph{homogeneous groups}, i.e. connected, simply connected nilpotent Lie groups whose Lie algebras admit a family of \emph{dilations} $\{\delta_\lambda\}_{\lambda>0}$; see \cite[Chapter 1]{gb28hardy} for the general definition. In these two cases, the dilations are given explicitly on $T^N(\mathbb{R}^d)$ by $\pi_k(\delta_\lambda \mathbf{a}) = \lambda^k \pi_k(\mathbf{a})$, which restricts to dilations on both $G^N(\mathbb{R}^d)$ and $\mathbf{1}+\mathfrak{t}^N(\mathbb{R}^d)$. Additionally, we may equip any homogeneous group $\mathbb{G}$ with a \emph{homogeneous gauge}, that is a continuous function $\| \cdot \|_h : \mathbb{G} \to [0,\infty)$ such that $\| \mathbf{g}\|_h = 0$ if and only if $\mathbf{g} = \mathbf{1}$, $\| \mathbf{g}^{-1}\|_h = \| \mathbf{g}\|_h$, and $\| \delta_\lambda \mathbf{g}\|_h = \lambda \| \mathbf{g}\|_h$ — see \cite[Chapter 1]{gb28hardy} or \cite[Definition 7.34]{Friz2010}.

\subsection{Weighted Spaces and Global Approximations}
\label{subsec_weighted_spaces_global_approximations}

We now introduce the key elements of the theory of weighted spaces, which play a central role in our analysis of Sig-CDEs and their approximation properties. In particular, we establish a new result (Lemma \ref{lem_characterising_B_psi})  
characterizing which Banach-valued functions $f : X \to Y$, defined on a non-locally compact space $X$, admit \emph{global approximation} — namely, approximation over the entire domain $X$ in a weighted sense made precise below.

Let $(X, \tau_X)$ be a completely regular Hausdorff space.\footnote{E.g. every metric space.} A map $\psi: X \to (0,\infty)$ is said to be an \emph{admissible weight function} if, for every $R > 0$, the pre-image 
\begin{equation}\label{eq:K_R}
   K_R \coloneqq \psi^{-1}((0,R]) = \{x \in X : \psi(x) \leq R\} 
\end{equation}
is compact in the topology $\tau_X$. The pair $(X, \psi)$ is then called a \emph{weighted space}.

Following \cite{cuchiero2023global}, we focus on weighted spaces $(X, \psi)$ where $X$ is a normed space with norm $\|\cdot\|_X$, and the weight function $\psi: X \to (0,\infty)$ takes the form $\psi(x) = \eta(\|x\|_X)$, with $\eta: [0,\infty) \to (0,\infty)$ continuous, increasing, and unbounded. By Riesz’s lemma (see \cite[Lemma 1.39]{van2022functional}), the closed unit ball of a normed space is compact if and only if the space is finite-dimensional. Consequently, for $\psi(x) = \eta(\|x\|_X)$ to be admissible in infinite-dimensional settings, it is typically necessary to endow $X$ with a weaker topology than the norm topology. We now illustrate this point by equipping the path spaces introduced earlier with the structure of weighted spaces — see also \cite[Example 2.3]{cuchiero2023global}. 

\begin{example}
    Fix $\alpha \in (0,1]$ and set $X = C^{\alpha\text{-Höl}}([0,T], \RR^m)$. For any $\beta < \alpha$, equip $X$ with the $\beta$-Hölder semi-norm $\|\cdot\|_{\subhol{\beta}}$, i.e., consider $C^{\alpha\text{-Höl}}([0,T], \RR^m)$ endowed with the $\beta$-Hölder topology. Let $\psi(x) \coloneqq \eta(|x_0| + \|x\|_{\subhol{\alpha}})$, with $\eta: [0,\infty) \to (0,\infty)$ as specified above. For each $R > 0$, the pre-image $K_R$ is bounded in the $\alpha$-Hölder norm. By the compact embedding $C^{\alpha\text{-Höl}}([0,T], \RR^m) \hookrightarrow C^{\beta\text{-Höl}}([0,T], \RR^m)$ of Hölder spaces (see \cite[Theorem A.4]{cuchiero2023global}), each $K_R$ is compact in the $\beta$-Hölder topology. Hence $X$, endowed with the $\beta-$Hölder topology and weight function $\psi$, is a weighted space.
\end{example}

\begin{remark}
    Recall that an embedding $i:X\to Y$ between Banach spaces is \emph{compact} if bounded subsets of $X$ are mapped to relatively compact subsets of $Y$. For Hölder spaces, closed balls in $C^{\alpha\text{-Höl}}([0,T],\mathbb{R}^m)$ are compact with respect to the $\beta-$Hölder topology whenever $\beta<\alpha$; see \cite[Proposition 5.28]{Friz2010} or \cite[Theorem A.4]{cuchiero2023global}. We will repeatedly exploit this compact embedding by endowing Hölder path spaces with a slightly weaker topology. To avoid confusion with the Banach space $C^\suphol{\alpha}([0,T],\RR^m)$, we write $C^\suphol{\alpha}_\beta([0,T],\RR^m)$, or simply $C^{\alpha\text{-Höl}}_\beta$,  whenever $C^\suphol{\alpha}([0,T],\RR^m)$ is endowed with a weaker $\beta-$Hölder topology.
\end{remark}

\begin{example}
\label{ex_weighted_stopped_paths}
    Let $\alpha \in (0,1]$ and consider $X = \Lambda^\suphol{\alpha}([0,T],\RR^m)$
    equipped with the metric 
    $$d_{\beta\text{-Höl};[0,T]}^\Lambda\big((t,x_{\cdot \land t}),(s,y_{\cdot \land s})\big) \coloneqq |t-s| + |x_0 - y_0| + \|x_{\cdot \land t} - y_{\cdot \land s}\|_{\beta\text{-Höl};[0,T]},$$
    for some $\beta<\alpha$. Let $\eta: [0,\infty) \to (0,\infty)$ be as before, and define the weight function $\psi\big( (t,x_{\cdot \land t}) \big) \coloneqq \eta(|x_0| + \|x_{\cdot \land t}\|_{\subhol{\alpha}})$. As in the previous example, $\psi$ is admissible due to the compact embedding of Hölder spaces. More precisely, for each $R > 0$,
    $$A_R \coloneqq \left\{ x_{\cdot \land t} \in  C^{\alpha\text{-Höl}}([0,T], \RR^m) : (t,x_{\cdot \land t}) \in \psi^{-1}((0,R])\right\}$$
    is compact in the $\beta$-Hölder topology. Indeed, by \cite[Theorem A.4]{cuchiero2023global} the set $A_R$ is relatively compact in $C^\suphol{\beta}([0,T],\RR^m)$ and closedness (in $\beta-$Hölder topology) follows by \cite[Proposition 5.5 and Lemma 5.12]{Friz2010}. It then follows by Tychonoff’s theorem that $[0,T] \times A_R$ is compact in the product space $[0,T] \times C^{\alpha\text{-Höl}}_\beta([0,T], \RR^m)$. Finally, note that the quotient map 
    $$q: [0,T] \times C^{\alpha\text{-Höl}}_\beta([0,T], \RR^m) \to \Lambda^{\alpha\text{-Höl}}([0,T],\RR^m)$$ is continuous with respect to $d_{\beta\text{-Höl};[0,T]}^\Lambda$. Hence, $q([0,T]\times A_R) \equiv K_R$ is compact, showing that $(X, \psi)$ is a weighted space. 
\end{example}

Analogously to the Hölder case, we write $\Lambda^\suphol{\alpha}_\beta([0,T],\RR^m)$ (or simply $\Lambda^{\alpha\text{-Höl}}_\beta$) for the space of stopped $\alpha-$Hölder paths $\Lambda^\suphol{\alpha}([0,T],\RR^m)$ endowed with the weaker topology induced by the metric $d^\Lambda_{\subhol{\beta}}$. When the weighted space structure is required, we additionally specify the corresponding weight function $\psi$.

Given a weighted space $(X, \psi)$, we now introduce a function space tailored to the task of \emph{global approximation}. Unlike the classical setting, where the target space consists of continuous functions over a (locally) compact domain, our aim is to approximate functions defined over the entire (potentially non-locally compact) space $X$. To this end, let $(Y,\|\cdot\|_Y)$ denote a Banach space and define the set
$$B_\psi(X,Y) \coloneqq \left\{ f: X \to Y : \|f\|_{\cB_\psi(X,Y)} \coloneqq \sup_{x \in X} \frac{\|f(x)\|_Y}{\psi(x)} < \infty \right\},$$
consisting of $Y$-valued functions on $X$ whose growth is controlled by the weight $\psi$. Note that the space of bounded continuous functions $C_b^0(X,Y)$ embeds continuously into $B_\psi(X,Y)$. 

We then define the $\| \cdot \|_{\cB_\psi(X,Y)}-$closure of $C_b^0(X,Y)$ in $B_\psi(X,Y)$ to be the \emph{weighted function space} $\cB_\psi(X,Y)$. Equipped with the norm $\| \cdot \|_{\cB_\psi(X,Y)}$, the space $\cB_\psi(X,Y)$ is a Banach space. When $X$ and $Y$ are clear from the context, we often abbreviate $\cB_{\psi}(X,Y)$ to $\cB_{\psi}$.

As mentioned before, $\cB_\psi(X,Y)$ consists of functions whose growth is tempered by the weight function $\psi$, and may include functions that are unbounded on $X$. A precise characterization of $\cB_{\psi}(X) \coloneqq \cB_\psi(X,\RR)$ is given in \cite[Theorem 2.7]{dorsek2010semigroup}: a function $f \in \cB_\psi(X)$ if and only if $f|_{K_R} \in C^0(K_R, \RR)$ for all $R > 0$ and
$$\lim_{R\to \infty} \sup_{x \in X \setminus K_R} \frac{|f(x)|}{\psi(x)} = 0.$$

Complementing \cite[Lemma 2.7(i)]{cuchiero2023global}, the following lemma shows that this characterization remains valid when the codomain is a general Banach space \(Y\), provided that \(\psi\) is assumed to be locally bounded in the `if' direction.

\begin{lemma}
\label{lem_characterising_B_psi}
    Let $X$ be a weighted space and $Y$ be a Banach space. If $f : X \to Y$ is continuous over $X$ and 
    \begin{align}
        \label{eq_vanishing_tail_condition}
        \lim_{R \to \infty} \sup_{x\in X\setminus K_R} \frac{\|f(x)\|_Y}{\psi(x)} = 0,
    \end{align}
    then $f\in \cB_\psi(X,Y)$. Additionally, if $\psi$ is locally bounded, then the converse implication also holds.  

\end{lemma}
\begin{proof}
    See Appendix \ref{Appendix_2}. 
\end{proof}

To establish global approximation results, it remains to extend the classical Stone–Weierstrass theorem to the weighted setting. This extension is provided in \cite[Theorem 3.9]{cuchiero2023global}, where the authors formulate and prove a weighted version of the theorem with $\cB_\psi(X)$ as the target space for (global) approximation. We refer to the original work for the precise statement and proof. 

We conclude this section by applying this weighted Stone–Weierstrass theorem in the context of signatures: taking $X = \Lambda^{\alpha\text{-Höl}}_\beta([0,T],\RR^m)$ as a weighted space, one can show that linear functionals of the signature — computed on stopped $\alpha-$Hölder continuous paths — approximate any function in $\cB_\psi(X)$ arbitrarily well. We henceforth refer to maps over $\Lambda^{\alpha\text{-Höl}}_\beta([0,T],\RR^m)$ as \emph{non-anticipative path functionals}.

\begin{remark}
\label{rem_time_augmentation}
    As commonly done, to ensure that the algebra of signature linear functionals is point-separating, we work with \emph{time-augmented} paths \cite[Lemma 2.6]{cuchiero2023signatureSIAM}. Define
    \[
    \widehat{C}^\suphol{\alpha}([0,T],\RR^{m+1}) \coloneqq \left\{ \hat{x} \in C^\suphol{\alpha}([0,T],\RR^{m+1}) : \hat{x}_t^0 = t \ \text{for all} \ t\in [0,T]\right\},
    \]
    the space of \emph{time-augmented $\alpha-$Hölder continuous paths}, where the time coordinate is referred to as the $0$–th coordinate. We further set
    \[
    \widehat{\Lambda}^\suphol{\alpha}([0,T],\RR^{m+1}) = \bigcup_{t\in [0,T]} \widehat{C}^\suphol{\alpha}([0,t],\RR^{m+1}),
    \]
    the space of \emph{stopped time-augmented $\alpha-$Hölder continuous paths} (with the time coordinate also stopped). As before, $\widehat{C}^{\alpha\text{-Höl}}_\beta$ and $\widehat{\Lambda}^{\alpha\text{-Höl}}_\beta$ denote these spaces equipped with the $\beta$–topology.
\end{remark}

\begin{theorem}
\label{prop_global_approx_signatures}
    Fix $\frac{1}{2} < \alpha \leq 1$ and $\beta < \alpha$. Consider $\big(  \widehat{\Lambda}^{\alpha\text{-Höl}}_\beta([0,T],\RR^{m+1}), \hat{\psi} \big)$ with 
    \[\hat{\psi}(\hat{y}|_{[0,t]}) = \exp\Big( \zeta \big( |\hat{y}_0| + \|\hat{y}|_{[0,t]}\|_{\alpha\text{-Höl}}\big)^\xi \Big),\]
    for some $\zeta > 0$ and $\xi \geq 1$. Then, the linear span of the set
    \[
    \Big\{ \widehat{\Lambda}^{\alpha\text{-Höl}}_\beta([0,T],\RR^{m+1}) \ni \hat{y}|_{[0,t]} \mapsto \langle e_w , S(\hat{y})_{0,t} \rangle : w \in \cW \Big\}
    \]
    is dense in $\cB_{\hat{\psi}}(\widehat{\Lambda}^{\alpha\text{-Höl}}_\beta)$. More precisely, for every $f \in \cB_{\hat{\psi}}(\widehat{\Lambda}^{\alpha\text{-Höl}}_\beta)$, $w_0 \in \mathbb{R}^m$ and every $\varepsilon > 0$, there exists a linear functional of the form $\hat{y}|_{[0,t]} \mapsto \langle \bell , S(\hat{y})_{0,t} \rangle \equiv \sum_{0 \leq |w| \leq N} \ell_w \langle e_w, S(\hat{y})_{0,t} \rangle$, with $N \in \NN_0$ and $\ell_w \in \RR$, such that
    \[\sup_{\substack{\hat{y}|_{[0,t]} \in \widehat{\Lambda}^{\alpha\text{-Höl}}_\beta\\y_0=w_0}} \frac{\big| f(\hat{y}|_{[0,t]}) - \langle \bell, S(\hat{y})_{0,t} \rangle \big|}{\hat{\psi}(\hat{y}|_{[0,t]})} < \varepsilon.\]
\end{theorem}
\begin{proof}
    The argument is identical to that of the original result \cite[Theorem 5.4]{cuchiero2023global}, and we refer the reader to that work for full details; see also \cite[Theorem 2.20]{cuchiero2025signature}.
\end{proof}

From Theorem \ref{prop_global_approx_signatures} we recover the classical universality  result (see \cite[Lemma 3.4]{kalsi2020optimal} and \cite[Theorem 2.12]{cuchiero2025signature}) for signatures as an immediate corollary.

\begin{corollary}
\label{cor_classical_universality}
    Let $K \subset \widehat{\Lambda}^{\alpha\text{-Höl}}_\beta([0,T_1],\RR^{m+1})$ be compact. Then for any continuous, non-anticipative functional $f: K \to \RR$, $w_0 \in \mathbb{R}^m$ and every $\varepsilon > 0$, there exists a signature linear functional $\bell$ such that 
    \[\sup_{\substack{\hat{y}|_{[0,t]} \in K\\y_0=w_0}} \big| f(\hat{y}|_{[0,t]}) - \langle \bell, S(\hat{y})_{0,t} \rangle \big| < \varepsilon.\]
\end{corollary}

\subsection{Limits of Tensor Spaces}
\label{subsec_limits_of_tensor_spaces}

The path functional in Sig-CDEs \eqref{eq:signature_cde_intro} factors through a subspace of the extended tensor algebra containing the set of signatures. Since several such ambient spaces can be considered, it becomes natural to ask which choice is most appropriate. As will be explained in Section \ref{subsec_well_posedness_sig_cdes}, the strength of the topology on the space containing the signatures directly affects the analytical assumptions required for well-posedness of Sig-CDEs: the stronger the topology, the weaker the conditions that need to be imposed on the vector field. This observation motivates the constructions that follow.

Fix $\lambda > 0$. The \emph{weighted $\ell^1$-sum tensor space}
\[
    T_{1,\lambda}(\RR^m) \coloneqq \Big\{ \mathbf{a} \in T((\RR^m)) : \|\mathbf{a}\|_{1,\lambda} \coloneqq \sum_{n\geq 0} \lambda^n |\mathbf{a}^{(n)}| < \infty \Big\}.
\]
contains every signature, and $\|\mathbf{a}\|_{1,\lambda} = \| \delta_\lambda \mathbf{a}\|_1$. If $\lambda \geq 1$, then $\|\cdot\|_1 \leq \|\cdot\|_{1,\lambda}$, so $\|\cdot\|_{1,\lambda}$ induces a stronger topology than $\|\cdot\|_1$ over the set of signatures. Moreover, whenever $\lambda \leq \lambda'$, we get $\|\cdot\|_{1,\lambda} \leq \|\cdot\|_{1,\lambda'}$, meaning we obtain infinitely many ambient spaces with increasingly stronger topologies. 

To obtain an even stronger topology, let $(\lambda_k)_{k\geq 0}$ satisfy $\lambda_k >0$ and $\lambda_k \to \infty$, and consider the projective limit
\[
    \varprojlim T_{1,\lambda_k}(\RR^m) \coloneqq \bigcap_{k\geq 0} T_{1,\lambda_k}(\RR^m),
\]
equipped with the locally convex topology generated by the countable family of norms $\mathbf{a} \mapsto \| \iota_k (\mathbf{a})\|_{1,\lambda_k}$ (here $\iota_k$ denotes the canonical inclusion $\varprojlim T_{1,\lambda_k}(\RR^m) \hookrightarrow T_{1,\lambda_k}(\RR^m)$). This construction yields a Fréchet space, i.e., a completely metrizable locally convex topological vector space, and one explicit metric inducing the topology is
\begin{align}
    \label{eq_projective_limit_metric}
    \rho_p(\mathbf{a},\mathbf{b}) \coloneqq \sum_{k\geq 0} \exp(-k^p) \frac{\|\mathbf{a} - \mathbf{b}\|_{1,\lambda_k}}{1 + \|\mathbf{a} - \mathbf{b}\|_{1,\lambda_k}},
\end{align}
for any $p>1$. By construction the projective limit topology is finer than the topology induced by any fixed $\| \cdot \|_{1,\lambda_k}$, as it coincides with the initial topology for the inclusion maps $\iota_k$. More importantly, we have the following result:

\begin{proposition}
\label{lem_projective_limit_topo}
    The locally convex topology generated by the norms $\big\{\|\iota_k(\cdot)\|_{1,\lambda_k} : k \geq 0\big\}$ is the unique, and therefore finest, topology under which $\varprojlim T_{1,\lambda_k}(\RR^m)$ becomes a Fréchet space and each inclusion $\iota_k$ is continuous.
\end{proposition}
\begin{proof}
    See Appendix \ref{Appendix_2}.
\end{proof}
Lastly, although $\varprojlim T_{1,\lambda_k}(\RR^m)$ is not normable, we may restrict our attention to the spaces
\begin{align}\label{eq:limitTp}
\mathcal{T}^{(p)}(\RR^m) \coloneqq \Big\{ \mathbf{a} \in \varprojlim T_{1,\lambda_k}(\RR^m) : \| \mathbf{a} \|_{(p)} \coloneqq \frac{1}{Z_p} \sum_{k\geq 0} \exp(-k^p) \| \mathbf{a} \|_{1,\lambda_k} < \infty \Big\},
\end{align}
for $p>1$ and $Z_p \coloneqq \sum_{k\geq 0} \exp(-k^p)$. Standard arguments show that each $\mathcal{T}^{(p)}(\RR^m)$ is a Banach space, and we now prove that, when used as codomain, the signature map is locally Lipschitz provided $p$ and $(\lambda_k)$ are chosen in accordance with the Hölder regularity on path space. 

\begin{proposition}
\label{lem_continuity_sig_map}
    Fix $\frac{1}{2}< \beta \leq \alpha \leq 1$. Consider $p > 1$ and weights $(\lambda_k)$ such that $\lambda^{1/\beta}_k/k^p \to 0$ as $k \to \infty$. Then, the signature map $S : C^{\beta\text{-Höl}}([0,t],\RR^m) \to \mathcal{T}^{(p)}(\RR^m)$ given by $y|_{[0,t]} \mapsto S(y)_{0,t}$ is locally Lipschitz. Moreover, for every compact set $K \subset C^{\alpha\text{-Höl}}_\beta([0,T],\RR^m)$ there exists $L_K>0$ such that, for $x,y \in K$
    \begin{align*}
        \| S(x)_{0,t} - S(y)_{0,t} \|_{(p)} \leq L_K \, \| x_{\cdot \land t} - y_{\cdot \land t} \|_{\beta\text{-Höl};[0,T]}, \quad t \in [0,T].
    \end{align*}
\end{proposition}

\begin{proof}
    The result is grounded in the continuity of Lyons’ extension \cite[Theorem 3.1.3]{Lyons2002} (see also \cite[Theorem 3.10]{lyons2007differential}). Fix $z \in C^{\beta\text{-Höl}}([0,T],\RR^m)$ with $|z_0| + \|z\|_{\beta\text{-Höl};[0,T]} < D_z$ for some $D_z>0$, and let $x,y \in C^{\beta\text{-Höl}}([0,T],\RR^m)$ satisfy 
    \[
        d_{\beta\text{-Höl};[0,T]}(x, z), d_{\beta\text{-Höl};[0,T]}(y, z) < D,
    \]
    for some $D > 0$. Set $R \coloneqq D + D_z$, and define for some $\gamma > 0$ the control function (see \cite[Definition 1.6]{Friz2010})
    \begin{align}
    \label{eq_control_lyons_extension}
         \big\{(u,v) : 0\leq u \leq v \leq T \big\} \ni (u,v) \mapsto \omega(u,v) = \big( (R \lor 2D) \gamma \beta ! \big)^\frac{1}{\beta} |u-v|,
    \end{align}
    where $\beta! = \Gamma(1 + \beta)$ is defined via the Gamma function. For fixed $t\in[0,T]$ and all $0 \leq u < v \leq T$, we have
    \[
        |(x_{\cdot \land t})_{u,v}|, |(y_{\cdot \land t})_{u,v}| \leq \frac{\omega(u,v)^\beta}{\gamma (\beta)!},
    \]
    and moreover,
    \[
        |(x_{\cdot \land t})_{u,v} - (y_{\cdot \land t})_{u,v}| \leq \| x_{\cdot \land t} - y_{\cdot \land t} \|_{\beta\text{-Höl};[0,T]}\, \frac{\omega(u,v)^\beta}{\gamma (\beta)! (R \lor 2D)}.
    \]
    Up to rescaling $\gamma$, these are precisely the assumptions of \cite[Theorem 3.1.3]{Lyons2002}. Hence, for any weight $\lambda_k > 0$, we obtain
    \begin{align}
    \label{eq_aux_lip_continuity_sig}
        \sum_{n=1}^\infty \lambda_k^n |S^{(n)}(x)_{0,t} - S^{(n)}(y)_{0,t}| 
        \leq \frac{\gamma^{-1}}{R \lor 2D} \sum_{n=1}^\infty \lambda_k^n \frac{\omega(0,T)^{n \beta}}{\left( n \beta \right)!}  \| x_{\cdot \land t} - y_{\cdot \land t} \|_{\beta\text{-Höl};[0,T]}.
    \end{align}
    
    The series on the right converges by the ratio test (alternatively, by knowing the Mittag–Leffler function). Thus, there exists $L_z>0$, independent of $t$, such that
    \[\|S(x)_{0,t} - S(y)_{0,t}\|_{1,\lambda_k} \leq L_z \| x_{\cdot \land t} - y_{\cdot \land t} \|_{\beta\text{-Höl};[0,T]},\] 
    for all $t\in [0,T]$ and $x,y \in B^\beta(z,D)$. To see that the same holds in $\cT^{(p)}(\RR^m)$, define
    \[
        E_\beta(u) \coloneqq \sum_{n=0}^\infty \frac{u^n}{(n\beta)!}.
    \]
    We recognize $E_\beta(u)$ as the one-parameter Mittag-Leffler function, which satisfies $|E_\beta(u)| \lesssim \exp( 2 u^{1/\beta})$ up to a constant depending only on $\beta$ \cite[Proposition 3.10]{gorenflo2020mittag}. Therefore, from \eqref{eq_aux_lip_continuity_sig} we conclude that
    \begin{align}
    \label{eq_sig_projective_continuity}
        \sum_{k \geq 0} \frac{1}{e^{k^p}} \|S(x)_{0,t} - S(y)_{0,t}\|_{1,\lambda_k} &\leq \frac{\gamma^{-1}}{R \lor 2D} \sum_{k \geq 0} \frac{E_\beta\big(\lambda_k \, \omega(0,T)^\beta \big)}{e^{k^p}} \, \| x_{\cdot \land t} - y_{\cdot \land t} \|_{\beta\text{-Höl};[0,T]} \\
        &\leq \frac{\gamma^{-1}}{R \lor 2D} \sum_{k \geq 0} \exp\Big(2 \lambda_k^{1/\beta} \omega(0,T) - k^p \Big) \, \| x_{\cdot \land t} - y_{\cdot \land t} \|_{\beta\text{-Höl};[0,T]}. \nonumber
    \end{align}
    Since $\lambda_k^{1/\beta}/k^p \to 0$, the final summation converges, implying that 
    \[
        \| S(x)_{0,t}  - S(y)_{0,t} \|_{(p)} \leq L_z \, \| x_{\cdot \land t} - y_{\cdot \land t} \|_{\beta\text{-Höl};[0,T]},
    \]
    again for some constant $L_z>0$ dependent on $z$ and independent of $t$. 
    
    Finally, Lipschitz continuity on compacts follows by repeating the previous argument over a sufficiently large ball centered at $0$. Indeed, let $K \subset C^{\beta\text{-Höl}}([0,T],\RR^m)$ be compact and set
    \[
    M_K \coloneqq \sup_{x \in K} \big( |x_0| + \|x\|_{\beta\text{-Höl};[0,T]} \big) < \infty.
    \]
    Choosing $D>M_K$, we have $K \subset B^\beta(D)$. Repeating the previous argument with $z=0$, estimate \eqref{eq_sig_projective_continuity} directly gives a constant $L_K>0$, independent of $t$, from which the claim follows. Since $C^{\alpha\text{-Höl}}([0,T],\RR^m) \subset C^{\beta\text{-Höl}}([0,T],\RR^m)$, the same conclusion holds in $C^{\alpha\text{-Höl}}_{\beta}([0,T],\RR^m)$.
\end{proof}

\begin{remark}
    This result holds in a much more general setting. In particular, for any $q \geq 1$, the proof above applies to the space of weakly geometric $q$-$\omega$ rough paths $C^{q-\omega}([0,T], G^{[q]}(\RR^m))$, equipped with the inhomogeneous distance $\rho_{q-\omega;[0,T]}(\bx,\by)$ — see \cite[Definitions 8.6, 9.15]{Friz2010}. In this case, we consider 
    \begin{align*}
    \tilde{\omega}(s,t) = \left( 1 \lor (R \lor 2D) \max_{i=1,\hdots,[q]} \gamma \Big(\frac{i}{q} \Big)! \right)^q \omega(s,t)
    \end{align*}
    instead of \eqref{eq_control_lyons_extension}, and set $1/\beta = q$; the remainder of the proof carries through unchanged. We also note that a direct application of \cite[Proposition 8.7]{Friz2010} yields local Lipschitz continuity with respect to the $q$-variation distance $\rho_{\subvar[T]{q}}(\bx,\by)$, as defined in \cite[Definition 8.6]{Friz2010}.
\end{remark}

The following corollary is an immediate consequence of \eqref{eq_sig_projective_continuity}. 

\begin{corollary}
\label{cor_signature_loc_lip}
    The signature map $S : C^{\beta\text{-Höl}}([0,T],\RR^m) \to \varprojlim T_{1,\lambda_k}(\RR^m)$ is locally Lipschitz with respect to $\| \cdot \|_{1,\lambda_k}$ for every $\lambda_k > 0$. The same holds true for all $\| \cdot \|_p$ with $p\geq 1$.
\end{corollary}
\begin{proof}
    Since $\| \cdot \|_{1,\lambda_k} \leq Z_p \exp(k^p)\| \cdot \|_{(p)}$, the first claim follows directly from \eqref{eq_sig_projective_continuity}. For the second claim, note that choosing $\lambda_k = 1$ for all $k \geq 0$ recovers the $\ell_1-$norm $\|\cdot\|_1$, and that $\| \cdot \|_p \leq \| \cdot \|_1$ for all $p\geq 1$.
\end{proof}

\begin{remark}
    Our projective limit $\varprojlim T_{1,\lambda_k}(\mathbb{R}^m)$ provides a concrete instance of the locally $\mathfrak{m}-$convex algebra $E$ in \cite[Definition 2.1]{chevyrev2016characteristic}, as ensured by \cite[Corollary 2.5]{chevyrev2016characteristic}. As for the Banach spaces $\mathcal{T}^{(p)}(\RR^m) \subset \varprojlim T_{1,\lambda_k}(\mathbb{R}^m)$, they appear not to have been explicitly identified in the existing literature.
\end{remark}

In particular, the spaces $\cT^{(p)}(\RR^m)$ can be viewed as linear ambient spaces for signatures whose norms enforce a levelwise decay comparable to that of signature tensors. For instance, when $\lambda_k = k$, letting $p \downarrow 1$ makes the spaces $\cT^{(p)}(\RR^m)$ progressively smaller, and membership imposes a decay of the form $|\mathbf{a}^{(n)}| \lesssim (n!)^{-1/p}$. In this way, the induced topology on signatures becomes increasingly strong, as the ambient spaces approach the set of signatures itself. While these spaces and their properties are interesting in their own right, a systematic study lies beyond the scope of the present work. For our purposes it suffices to record the following aspects needed later.

\begin{lemma}
\label{lem_compact_emb_limit_spaces}
    Fix $1 < p < q$. Then, for every $R>0$
    \begin{align}
        \label{eq_tail_estimate}
        \sup_{\| \mathbf{a} \|_{(p)} \leq R} \| \mathbf{a} - \pi_{0,N} \mathbf{a} \|_{(q)} \longrightarrow 0, \qquad N \to \infty.
    \end{align}
    Consequently, the embedding $\cT^{(p)}(\RR^m) \hookrightarrow \cT^{(q)}(\RR^m)$ is compact.
\end{lemma}
\begin{proof}
    See Appendix \ref{Appendix_2}.
\end{proof}

\begin{lemma}
\label{lem_continuity_tensor_product}
    Fix $1 < p < q$. Then, the tensor product $\otimes : \cT^{(p)}(\RR^m) \times \cT^{(p)}(\RR^m) \to \cT^{(q)}(\RR^m)$ is a continuous bilinear operator.
\end{lemma}
\begin{proof}
    See Appendix \ref{Appendix_2}.
\end{proof}

To conclude, we note that the construction above can be pushed beyond the spaces $\cT^{(p)}(\RR^m)$ by iterating the procedure, which leads to the second projective limit
\[
\varprojlim \mathcal{T}^{(1+1/k)}(\mathbb{R}^m) = \bigcap_{k\geq 1} \mathcal{T}^{(1 + 1/k)}(\mathbb{R}^m),
\]
equipped, as before, with the initial topology induced by the inclusion maps. For convenience, we write
\begin{align}
    \label{eq_second_projective_limit}
    \mathbf{T}(\mathbb{R}^m) \coloneqq \varprojlim \mathcal{T}^{(1+1/k)}(\mathbb{R}^m).
\end{align}
These limit spaces will only be used in Section \ref{subsec_scaling_argument}, and a more detailed discussion of their properties and relation to the signature literature is deferred to future work.

\section{Path-Dependent Controlled Differential Equations}
\label{sec_path_dependent_cdes}

This section develops a well-posedness theory for path-dependent CDEs. While the literature on functional differential equations is broad and encompasses various levels of generality, we present a formulation that is both sufficiently robust for wider applicability and tailored to the needs of this work. In particular, the results established here will serve as a foundation for the analysis of Sig-CDEs in Section \ref{sec_sig_cdes}.

The techniques employed are conceptually classical, tracing back to the theory of ordinary differential equations. Nonetheless, our path-dependent setting on Hölder spaces introduces enough technical subtleties that justify a concise yet comprehensive treatment of the foundational aspects — namely, global existence, uniqueness, and continuity of solutions with respect to initial data.

Fix $T_1 > 0$ and $T_0 \in [0,T_1)$. Throughout, unless explicitly stated otherwise, we take $\alpha \in \left(\tfrac{1}{2},1\right)$ and $\beta < \alpha$. Let $x : [0,T_1] \to \RR^d$ be a Lipschitz path, and let $f : \Lambda^{\alpha\text{-Höl}}_\beta([0,T_1],\RR^m) \to L(\RR^d,\RR^m)$ be a non-anticipative functional. Given an initial stopped path $(T_0, w_{\cdot \land T_0}) \in \Lambda^{\alpha\text{-Höl}}_\beta([0,T_1],\RR^m)$, we consider the following path-dependent CDE:
\begin{align}
\label{eq_functional_cde}
    \begin{cases}
    y_t = w_t, \quad &t \in [0,T_0]\\
    y_t = w_{T_0} + \displaystyle\int_{T_0}^t f(s, y_{\cdot \land s}) \dd x_s \quad &t \in [T_0,T_1]
    \end{cases},
\end{align}
where the integral is understood in the (vector-valued) Riemann–Stieltjes sense. More precisely, writing $f^j(t,y_{\cdot \land t})$ for the $j-$th column of $f(t,y_{\cdot \land t})$, we set
\[\int_{T_0}^t f(s, y_{\cdot \land s}) \dd x_s \coloneqq \sum_{j=1}^d \int_{T_0}^t f^j(s, y_{\cdot \land s}) \dd x^j_s.\]
For a matrix-valued function $f$, $f^{ij}$ denotes its $(i,j)-$entry. We refer to \cite[Chapter 3]{Friz2010} for related formulations.

Equivalently, define $z_t = y_{t + T_0} - w_{T_0}$ for $t\in [0,T_1 -T_0]$. Then, solving \eqref{eq_functional_cde} is equivalent to solving:
\begin{align}
\label{eq_functional_cde_equivalent_form}
    z_t = \int_0^t f\big(T_0 + s, (w \sqcup z)_{\cdot \land (T_0 + s)}\big) \dd x_s, \quad z_0 = 0,
\end{align}
for $t\in [0,T]$, where $T \coloneqq T_1 - T_0$, and the concatenated path $w \sqcup z$ is defined by
\begin{align*}
    (w \sqcup z)_t = 
    \begin{cases}
        w_t,\quad &t \in [0,T_0] \\
        w_{T_0} + z_{t - T_0},\quad &t\in [T_0,T_1]
    \end{cases}.
\end{align*}
Here, we commit the slight abuse of notation of denoting the control in both formulations by $x$. Strictly speaking, by change of variables, the CDE in \eqref{eq_functional_cde_equivalent_form} should be driven by $\tilde{x}_t \coloneqq x_{t + T_0}$ with $t \in [0,T]$. Since this distinction plays no role in what follows, we suppress it for notational simplicity and freely use the formulation \eqref{eq_functional_cde_equivalent_form} when convenient. We now state the main result of this section.

\begin{theorem}
\label{prop_EU_global_sol_pd_cde}
    Consider an initial condition $(T_0, w_{\cdot \land T_0}) \in \Lambda^{\alpha\text{-Höl}}_\beta([0,T_1],\RR^m)$ and a continuous, non-anticipative functional $f: \Lambda^{\alpha\text{-Höl}}_\beta([0,T_1],\RR^m) \to L(\RR^d, \RR^m)$. Assume that $f$ is Lipschitz in the path variable on compact sets, i.e. for every compact $K \subset C^{\alpha\text{-Höl}}_\beta([0,T_1],\RR^m)$, there exists a constant $L_{K} > 0$ such that
    \begin{align}
    \label{eq_loc_lip_2nd_variable_intro}
        \big| f (t,x_{\cdot \land t}) - f (t,y_{\cdot \land t}) \big| \leq L_K \big( |x_0 - y_0| + \| x_{\cdot\land t} - y_{\cdot \land t} \|_{\beta\text{-Höl};[0,T_1]} \big),
    \end{align}
    for all $t \in [0,T_1]$ and all $x,y \in K$. In addition, suppose $f$ satisfies the linear growth condition 
    \begin{align}
    \label{eq_path_linear_growth_condition_intro}
    |f(t, y_{\cdot \land t})| \leq C_f \big(1 + |y_0| + \| y_{\cdot \land t} \|_{\subhol[T_1]{\beta}}\big),
    \end{align}
    for all $(t, y_{\cdot \land t}) \in \Lambda^\suphol{\alpha}_\beta([0,T_1],\RR^m)$ and some constant $C_f>0$. Then the path-dependent CDE \eqref{eq_functional_cde} admits a unique global solution. 
\end{theorem}

To mirror classical ODE/FDE theory, we split the analysis of this result across two sections. In Section \ref{subsec_local_existence_uniqueness}, we establish local existence assuming only continuity of $f$, and prove uniqueness under the Lipschitz condition \eqref{eq_loc_lip_2nd_variable_intro} (Propositions \ref{prop_local_existence} and \ref{prop_uniqueness}). Building on this, Section \ref{subsec_global_existence_continuity} shows that local solutions extend globally under the linear growth condition \eqref{eq_path_linear_growth_condition_intro} — see Proposition \ref{prop_global_solution} and Lemma \ref{lem_pasting_argument_bound} — and establishes continuity of the solution map $\big((T_0, w_{\cdot \land T_0}), f\big) \mapsto y $ by deriving explicit stability estimates (Theorem \ref{prop_stability_pd_cdes}).

\begin{remark}
\label{rem_other_path_norms}
    We work with Hölder spaces rather than $p-$variation spaces precisely to maintain a closer parallel with the literature on functional differential equations (see \cite{hale2013introduction}). This choice allows us to separate the proofs of local existence and local uniqueness, each relying on compactness arguments. It also prepares the ground for Section \ref{subsec_universality_of_sig_cdes}, where weighted Hölder spaces are used to obtain global approximation results. Crucially, these arguments depend on the compact embeddings between Hölder spaces — a feature that $p-$variation spaces lack. Nonetheless, well-posedness can equally well be formulated in the $p-$variation setting.
\end{remark}

\subsection{Local Existence and Uniqueness}
\label{subsec_local_existence_uniqueness}

We begin by establishing local existence. The key technical ingredient is an auxiliary lemma,  an adaptation of \cite[Lemma 2.2]{hale2013introduction}, which shows that, on a sufficiently small time interval, the integrand in \eqref{eq_functional_cde_equivalent_form} remains bounded when evaluated along a fixed initial history $w$ and an appropriate class of candidate paths $z$. The proof is deferred to the Appendix \ref{Appendix_3} in order not to interrupt the main line of argument. With this result in hand, local existence of solutions to \eqref{eq_functional_cde} follows from an application of Schauder’s fixed-point theorem, whose statement we recall below for the reader’s convenience.

\begin{theorem}[Schauder's fixed-point theorem, \cite{schauder1930fixpunktsatz}]
    Let $C$ be a non-empty, closed, convex, and bounded subset of a Banach space $X$, and let $T : C \to C$ be a continuous and compact operator. Then $T$ admits a fixed point.
\end{theorem}

\begin{proposition}[Existence]
    \label{prop_local_existence}
    Let $f: \Lambda^{\alpha\text{-Höl}}_\beta([0,T_1],\RR^m) \to L(\RR^d, \RR^m)$ be a continuous, non-anticipative functional, and let $(T_0, w_{\cdot \land T_0}) \in \Lambda^{\alpha\text{-Höl}}_\beta([0,T_1],\RR^m)$ be a given initial condition. Then there exists $\tau > 0$ such that the path-dependent CDE \eqref{eq_functional_cde} admits a solution on the interval $[0, T_0 + \tau]$.
\end{proposition}
\begin{proof}
    Our aim is to show that the equivalent formulation \eqref{eq_functional_cde_equivalent_form} admits a fixed point. For this purpose, consider a compact subset 
    $$
      W \subset [0,T_1] \times C^{\alpha\text{-Höl}}_\beta([0,T_1], \RR^m)
    $$ 
    of the form $W = [0,T_1-\delta] \times B^\alpha(r)$ for some $r>0$ and $\delta>0$.  Let $V, M, \tau$, and $R$ be as specified in Lemma \ref{lem_aux_hale}. Define the operator $\Phi : W/_\sim \times \cD(\tau,R) \to C^{\beta\text{-Höl}}([0,\tau],\RR^m)$ by
    $$\Phi \big( (T_0,w_{\cdot \land T_0}), z \big)(t) = \int_0^t f\big(T_0 + s, (w \sqcup z)_{\cdot \land (T_0 + s)}\big) \dd x_s, \ \text{for} \ t\in [0,\tau].$$
    Clearly, $\Phi$ maps into $C^{\beta\text{-Höl}}([0, \tau], \RR^m)$, since $t\mapsto \Phi \big( (T_0, w_{\cdot \land T_0}), z \big)(t)$ is Lipschitz. Indeed, we have
    \begin{align*}
        &\big|\Phi \big( (T_0,w_{\cdot \land T_0}), z \big)_{s,t} \big| 
        \leq \int_s^t \big|f\big(T_0 + u, (w \sqcup z)_{\cdot \land (T_0 + u)}\big)\big| |\dd x_u| 
        \leq M \| x\|_{1\text{-Höl};[0,\tau]}|t-s|,
    \end{align*}
    for any $0 \leq s \leq t \leq \tau$. Taking this into account, define
    $$
        K \coloneqq \Big\{ z \in C^{1\text{-Höl}}([0,\tau],\RR^m) : z_0 = 0, \|z\|_{1\text{-Höl};[0,\tau]} \leq \tilde{M} \Big\},
    $$
    where $\tilde{M}: = M \|x\|_{1\text{-Höl};[0,\tau]}$. By the compact embedding of Hölder spaces, $K$ is a compact subset of $C^{\beta\text{-Höl}}([0,\tau],\RR^m)$. Furthermore, we may assume that $\tau$ is chosen such that $\tau^{1-\alpha} \tilde{M} \leq R$. This way
    $$
        \|z\|_{\alpha\text{-Höl};[0,\tau]} \leq \tau^{1-\alpha} \tilde{M} \leq R,
    $$
    for all $z \in K$. Consequently, $K \subset \cD(\tau,R) \subset C^{\alpha\text{-Höl}}_\beta([0,\tau],\RR^m)$, and to apply the Schauder fixed-point theorem, it remains to verify the continuity of $\Phi$. With this in mind, assume that
    $$
        \big((T_0^n, w^n_{\cdot \land T_0^n}), z^n\big) \longrightarrow \big((T_0, w_{\cdot \land T_0}), z\big) \ \text{as} \ n\to \infty,
    $$
    in the product space $W/_\sim \times \cD(\tau,R)$. Then, up to a universal constant depending on $R$,
    \begin{align}
    \label{eq_continuity_operator_Phi}
        &d_{\beta\text{-Höl};[0,T_1]}^\Lambda\big( (T_0 + s, (w \sqcup z)_{\cdot \land (T_0+s)}) ,  (T^n_0 + s, (w^n \sqcup z^n)_{\cdot \land (T^n_0+s)}) \big) \nonumber \\ 
        \lesssim& \ d_{\beta\text{-Höl};[0,T_1]}^\Lambda\big( (T^n_0,w^n_{\cdot \land T_0^n}), (T_0,w_{\cdot \land T_0}) \big) + \| z - z^n \|_{\beta\text{-Höl};[0,\tau]} + |T_0 - T_0^n|^{\alpha - \beta}.
    \end{align}

    To see this, fix $s \in (0,\tau]$ and assume $T^n_0 \leq T_0$ (the case $T^n_0 > T_0$ being entirely analogous). Since we may assume $T^n_0$ is sufficiently close to $T_0$, we also have $T_0^n + s > T_0$. The case $s=0$ is trivial. With this in mind, we first observe that
    \begin{align*}
        \| (w \sqcup z)_{\cdot \land (T_0+s)} - (w^n \sqcup& z^n)_{\cdot \land (T^n_0+s)} \|_{\beta\text{-Höl};[0,T_1]} \\
        \leq& \| (w \sqcup z  - w^n \sqcup z^n)_{\cdot \land (T_0+s)} \|_{\beta\text{-Höl};[0,T_1]} + \| w^n \sqcup z^n \|_{\beta\text{-Höl};[T^n_0 + s, T_0 +s]}.
    \end{align*}
    Here, we extend $(w \sqcup z)$ to $[0,T_1]$ by setting $(w \sqcup z)_t = (w \sqcup z)_{T_0 + \tau}$ for $t > T_0+\tau$, and similarly extend $(w^n \sqcup z^n)$ from $[0,T^n_0+\tau]$ to $[0,T_1]$. We now estimate the two terms separately. For the second term,
    \begin{align*}
        \| w^n \sqcup z^n \|_{\beta\text{-Höl};[T^n_0 + s, T_0 +s]} 
        \leq \| z^n \|_{\alpha\text{-Höl};[0,\tau]}  |T_0 - T_0^n|^{\alpha-\beta}.
    \end{align*}
    For the first term, we have
    \begin{align*}
        \| (w \sqcup z  - w^n \sqcup z^n) \|_{\beta\text{-Höl};[0,T_0+s]} 
        &\leq  \| w - w^n \|_{\beta\text{-Höl};[0,T_0^n]} + \big( \| w \|_{\beta\text{-Höl};[T_0^n,T_0]} + \|z^n\|_{\beta\text{-Höl};[0,T_0-T_0^n]} \big)  \\
        & \hspace{5.4cm} + \| z_{\cdot - T_0} - z^n_{\cdot - T^n_0} \|_{\beta\text{-Höl};[T_0,T_0+s]}
    \end{align*}
    and, since $z^n \in \cD(\tau,R)$, we can further estimate
    \begin{align*}
        \| z_{\cdot - T_0} - z^n_{\cdot - T^n_0} \|_{\beta\text{-Höl};[T_0,T_0+s]} &\leq \| z - z^n \|_{\beta\text{-Höl};[0,s]} + \| z^n_{\cdot - T_0} - z^n_{\cdot - T^n_0} \|_{\beta\text{-Höl};[T_0,T_0+s]}\\
        &\leq \| z - z^n \|_{\beta\text{-Höl};[0,\tau]} + 2 \|z^n\|_{\alpha\text{-Höl};[0,\tau]} |T_0 - T_0^n|^{\alpha-\beta}.
    \end{align*}
    Combining these estimates and recalling the crucial assumption that all paths involved have uniformly bounded $\alpha$–Hölder norms, we obtain the desired bound \eqref{eq_continuity_operator_Phi}. Therefore, 
    $$
        (T^n_0 + s,(w^n \sqcup z^n)_{\cdot \land (T^n_0+s)}) \to (T_0 + s, (w \sqcup z)_{\cdot \land (T_0+s)}) \ \text{in} \ \Lambda^{\alpha\text{-Höl}}_\beta([0,T_1],\RR^m),
    $$ 
    and the continuity of $\Phi$ follows by the continuity of $f$. Take $h(u) = f\big(T_0 + u, (w\sqcup z)_{\cdot \land (T_0+u)}\big)$ and $h_n(u) = f\big(T^n_0 + u, (w^n\sqcup z^n)_{\cdot \land (T^n_0+u)}\big)$, then
    \[
    \big|\big(\Phi\big((T_0,w_{\cdot\land T_0}),z\big) - \Phi\big((T^n_0,w^n_{\cdot\land T^n_0}),z^n\big)\big)_{s,t}\big| \leq \| h -h_n \|_{\infty;[0,\tau]} \|x\|_{1\text{-Höl};[0,T_1]}|t-s| \longrightarrow 0.
    \]

    Existence of a solution now follows from the fact that, for any initial condition $(T_0, w_{\cdot \land T_0})$, the singleton set $\{(T_0, w_{\cdot \land T_0})\}$ is compact, and $\cD(\tau, R)$ is a closed, convex, and bounded subset of $C^{\beta\text{-Höl}}([0, \tau], \RR^m)$. Therefore, by the Schauder fixed-point theorem, the map $\Phi\big((T_0, w_{\cdot \land T_0}), \cdot \big): \cD(\tau, R) \to \cD(\tau, R)$ admits a fixed point, which in turn yields a solution to \eqref{eq_functional_cde_equivalent_form}. %
\end{proof}

\begin{remark}
    By Lemma \ref{lem_aux_hale} and Remark \ref{rem_function_neighbourhood}, Proposition \ref{prop_local_existence} extends to the following more general form. Given a compact set $W = [0,T_1-\delta] \times B^\alpha(r)$, and a continuous, non-anticipative functional $f$, there exist: a neighbourhood $V$ of $W$ such that $f|_{V/\sim}$ is bounded; a neighbourhood $U \subset C^0_b(V/_\sim , L(\RR^d,\RR^m))$ of $f|_{V/\sim}$; and a time horizon $\tau >0$, such that for every $(T_0, w_{\cdot\land T_0}) \in W/_\sim$ and every $\tilde{f} \in U$, the functional CDE \eqref{eq_functional_cde} admits a solution on $[0,T_0+\tau]$.
\end{remark}

Next, we establish uniqueness of solutions to \eqref{eq_functional_cde} under the additional assumption that the functional $f$ is locally Lipschitz in the path variable (over compacts). This assumption aligns with classical results from ODE theory, and the proof below follows a standard argument.

\begin{proposition}[Uniqueness]
    \label{prop_uniqueness}
    Let $f : \Lambda^{\alpha\text{-Höl}}_\beta([0,T_1],\RR^m) \to L(\RR^d,\RR^m)$ be a continuous, non-anticipative functional. Suppose $f$ is Lipschitz in the path variable over compacts, that is, for each compact set $K \subset C^{\alpha\text{-Höl}}_\beta([0,T_1],\RR^m)$, there exists a constant $L_{K} > 0$ such that
    \begin{align}
    \label{eq_loc_lip_2nd_variable}
        \big| f (t,x_{\cdot \land t}) - f (t,y_{\cdot \land t}) \big| \leq L_K \big( |x_0 - y_0| + \| x_{\cdot\land t} - y_{\cdot \land t} \|_{\beta\text{-Höl};[0,T_1]} \big),
    \end{align}
    for all $t \in [0,T_1]$ and $x,y \in K$. Then, given any $(T_0, w_{\cdot \land T_0}) \in \Lambda^{\alpha\text{-Höl}}_\beta([0,T_1],\RR^m)$, there exists a unique local solution to \eqref{eq_functional_cde} with initial condition $(T_0, w_{\cdot \land T_0})$.
\end{proposition}
\begin{remark}
    We avoid the term \emph{locally Lipschitz}, since condition \eqref{eq_loc_lip_2nd_variable} is formulated on compact subsets rather than neighborhoods of points. While on locally compact domains these notions coincide, the path space $\Lambda^{\alpha\text{-Höl}}_\beta([0,T_1],\RR^m)$ is not locally compact. Thus, Lipschitz continuity on compacts is necessary but not sufficient for local Lipschitz continuity: in particular, if $f :\Lambda^{\alpha\text{-Höl}}_\beta([0,T_1],\RR^m) \to L(\RR^d,\RR^m)$ is locally Lipschitz, then condition \eqref{eq_loc_lip_2nd_variable} follows, but not conversely. 
\end{remark}
\begin{proof}
    Assume $y^{(1)}$ and $y^{(2)}$ are two distinct solutions of the path-dependent CDE \eqref{eq_functional_cde} on $[0,T_0+\tau]$, with $\tau>0$ defined as in Proposition \ref{prop_local_existence}. Then, $y^{(1)}_t \equiv y^{(2)}_t \equiv w_t$ for $t \in [0,T_0]$ and, for $t \in [T_0,T_0 + \tau]$, we have
    $$
        y^{(1)}_t - y^{(2)}_t = \int_{T_0}^t \big( f(s, y^{(1)}_{\cdot \land s}) - f(s,y^{(2)}_{\cdot \land s}) \big) \dd x_s.
    $$
   Next, let $L_K$ be the Lipschitz constant of $f$ associated with the compact set $K = B^\alpha(R)$, where 
   \[
   R \coloneqq \| y^{(1)}_{\cdot \land (T_0 + \tau)} \|_{\alpha\text{-Höl};[0,T_1]} \lor \| y^{(2)}_{\cdot \land (T_0 + \tau)} \|_{\alpha\text{-Höl};[0,T_1]},
   \]
    and fix $\sigma_K>0$ such that $ L_K \|x\|_{1\text{-Höl};[0,T_1]} \sigma_K^{1-\beta} < 1$. 
    Then, using \eqref{eq_loc_lip_2nd_variable} and the fact that $\| y^{(1)} - y^{(2)} \|_{\beta\text{-Höl};[0,T_0+\sigma_K]} = \| y^{(1)} - y^{(2)} \|_{\beta\text{-Höl};[T_0,T_0+\sigma_K]}$, we obtain
    \begin{align*}
        \| y^{(1)} - y^{(2)} \|_{\beta\text{-Höl};[T_0,T_0+\sigma_K]} &\leq L_K \| y^{(1)} - y^{(2)} \|_{\beta\text{-Höl};[T_0,T_0+\sigma_K]} \|x\|_{1\text{-Höl};[0,T_1]} \sigma_K^{1-\beta} \\
        &< \| y^{(1)} - y^{(2)} \|_{\beta\text{-Höl};[T_0,T_0+\sigma_K]}
    \end{align*}
    leading to a contradiction. Hence $y^{(1)}_t = y^{(2)}_t$ for all $t \in [T_0,T_0 + \sigma_K]$, and repeating the same argument on successive subintervals of length $\sigma_K$ yields the claim on the whole interval $[T_0,T_0+\tau]$. Note that $\sigma_K$ depends only on $\beta, \|x\|_{1\text{-Höl};[0,T_1]}$, and $K$, all of which remain fixed throughout the iteration.
\end{proof}

\subsection{Global Solutions and Continuity}
\label{subsec_global_existence_continuity}

So far, we have established the existence and uniqueness of a local solution $y$ to \eqref{eq_functional_cde}. More precisely, under the assumptions of Proposition \ref{prop_uniqueness}, we can iteratively apply Proposition \ref{prop_local_existence} to obtain a unique solution first on $[0,T_0 + \tau_1]$, then on $[0,T_0 + \tau_1 + \tau_2]$, and so forth. This procedure yields a unique maximal solution, which is defined on an interval of the form $[0,\tilde{T})$ for some $\tilde{T} \leq T_1$. Importantly, this interval must have an open right endpoint. If the sequence $(\tau_n)$ decays sufficiently quickly so that $\sum_n \tau_n < T_1 - T_0$, then $\tilde{T} < T_1$, and we refer to the resulting solution as \textit{truly maximal}. The next result provides a criterion to determine whether a solution is truly maximal or can be extended further. In turn, this will yield sufficient conditions for obtaining a \textit{global solution} — that is, a solution defined on the full interval $[0,T_1]$, for any fixed $T_1 > 0$.

\begin{proposition}[Global Solution]
\label{prop_global_solution}
    Let $f: \Lambda^{\alpha\text{-Höl}}_\beta([0,T_1],\RR^m) \to L(\RR^d, \RR^m)$ be a continuous, non-anticipative functional. If $y: [0, \tilde{T}) \to \RR^m$ is a truly maximal solution, then for any compact set $K \subset [0, T_1] \times C^{\alpha\text{-Höl}}_\beta([0, T_1], \RR^m)$, there exists a time $T_K < \tilde{T}$ such that $(t, y_{\cdot \land t}) \notin K/_\sim$ for all $T_K \leq t < \tilde{T}$.
\end{proposition}
\begin{proof}
    Arguing by contradiction, let us assume that $y:[0,\tilde{T}) \to \RR^m$ is a truly maximal solution of \eqref{eq_functional_cde}, and there exists a compact set $K$ together with a sequence $t_k \uparrow \tilde{T}$ such that $(t_k, y_{\cdot \land t_k}) \in K/_\sim$ for every $k\geq 1$. Then, by compactness and up to sequence relabeling, there exists $(\tilde{T},\tilde{y}) \in K$ such that $(t_k, y_{\cdot \land t_k}) \longrightarrow (\tilde{T},\tilde{y}_{\cdot \land \tilde{T}}) \in K/_\sim$.
    
    Subsequently, note that $y_{\cdot \land t_k} \to \tilde{y}_{\cdot \land \tilde{T}}$ in the uniform topology. Hence, for all $t \in [0,T_1]$, we have that $y_{t \land t_k} \longrightarrow \tilde{y}_{t \land \tilde{T}}$ as $k \to \infty$. In particular, $y_t = \tilde{y}_t$ for every $t < \tilde{T}$, and since $\tilde{y}$ is continuous $\lim_{t \to \tilde{T}^-} y_t = \lim_{t \to \tilde{T}^-} \tilde{y}_t = \tilde{y}_{\tilde{T}}$. Setting $y_{\tilde T} = \tilde{y}_{\tilde T}$ naturally extends $y$ to a solution on $[0,\tilde T]$. Indeed, by the continuity of the Riemann-Stieltjes integral, passing $t \to \tilde T^-$ gives
    \[
    y_{\tilde T} = w_{T_0} + \int_{T_0}^{\tilde T} f(s,y_{\cdot \land s}) \dd x_s.
    \]
    However, by Proposition \ref{prop_local_existence}, this implies that there exists a solution of \eqref{eq_functional_cde} passing through $(\tilde{T}, y_{\cdot \land \tilde{T}})$ and defined to the right of $\tilde{T}$, thus contradicting the assumption of having a truly maximal solution.
\end{proof}

Although Proposition \ref{prop_global_solution} may appear too abstract for practical use, it yields concrete sufficient conditions for the existence of a global solution. As in classical ODE theory, a linear growth condition is sufficient for the existence of a global solution. More precisely, suppose there exists $C_f > 0$ such that
\begin{align}
\label{eq_path_linear_growth_condition}
    |f(t, y_{\cdot \land t})| \leq C_f \big(1 + |y_0| + \| y_{\cdot \land t} \|_{\subhol[T_1]{\beta}}\big),
\end{align}
for all $(t, y_{\cdot \land t}) \in \Lambda^\suphol{\alpha}_\beta([0,T_1],\RR^m)$. Assume $y : [0, \tilde{T}) \to \RR^m$ to be a truly maximal solution. For $T_0 \leq \sigma < \tau < \tilde{T}$ and $\sigma \leq s < t \leq \tau$, we have
\begin{align}
\label{eq_a_priori_bound_first_estimate}
    |y_{s,t}| &\leq \int_s^t C_f \big(1 + |y_0| + \| y_{\cdot \land u} \|_{\subhol[T_1]{\beta}}\big) |dx_u| \\
    &\leq C_f \big(1 + |y_0| + \| y_{\cdot \land t} \|_{\subhol[T_1]{\beta}}\big) \| x \|_{\subhol[T_1]{1}} |t-s|.
\end{align}
Dividing through by $|t-s|^\gamma$ for any $\gamma \in [\beta, \alpha]$ and taking the supremum over $\sigma \leq s < t \leq \tau$ yields
\begin{align}
\label{eq_local_to_global_lin_growth_bound}
    \|y\|_{\gamma\text{-Höl};[\sigma,\tau]} &\leq C_f \big(1 + |y_0| + \| y \|_{\subhol[\tau]{\beta}}\big) \| x \|_{\subhol[T_1]{1}} |\tau-\sigma|^{1-\gamma} \nonumber \\
    &\leq C_f \big(1 + |y_0| + \tau^{\gamma - \beta} \| y \|_{\subhol[\tau]{\gamma}}\big) \| x \|_{\subhol[T_1]{1}} |\tau-\sigma|^{1-\gamma},
\end{align}
leading to a local-to-global estimate. The next lemma demonstrates that estimates of this kind yield an explicit bound for the $\gamma$–Hölder semi-norm of $y$ on any interval $[T_0,T']$ with $T' > T_0$ for which $y|_{[0,T']}$ is well-defined.

\begin{lemma}
\label{lem_pasting_argument_bound}
    Fix $T' > T_0$ and let $y \in C^{\gamma\text{-Höl}}([0,T'],\RR^m)$. Assume there exist  constants $C_0,C_1 > 0$ such that
    \begin{align}
        \label{eq_pasting_argument_lem_hyp}
        \|y\|_{\gamma\text{-Höl};[\sigma,\tau]} \leq C_0 (C_1 + \| y \|_{\gamma\text{-Höl};[0,\tau]}) |\tau - \sigma|^{1-\gamma},
    \end{align}
    for every $T_0 \leq \sigma < \tau \leq T'$. Then it follows that
    $$\|y\|_{\gamma\text{-Höl};[T_0,T']} \leq \exp\big( \log(2)  \lceil T' (2C_0)^{1/(1-\gamma)} \rceil\big) (C_1 + 2\| y \|_{\gamma\text{-Höl};[0,T_0]}).$$
\end{lemma}
\begin{proof}
    Recall that for any $0 \leq \sigma < \tau \leq T'$, the following subadditivity property holds:
    \begin{align}
    \label{eq_subadditivity_Holder_seminorm}
        \| y \|_{\gamma\text{-Höl};[0,\tau]} \leq \| y \|_{\gamma\text{-Höl};[0,\sigma]} + \| y \|_{\gamma\text{-Höl};[\sigma,\tau]}.
    \end{align}
    Set $\delta = (2C_0)^{-1/(1-\gamma)}$ and construct a partition $T_0 = t_0 < t_1 < \cdots < t_N = T'$ such that $|t_k - t_{k-1}| = \delta$ for $k=1, \hdots, N-1$. Applying \eqref{eq_subadditivity_Holder_seminorm} with $\sigma = t_{k-1}$, $\tau = t_k$, and invoking \eqref{eq_pasting_argument_lem_hyp}, we obtain
    \begin{align*}
        \| y \|_{\gamma\text{-Höl};[0,t_k]} &\leq \| y \|_{\gamma\text{-Höl};[0,T_0]} + \| y \|_{\gamma\text{-Höl};[T_0,t_{k-1}]} + \| y \|_{\gamma\text{-Höl};[t_{k-1},t_k]} \\
        &\leq \| y \|_{\gamma\text{-Höl};[0,T_0]} + \| y \|_{\gamma\text{-Höl};[T_0,t_{k-1}]} + C_0(C_1 + \| y \|_{\gamma\text{-Höl};[0,t_k]}) \delta^{1-\gamma}.
    \end{align*}
    Since $\delta$ was chosen so that $C_0\delta^{1-\gamma} = \frac{1}{2}$, it follows that
    $$
       \| y \|_{\gamma\text{-Höl};[T_0,t_k]}  \leq \| y \|_{\gamma\text{-Höl};[0,t_k]} \leq 2\| y \|_{\gamma\text{-Höl};[0,T_0]} + 2 \| y \|_{\gamma\text{-Höl};[T_0,t_{k-1}]}  + C_1. 
    $$
    For $k=1$ this reduces to $\| y \|_{\gamma\text{-Höl};[T_0,t_1]} \leq C_1 + 2\| y \|_{\gamma\text{-Höl};[0,T_0]}$, and solving the recurrence yields
    \begin{align}
    \label{eq_constant_N_a_priori_estimate}
        \| y \|_{\gamma\text{-Höl};[T_0,T']} \leq (2^N - 1) \big(C_1 + 2\| y \|_{\gamma\text{-Höl};[0,T_0]}\big), \quad \text{with} \ N = \lceil (T'-T_0) (2C_0)^{1/(1-\gamma)} \rceil.
    \end{align}
This concludes the proof.   
\end{proof}

Applying \eqref{eq_local_to_global_lin_growth_bound} in combination with Lemma \ref{lem_pasting_argument_bound}, we deduce an \textit{a priori} estimate $\|y\|_{\subhol[T']{\gamma}} \leq A \equiv A(x,w,C_f,T_1,\gamma)$ for all $T' < \tilde{T}$. When $\gamma = \alpha$, it then follows that the path $t \mapsto (t, y_{\cdot \land t})$ is confined to the compact set $K /_{\sim}$ with $K = [0,T_1] \times B^\alpha(A)$. By Proposition \ref{prop_global_solution}, this leads to a contradiction, proving that $y$ must in fact be global.

We conclude this section with a stability result, which will be instrumental in Section \ref{subsec_universality_of_sig_cdes} when we specialize to Sig-CDEs and analyze their universality property. Specifically, the next proposition ensures that solutions to two path-dependent CDEs remain close whenever their respective initial conditions and non-anticipative vector fields are close. 

\begin{theorem}%
\label{prop_stability_pd_cdes}
    Let $f^{(1)}, f^{(2)}$ be continuous, non-anticipative functionals satisfying \eqref{eq_loc_lip_2nd_variable} and \eqref{eq_path_linear_growth_condition}. Consider initial data $(T_0, w^{(1)}_{\cdot \land T_0}), (T_0, w^{(2)}_{\cdot \land T_0}) \in \Lambda^{\alpha\text{-Höl}}_\beta([0,T_1],\RR^m)$. For $i = 1,2$, let $y^{(i)}$ denote the global solution of the path-dependent CDE
    \begin{align*}
    \begin{cases}
    y^{(i)}_t = w^{(i)}_t, \quad &t \in [0,T_0]\\
    y^{(i)}_t = w^{(i)}_{T_0} + \displaystyle\int_{T_0}^t f^{(i)}\big(s, y^{(i)}_{\cdot \land s}\big) \dd x_s, \quad &t \in [T_0,T_1]
    \end{cases}.
    \end{align*}
    Then there exists a compact set $\mathcal{K} \subset [0,T_1] \times C^{\suphol{\alpha}}_\beta([0,T_1],\RR^m)$ containing the trajectories of $y^{(1)}$ and $y^{(2)}$ such that, up to a constant depending on $\mathcal{K}$ and $x$, the following stability estimate holds:
    \begin{align}
    \label{eq_stability_bound}
        \| y^{(1)} - y^{(2)} \|_{\beta\text{-Höl};[0,T_1]} \lesssim  \Big( |w^{(1)}_0 - w^{(2)}_0| + \| w^{(1)} - w^{(2)} \|_{\beta\text{-Höl};[0,T_0]}  + \| (f^{(1)} - f^{(2)})|_{\mathcal{K}/_\sim} \|_{\infty} \Big).
    \end{align}
    Particularly, if $w^{(1)}|_{[0,T_0]} = w^{(2)}|_{[0,T_0]}$, then
    \begin{align}
        \| y^{(1)} - y^{(2)} \|_{\beta\text{-Höl};[T_0,T_1]} \leq 2^N \Big(1 + \frac{1}{L_{\cK}} \Big) \| (f^{(1)} - f^{(2)})|_{\mathcal{K}/_\sim} \|_{\infty}
    \end{align}
    where $N = \left\lceil T_1 \big(2 L_{\cK} \|x\|_{1\text{-Höl};[0,T_1]} \big)^{1/(1-\beta)} \right\rceil$ and $L_{\cK} > 0$ is a constant depending on $\cK$.
\end{theorem}
\begin{proof}
    For $i \in \{1,2\}$, let $z^{(i)}$ denote the solution of the equivalent formulation \eqref{eq_functional_cde_equivalent_form}. We prove an estimate for $\| z^{(1)} - z^{(2)} \|_{\beta\text{-Höl};[0,T]}$, which then translates straightforwardly to \eqref{eq_stability_bound}. Fix $0 \leq \sigma < \tau \leq T$ and observe that, for any $\sigma \leq s \leq t \leq \tau$,
    \begin{align*}
        &\big| (z^{(1)} - z^{(2)})_{s,t} \big| \\
        \leq& \int_s^t \big| f^{(1)}\big( u + T_0, (w^{(1)} \sqcup z^{(1)})_{\cdot \land (u + T_0)} \big) -  f^{(1)}\big( u + T_0, (w^{(2)} \sqcup z^{(2)})_{\cdot \land (u + T_0)} \big) \big| \cdot |\dd x_u| \\
        &\hspace{1.5cm} + \int_s^t \big|f^{(1)}\big( u + T_0, (w^{(2)} \sqcup z^{(2)})_{\cdot \land (u + T_0)} \big) - f^{(2)}\big(  u + T_0, (w^{(2)} \sqcup z^{(2)})_{\cdot \land (u + T_0)} \big) \big| \cdot |\dd x_u|.
    \end{align*}
    Let $\mathcal{K}$ denote a compact subset of $[0,T_1] \times C^{\suphol{\alpha}}_\beta([0,T_1],\RR^m)$ containing 
    $$
        \Big\{ \big(u + T_0, (w^{(i)} \sqcup z^{(i)})\big) : u \in [0,T] \Big\},
    $$ 
    for $i \in \{1,2\}$. Such compact set exists thanks to Lemma \ref{lem_pasting_argument_bound}. Then, by condition \eqref{eq_loc_lip_2nd_variable} we obtain
    \begin{align*}
        &\big| (z^{(1)} - z^{(2)})_{s,t} \big| \\
        \leq& \int_s^t L_{\cK} \big( |w^{(1)}_0 - w^{(2)}_0| + \| (w^{(1)} \sqcup z^{(1)})_{\cdot \land (u + T_0)} - (w^{(2)} \sqcup z^{(2)})_{\cdot \land (u + T_0)} \|_{\beta\text{-Höl};[0,T_1]} \big) \cdot |\dd x_u| \\
        &\hspace{2.5cm} + \int_s^t \| (f^{(1)} - f^{(2)})|_{\mathcal{K}/_\sim} \|_{\infty} \cdot |\dd x_u|,
    \end{align*}
    where $L_{\mathcal{K}}>0$ is a constant depending on the compact $\mathcal{K}$. Dividing by $|t-s|^\beta$ on both sides and taking the supremum over $\sigma \leq s \leq t \leq \tau$ then yields:
    \begin{align}
    \label{eq_local_global_bound}
        &\| z^{(1)} - z^{(2)} \|_{\beta\text{-Höl};[\sigma,\tau]} \nonumber \\
        \leq&  \ L_{\cK} \big( |w^{(1)}_0 - w^{(2)}_0| + \| w^{(1)} - w^{(2)} \|_{\beta\text{-Höl};[0,T_0]} + \| z^{(1)} - z^{(2)} \|_{\beta\text{-Höl};[0,\tau]} \big) \| x \|_{1\text{-Höl};[0,T]} |\tau - \sigma|^{1-\beta} \nonumber \\ 
        &\hspace{1.5cm} +  \| (f^{(1)} - f^{(2)})|_{\cK/_\sim} \|_{\infty} \| x \|_{1\text{-Höl};[0,T]} |\tau - \sigma|^{1-\beta}.
    \end{align} 
    Now, to abbreviate notation define
    \begin{align*}
        C_w &\coloneqq L_{\mathcal{K}}  \big( |w^{(1)}_0 - w^{(2)}_0| + \| w^{(1)} - w^{(2)} \|_{\beta\text{-Höl};[0,T_0]} \big)   \| x \|_{1\text{-Höl};[0,T]}, \\
        C_z &\coloneqq L_{\mathcal{K}}  \| x \|_{1\text{-Höl};[0,T]}, \ \text{and} \ D_f \coloneqq  \| (f^{(1)} - f^{(2)})|_{\mathcal{K}/_\sim} \|_{\infty} \, \| x \|_{1\text{-Höl};[0,T]}.
    \end{align*}
    Then, estimate \eqref{eq_local_global_bound} reads
    $$
        \| z^{(1)} - z^{(2)} \|_{\beta\text{-Höl};[\sigma,\tau]} \leq \Big( C_w + D_f + C_z \| z^{(1)} - z^{(2)} \|_{\beta\text{-Höl};[0,\tau]} \Big) |\tau-\sigma|^{1-\beta},
    $$
    and we are under the conditions of Lemma \ref{lem_pasting_argument_bound} with $\gamma = \beta$, $C_0 = C_z$, and $C_1 = (C_w + D_f)C_z^{-1}$. Note that the estimate   \eqref{eq_local_global_bound} holds for all $0 \leq \sigma < \tau \leq T$, so we may take $T_0=0$ in the statement of Lemma \ref{lem_pasting_argument_bound}. We then obtain
    \begin{align*}
        &\| z^{(1)} - z^{(2)} \|_{\beta\text{-Höl};[0,T]}\\
        &\hspace{1.2cm} \leq  2^N \, \frac{L_{\mathcal{K}} \lor 1}{L_\mathcal{K}} \Big( |w^{(1)}_0 - w^{(2)}_0| + \| w^{(1)} - w^{(2)} \|_{\beta\text{-Höl};[0,T_0]}  + \| (f^{(1)} - f^{(2)})|_{\mathcal{K}/_\sim} \|_{\infty} \Big),
    \end{align*}
    where $N$ is as in \eqref{eq_constant_N_a_priori_estimate} with $T' = T$. Finally, the estimate \eqref{eq_stability_bound} follows by recalling that $z^{(i)}_t = y^{(i)}_{t + T_0} - w^{(i)}_{T_0}$ for $i \in \{1,2\}$.
\end{proof}

\section{Signature Controlled Differential Equations}
\label{sec_sig_cdes}

We now focus on a specific class of path-dependent controlled differential equations, which we call Sig-CDEs. Rather than working with a generic non-anticipative functional as in \eqref{eq_functional_cde}, we define the vector field through the signature of the solution path. More precisely, using the notation of the previous section, we set $f = F \circ S$, where $S : \Lambda^{\alpha\text{-Höl}}_\beta([0,T_1],\RR^m) \to T((\RR^m))$ denotes the signature map, and $F : T((\RR^m)) \to L(\RR^d,\RR^m)$ is a matrix-valued map defined on the extended tensor algebra (or on a subset containing all signatures).

Given a control path $x:[0,T_1] \to \RR^d$ and an initial condition $y|_{[0,T_0]} = w|_{[0,T_0]}$, we then define a Sig-CDE as the path-dependent equation
\begin{align}
\label{eq_sig_cde_integral_form}
    y_t = w_{T_0} + \displaystyle\int_{T_0}^t F(S(y|_{[0,s]})) \dd x_s, \qquad t \in [T_0,T_1].
\end{align}

The motivation for this class is twofold. First, by the universality of signatures (cf. \Cref{prop_global_approx_signatures}), it is natural to expect that replacing a general path functional with a signature-based one yields a universal subclass of path-dependent CDEs. In particular, Theorem \ref{prop_stability_pd_cdes} suggests that solutions to \eqref{eq_functional_cde} can be approximated by those of \eqref{eq_sig_cde_integral_form}. Second, when $f$ admits a factorization $f=F\circ S$, the equation can be lifted to an infinite-dimensional classical CDE, which — after truncation — reduces to a CDE on a finite-dimensional manifold. The following sections make these ideas precise.

\subsection{Well-Posedness}
\label{subsec_well_posedness_sig_cdes}

In this section, we study the well-posedness of the Sig-CDEs \eqref{eq_sig_cde_integral_form} by viewing them as a special case of the path-dependent CDEs \eqref{eq_functional_cde}. This allows us to invoke the results of Section \ref{sec_path_dependent_cdes} to obtain existence and uniqueness of global solutions. \Cref{prop_well_posedness_SigCDEs} provides concrete conditions on $F$, formulating linear growth in terms of tensor dilations and Lipschitz continuity via the tensor limit spaces introduced in \Cref{subsec_limits_of_tensor_spaces}.

We begin by noting that the composition $F\circ S$ defines a non-anticipative functional. Hence, a direct application of \Cref{prop_EU_global_sol_pd_cde} yields the following general result:

\begin{corollary}
\label{cor_E&U_SigCDEs_general}
    Let $\alpha \in \left(\frac{1}{2},1\right)$ and $\beta < \alpha$. Suppose $F \circ S$ is a continuous functional that is  Lipschitz in the path variable over compacts (cf. \eqref{eq_loc_lip_2nd_variable}), and further satisfies the linear growth condition \eqref{eq_path_linear_growth_condition}. Then the Sig-CDE \eqref{eq_sig_cde_integral_form} admits a unique solution on the interval $[0,T_1]$. 
\end{corollary}

The main interest, however, is in deriving explicit sufficient conditions directly on $F$, since the signature map is a fixed and intrinsic component of Sig-CDEs. The following result provides a first step in this direction by showing that the linear growth condition \eqref{eq_path_linear_growth_condition} at the path level admits an equivalent formulation at the signature level.

\begin{proposition}
\label{prop_homogeneous_linear_growth}
    The functional $f = F \circ S$ satisfies the linear growth condition \eqref{eq_path_linear_growth_condition} if and only if $F$ exhibits homogeneous linear growth; that is, for every $r>0$, there exists $C_F > 0$ such that
    \begin{align}
    \label{eq_homogeneous_lin_growth}
        |F(\delta_\lambda \mathbf{g})| \leq C_F (1 + \lambda), \quad \text{for all}  \ \mathbf{g} \in S\big(B^\beta(r)\big), \text{and} \ \lambda \geq 0.
    \end{align}
\end{proposition}
\begin{proof}
    Suppose $f = F \circ S$ satisfies the linear growth condition \eqref{eq_path_linear_growth_condition}. Let $\mathbf{g}$ denote $S(y|_{[0,t]})$ for some $y|_{[0,t]} \in \Lambda^\suphol{\alpha}_\beta([0,T_1],\RR^m)$. Then, for any $\lambda \geq 0$,
    \begin{align*}
        |F(\delta_\lambda \mathbf{g})| = |F(\delta_\lambda S(y|_{[0,t]}))| = |f(\lambda y|_{[0,t]})| \leq C_f(1 + \lambda \big(|y_0| + \| y_{\cdot\land t}\|_{\subhol[T_1]{\beta}}) \big),
    \end{align*}
    which immediately yields \eqref{eq_homogeneous_lin_growth} with constant $C_F = C_f (1 \lor r)$ for any fixed $r>0$. 
    
    Conversely, let $r, C_F > 0$ be such that \eqref{eq_homogeneous_lin_growth} holds. For any $y|_{[0,t]} \in \Lambda^\suphol{\alpha}_\beta([0,T_1], \RR^m)$, set $\lambda = r^{-1} (|y_0| + \| y_{\cdot\land t}\|_{\subhol[T_1]{\beta}})$ and define $x|_{[0,t]} = \lambda^{-1} y|_{[0,t]}$. Here, we may assume that $y|_{[0,t]} \neq 0$, otherwise the linear growth condition is trivially satisfied. Then $S(x|_{[0,t]}) \in S(B^\beta(r))$, and
    \begin{align*}
        |f(y|_{[0,t]})| &= |F(S(\lambda \lambda^{-1} y|_{[0,t]}))| = |F(\delta_\lambda S(x|_{[0,t]}))| \leq C_F (1 + \lambda) \\
        &\leq C_F \big(1 + r^{-1} (|y_0| + \| y_{\cdot\land t}\|_{\subhol[T_1]{\beta}})\big),
    \end{align*}
    which shows that the linear growth condition \eqref{eq_path_linear_growth_condition} holds with constant $C_f = C_F(1 \lor r^{-1})$.
\end{proof}

\begin{remark}
    We note that Proposition \ref{prop_homogeneous_linear_growth} remains valid under alternative choices of path semi-norm. For example, if the linear growth condition \eqref{eq_path_linear_growth_condition} were formulated using another semi-norm, such as $1-$variation, the same equivalence would hold with the corresponding modification of \eqref{eq_homogeneous_lin_growth}. This does not mean that all path-level linear growth conditions are equivalent; rather, it highlights once more the flexibility in the choice of path space; recall Remark \ref{rem_other_path_norms}.
\end{remark}

In \Cref{subsec_universality_of_sig_cdes}, we show that linear functionals of the signature, composed with activation functions of logarithmic growth, generate a universal class of maps $F$ satisfying homogeneous linear growth. Moreover, homogeneous gauges (Section \ref{subsec_path_signatures_tensor_algebras}) also provide concrete examples of functions fulfilling \eqref{eq_homogeneous_lin_growth}; see \cite[Section 7.5.3]{Friz2010}. 

Turning to continuity, the key point is the choice of codomain for the signature map, as $F \circ S$ factors through a tensor space containing the signatures and this choice directly determines the regularity required of $F$. In view of Theorem \ref{prop_EU_global_sol_pd_cde}, a natural option is to work in a metric space where the signature map is locally Lipschitz. In this setting, a Lipschitz condition on $F$ suffices to ensure the continuity assumption for well-posedness, and stronger metric topologies correspond to weaker Lipschitz requirements on $F$.

\begin{remark}
\label{rem_degenerate_topology}
   This suggests endowing the set of signatures with the strongest metric for which the signature map remains locally Lipschitz, thereby maximizing the class of admissible functionals $F$ — a ``final Lipschitz topology". However, such an approach comes at the cost of tractability. In the extreme, one obtains a setting in which paths are effectively identified with their signatures, leading to a topology akin to a final topology on unparametrized path space; see \cite{cass2024topologies} for related discussions.
\end{remark} 

We therefore take $\cT^{(p)}(\RR^m)$ as the codomain of the signature map. By Proposition \ref{lem_continuity_sig_map}, the signature map is locally Lipschitz in this space, and Lipschitz continuity of $F$ with respect to standard tensor spaces $T_p(\RR^m)$ implies Lipschitz continuity in $\cT^{(p)}(\RR^m)$, which we require in the proposition below. Combining Proposition \ref{lem_continuity_sig_map} with Proposition \ref{prop_homogeneous_linear_growth}, we obtain the following well-posedness result.

\begin{proposition}
    \label{prop_well_posedness_SigCDEs}
    Fix $1/2 < \beta < \alpha < 1$. Consider $p>1$ and $(\lambda_k)$ such that $\lambda_k^{1/\beta}/k^p \to 0$ as $k \to \infty$. Suppose $F : \cT^{(p)}(\RR^m) \to L(\RR^d,\RR^m)$ is locally Lipschitz and satisfies the homogeneous linear growth condition \eqref{eq_homogeneous_lin_growth}. Under these assumptions, the Sig-CDE \eqref{eq_sig_cde_integral_form} possesses a unique solution on the interval $[0,T_1]$.
\end{proposition}

\begin{remark}
Beyond the topological considerations discussed above and in Section \ref{subsec_limits_of_tensor_spaces}, our preference for $\cT^{(p)}(\RR^m)$ is also motivated by two additional points. First, its Banach space structure allows us to develop in Section \ref{subsec_scaling_argument} a well-posedness theory for lifted Sig-CDEs that closely mirrors the path-dependent setting. Second, the space $\cT^{(p)}(\RR^m)$ captures the signatures relevant to our setting in a more precise way: if $1/2 \leq \gamma < \gamma + \varepsilon \leq 1$ for some small $\varepsilon>0$, and the weights $(\lambda_k)$ satisfy $\lambda_k^{1/(\gamma + \varepsilon)}/k^p \to 0$ and $k^p/\lambda_k^{1/\gamma} \to 0$, then
\[
S\big(C^{(\gamma + \varepsilon)\text{-Höl}}([0,T],\RR^m)\big) \subset \cT^{(p)}(\RR^m) \cap S\big(C^{\gamma\text{-Höl}}([0,T],\RR^m)\big) \subsetneq S\big(C^{\gamma\text{-Höl}}([0,T],\RR^m)\big).
\]
The first inclusion follows directly from Proposition \ref{lem_continuity_sig_map} (via an estimate of the type in \eqref{eq_sig_projective_continuity}). The strictness of the second inclusion requires constructing paths with sharp signature decay, namely paths $y$ such that $\| S(y)_{0,T} \|_{1,\lambda_k} \gtrsim \exp(c \lambda_k^{1/\gamma})$. Such lower bounds are nontrivial and rely on techniques from the literature on signature asymptotics \cite{boedihardjo2020,boedihardjo2023}. Since this is not needed here, we leave further exploration for future work. The key point is that $\cT^{(p)}(\RR^m)$ is designed to exclude signatures with sharp decay at a given regularity exponent $\gamma$. Ultimately, the choice of signature codomain depends on the intended application.
\end{remark}

\subsection{Dynamic Universal Approximation}
\label{subsec_universality_of_sig_cdes}

In this section we show that Sig-CDEs \eqref{eq_sig_cde_integral_form} form a universal subclass of path-dependent CDEs \eqref{eq_functional_cde}. More precisely, under the well-posedness assumptions of Theorem \ref{prop_EU_global_sol_pd_cde}, the solution of any path-dependent CDE can be approximated arbitrarily well by the solution of a suitably chosen Sig-CDE. This may be viewed as a \emph{dynamic universal approximation theorem}, highlighting the expressive power of signatures when used as driving vector fields in controlled differential equations.

The starting point is the classical universality of signatures (Corollary \ref{cor_classical_universality}). If the functional $f$ in \eqref{eq_functional_cde} can be approximated by a map of the form $\bell \circ S$, where $\bell$ is linear, then Theorem~\ref{prop_stability_pd_cdes} suggests that the corresponding solutions of \eqref{eq_functional_cde} and \eqref{eq_sig_cde_integral_form} should remain close. However, interpreting the signature as a collection of path monomials reveals an obstruction: linear functionals of the signature typically grow superlinearly (indeed polynomially) at the path level, violating the linear growth condition \eqref{eq_path_linear_growth_condition} required for well-posedness. To control this behavior we introduce an \emph{activation function} that regulates the growth of the signature linear functional.

Specifically, to offset the polynomial growth of arbitrary degree arising from the signature (see \eqref{eq_factorial_decay_signatures}), we consider an activation $\sigma:\RR\to\RR$ satisfying the logarithmic bound
\begin{align}
\label{eq_logarithmic_growth_activation}
    |\sigma(x)| \le C_\sigma\big(1+\ln(1+|x|)\big), \qquad x\in\RR .
\end{align}
Since $|\langle \bell, S(y|_{[0,t]})\rangle|\lesssim \exp(\|y|_{[0,t]}\|_{\beta\text{-Höl}})$ for $\bell\in T(\RR^m)$, it follows that functionals of the form $ \sigma \, \big(\langle \ell, S(y|_{[0,t]})\rangle\big)$ satisfy the required linear growth condition (see the proof below for details). Under mild additional assumptions on $\sigma$, Sig-CDEs built from such functionals can therefore approximate the solutions of any well-posed path-dependent CDE, yielding several universal subclasses parametrized by the choice of the activation function.

Finally, since our analysis already takes place on $\Lambda^{\alpha\text{-Höl}}_\beta([0,T_1],\RR^m)$ — a space endowed with a topology weaker than that induced by $d^{\Lambda}_{\alpha\text{-Höl}}$ — it is natural to work with weighted spaces. In this setting unbounded activations can be accommodated by fixing a suitable weight function (see Example \ref{ex_weighted_stopped_paths}). Given $\zeta>0$, we consider
\[
\hat{\psi}(\hat{y}|_{[0,t]}) =
\exp \big(\zeta(|\hat y_0|+\|\hat{y}|_{[0,t]}\|_{\alpha\text{-Höl}})\big), \qquad \hat y|_{[0,t]}\in \widehat{\Lambda}^{\alpha\text{-Höl}}_\beta([0,T_1],\RR^{m+1}).
\]
Recall from Remark \ref{rem_time_augmentation} that $\widehat{\Lambda}^{\alpha\text{-Höl}}_\beta([0,T_1],\RR^{m+1})$ is closed in $\Lambda^{\alpha\text{-Höl}}_\beta([0,T_1],\RR^{m+1})$ and therefore inherits its weighted structure. A technical lemma on time-augmented paths is recorded in Appendix \ref{Appendix_4}

\begin{theorem}%
\label{prop_dynamic_universality_activations}
   Fix $B_x,B_w>0$. Consider $x \in C^{1\text{-Höl}}([0,T_1],\RR^d)$ and $w|_{[0,T_0]} \in \Lambda^{\alpha\text{-Höl}}_\beta([0,T_1],\RR^m)$ satisfying
   \begin{align}
       \label{eq_fixed_bounded_initial_data_set}
       |x_0| + \|x \|_{1\text{-Höl};[0,T_1]} \leq B_x,\quad \mathrm{and}\quad  \| w|_{[0,T_0]} \|_{\alpha\text{-Höl}} \leq B_w,
   \end{align}
   with $w_0$ fixed. Assume $f : \Lambda^{\alpha\text{-Höl}}_\beta([0,T_1],\RR^m) \to L(\RR^d,\RR^m)$ is a continuous, non-anticipative functional satisfying \eqref{eq_loc_lip_2nd_variable} and \eqref{eq_path_linear_growth_condition}. Let $\sigma : \RR \to \RR$ be a Lipschitz activation that is strictly monotone and satisfies the logarithmic growth condition \eqref{eq_logarithmic_growth_activation}. 
    
    Then there exists a sequence of linear functionals $(\bell_n) \subset T(\RR^{m+1})^{m\times d}$ and constants $c_0 \in \RR$ and $c_1 > 0$ for which the unique solutions $y^{(n)}$ of the approximating systems
    \begin{align}
    \label{eq_approximating_sig_cdes}
        &y^{(n),i}_t 
        = w^{i}_{T_0} 
        + \sum_{j=1}^{d} \int_{T_0}^t \Big(  c_1 \sigma \big( \langle \bell^{ij}_n, S(\hat{y}^{(n)}|_{[0,s]}) \rangle \big) + c_0 \Big) \dd x^j_s, \\
        &y^{(n),i}|_{[0,T_0]} = w^i|_{[0,T_0]}, \quad i \in \{1,\hdots,m\}, \nonumber
    \end{align}
    converge in the $\beta$-Hölder topology to the unique solution of the path-dependent CDE
    \begin{align}
    \label{eq_target_path_dependent_cde}
        y^i_t 
        &= w^i_{T_0} 
        + \sum_{j=1}^d \int_{T_0}^t f^{ij}(y|_{[0,s]}) \dd x^j_s,  
    \quad y^i|_{[0,T_0]} = w^i|_{[0,T_0]}.
    \end{align}
    Moreover, the convergence is uniform over the bounded set of controls and initial data above (i.e. uniformly for all $(x,w|_{[0,T_0]})$ with the displayed $B_x,B_w$ bounds and $w_0$ fixed).
\end{theorem}
\begin{proof}
    Since $f$ satisfies the assumptions of Theorem \ref{prop_EU_global_sol_pd_cde}, the path-dependent CDE \eqref{eq_target_path_dependent_cde} admits a unique global solution $y$. We approximate each scalar functional $f^{ij}$ individually and then combine these approximations into matrix form. As componentwise convergence  implies convergence in any matrix norm, this reduction is sufficient.

    Fix $i \in \{1,\hdots,m\}$ and $j \in \{1,\hdots,d\}$. Following Lemma \ref{lem_aux_time_augmentation}, we lift $f^{ij}$ to the space of stopped time-augmented paths, obtaining $\hat{f}^{ij}$. By construction, $\hat{f}^{ij}(\hat{y}|_{[0,t]}) = f^{ij}(y|_{[0,t]})$, so $\hat{f}^{ij}$ inherits the well-posedness properties of $f^{ij}$. 

    Moreover, by Lemma \ref{lem_pasting_argument_bound} (with $\gamma = \alpha$) and the ensuing discussion, we know that $\{(t,y)\}_{t\in[0,T_1]}$ is contained in a compact set $K = [0,T_1] \times B^\alpha(R_f)$ for some sufficiently large $R_f > 0$. Consequently, we may apply the Tietze extension theorem to $\hat{f}^{ij}|_{K/_\sim}$, thereby obtaining a continuous and bounded extension to the entire space $\widehat{\Lambda}^{\alpha\text{-Höl}}_\beta([0,T_1],\RR^{m+1})$.\footnote{Strictly speaking, since $\hat{f}^{ij}$ is defined on stopped time-augmented paths, the compact set $K$ should be lifted to a corresponding $\hat{K} \subset \widehat{C}^{\alpha\text{-Höl}}([0,T_1],\RR^{m+1})$. But because $\hat{f}^{ij}$ ignores the time component, this distinction is immaterial. For simplicity, we therefore continue to write $K$, with the intended meaning clear from context.} The solution $y$ remains unchanged under this modification, since the dynamics coincide on $K$. In particular, the same compact set $K$ continues to contain $\{(t,y)\}_{t\in[0,T_1]}$ if $f^{ij}$ in \eqref{eq_target_path_dependent_cde} is replaced by this extension. For notational simplicity, we continue to write the extended functional as $\hat{f}^{ij}$.

    Next, since $\sigma$ is invertible on its image, we define
    \begin{align}
        \label{eq_defining_constants_c0c1}
        \tilde{f}^{ij} \coloneqq \sigma^{-1} \circ (c_1^{-1} \hat{f}^{ij} - c_0c_1^{-1}),
    \end{align}
    with $c_0 \in \RR$ and $c_1>0$ chosen so that the composition is well-defined. Observe that $K$ does not depend on the indices $i,j$, so choice of $c_0,c_1$ works uniformly for all $i,j$. More importantly, because $\hat{f}^{ij}$ is bounded and $\sigma$ is continuous, it follows that $\tilde{f}^{ij} \in \mathcal{B}_{\hat{\psi}}\big(\widehat{\Lambda}^{\alpha\text{-Höl}}_\beta([0,T_1],\RR^{m+1})\big)$. Theorem \ref{prop_global_approx_signatures} then guarantees that $\tilde{f}^{ij}$ can be approximated by linear functionals of the signature.

    Concretely, given any sequence $(\varepsilon_n) \subset (0,1)$ with $\varepsilon_n \to 0$, there exists a sequence of linear functionals $\bell_n^{ij}$ in the span of
    $$
        \Big\{\widehat{\Lambda}^{\alpha\text{-Höl}}_\beta([0,T_1],\RR^{m+1}) \ni \hat{z}|_{[0,t]} \mapsto \langle e_w , S(\hat{z}|_{[0,t]}) \rangle : w \in \{0,1,\cdots, m\}^N, N \in \NN_0 \Big\}
    $$
    such that, for all $n \geq 1$,
    \begin{align}
        \label{eq_f_tilde_weighted_approximation}
        \sup_{\hat{z}|_{[0,t]} \in \widehat{\Lambda}^{\alpha\text{-Höl}}_\beta} \frac{|\tilde{f}^{ij}(\hat{z}|_{[0,t]}) - \langle \bell^{ij}_n,S(\hat{z}|_{[0,t]})\rangle|}{\hat{\psi}(\hat{z}|_{[0,t]})} < \frac{\varepsilon_n}{c_1 L_\sigma},
    \end{align}    
    where $L_\sigma$ is the Lipschitz constant of $\sigma$. By Lipschitz continuity, this then implies
    \begin{align} \label{eq_weighted_convergence_universality}
        \sup_{\hat{z}|_{[0,t]} \in \widehat{\Lambda}^{\alpha\text{-Höl}}_\beta} \frac{|\hat{f}^{ij}(\hat{z}|_{[0,t]}) - \big(c_1 \sigma \big( \langle \bell^{ij}_n, S(\hat{z}|_{[0,t]}) \rangle \big) + c_0 \big) |}{\hat{\psi}(\hat{z}|_{[0,t]})} < \varepsilon_n,
    \end{align}
    which, by Lemma \ref{lem_aux_time_augmentation}, is equivalent to $\| f^{ij} - \big( c_1\sigma ( \langle \bell^{ij}_n, S(\ \widehat{\cdot}\ ) \rangle ) + c_0 \big) \|_{\cB_{\psi}} < \varepsilon_n$.

    At this point, we note that although the maps $\bell_n = [\bell_n^{ij}]_{i,j}$ may not directly satisfy the conditions required for existence and uniqueness, the functionals $c_1(\sigma \circ \bell_n^{ij})+c_0$ do. Indeed, since each $\bell^{ij}_n$ belongs to the dual of $T_1(\RR^{m+1})$ (or any other weighted $\ell_p$-sum tensor space) and $\sigma$ is globally Lipschitz, the functionals $c_1(\sigma \circ \bell_n^{ij})+c_0$ are globally Lipschitz on $T_1(\RR^{m+1})$. Furthermore, using the logarithmic growth of $\sigma$ together with \cite[Theorem 3.1.3]{Lyons2002}, one shows that 
    \begin{align}
    \label{eq_non_uniform_linear_growth_activation}
        |\sigma \big( \langle \bell^{ij}_n ,S(\hat{z}|_{[0,t]}) \rangle\big)| 
        &\lesssim \left( 1 + \ln\Big( 1 +\| \bell_n^{ij} \|_{op} \sum_{k = 0}^{N_n} \frac{\| \hat{z}|_{[0,t]} \|_{\beta\text{-Höl}}^k}{(\beta k)!}  \Big) \right) \nonumber \\
        &\lesssim \ln\left( 1 + \| \bell_n^{ij} \|_{op} \frac{N_n!}{(\beta N_n)!} \right)\Big(1 +  \| \hat{z}|_{[0,t]} \|_{\beta\text{-Höl}}\Big),
    \end{align}
    where $N_n$ denotes the minimal level such that $\bell^{ij}_n \in T^{N_n}(\RR^{m+1})$. Hence, the functionals $\big( c_1\sigma ( \langle \bell^{ij}_n, S(\ \widehat{\cdot}\ ) \rangle ) + c_0\big)_{n\geq 1}$ satisfy the required linear growth condition, and consequently, the approximating CDEs \eqref{eq_approximating_sig_cdes} each admit a unique global solution $y^{(n)} : [0,T_1] \to \RR^m$ — see Theorem \ref{prop_EU_global_sol_pd_cde} or Proposition \ref{prop_well_posedness_SigCDEs}. 

    That said, the \textit{a priori} bound for each $y^{(n)}$ (obtained via Lemma \ref{lem_pasting_argument_bound} with any $\gamma$ and the estimate \eqref{eq_non_uniform_linear_growth_activation}, as in \eqref{eq_a_priori_bound_first_estimate} and \eqref{eq_local_to_global_lin_growth_bound}) may, in principle, diverge as $n \to \infty$. The issue arises from potential growth in the operator norms $\|\bell^{ij}_n\|_{op}$, which could cause the stability estimate \eqref{eq_stability_bound} — our main tool for proving convergence — to blow up and thus prevent the desired limit.

    To overcome this, we note that \eqref{eq_f_tilde_weighted_approximation} yields the bound
    $$
        \| \bell_n^{ij} \circ S \|_{\cB_{\hat{\psi}}} \leq \| \tilde{f}^{ij} \|_{\cB_{\hat{\psi}}} + \| \tilde{f}^{ij} - \bell_n^{ij} \circ S \|_{\cB_{\hat{\psi}}} \lesssim  \| \tilde{f}^{ij} \|_{\cB_{\hat{\psi}}} + \varepsilon_n.
    $$
    Combining this with the logarithmic growth condition \eqref{eq_logarithmic_growth_activation} for $\sigma$, we deduce that for all $n \geq 1$,
    \begin{align}
    \label{eq_uniform_linear_growth_activation}
         |\sigma \big( \langle \bell^{ij}_n ,S(\hat{z}|_{[0,t]}) \rangle\big)| %
         &\lesssim \Big( 1 + \log\Big( 1 + \| \bell^{ij}_n \circ S \|_{\cB_{\hat{\psi}}} \hat{\psi}(\hat{z}|_{[0,t]})  \Big) \Big) \nonumber \\
         &\lesssim \log(\| \tilde{f}^{ij} \|_{\cB_{\hat{\psi}}} + 1) \Big(1 +  \zeta ( |\hat{z}_0| + \| \hat{z}|_{[0,t]} \|_{\alpha\text{-Höl}}) \Big).
    \end{align}
    Hence, by \eqref{eq_a_priori_bound_first_estimate} and \eqref{eq_local_to_global_lin_growth_bound} together with Lemma \ref{lem_pasting_argument_bound} ($\gamma = \alpha$), we obtain an \textit{a priori} bound on the $\alpha$–Hölder seminorm of $y^{(n)}$ that holds uniformly in $n$.

    Finally, to apply the stability result, choose $R_{\text{Sig}} > 0$ large enough so that all $\{(t,y^{(n)})\}_{t \in [0,T_1]}$ lie in $[0,T_1] \times B^\alpha(R_\text{Sig})$. We emphasize that $R_{\text{Sig}}$ is independent of $n$ thanks to \eqref{eq_uniform_linear_growth_activation}. Setting $R \coloneqq R_f \lor R_{\text{Sig}}$ and $\mathcal{K} \coloneqq [0,T_1] \times B^\alpha(R)$, we ensure that both the target solution trajectory and its approximations remain inside $\mathcal{K}/_\sim$. Moreover, by \eqref{eq_weighted_convergence_universality}, we have that 
    $$
        \| f^{ij}|_{\mathcal{K}/_\sim} - \big( c_1\sigma ( \langle \bell^{ij}_n, S(\ \widehat{\cdot}\ ) \rangle ) + c_0 \big)|_{\mathcal{K}/_\sim} \|_{\infty} \leq \tilde{\varepsilon}_n \coloneqq \sup_{z|_{[0,t]} \in \mathcal{K}/_\sim} \psi(z|_{[0,t]}) \varepsilon_n \to 0.
    $$
    We now assemble the scalar maps $f^{ij}$ and their respective functionals $\bell^{ij}_n$ into matrix form. Applying Theorem \ref{prop_stability_pd_cdes}, we then obtain, up to a constant depending on $d$ and $m$,
    \begin{align}
    \label{eq_final_convergence_estimate}
        \| y - y^{(n)} \|_{\beta\text{-Höl};[T_0,T_1]} \lesssim 2^{N} \Big( 1 + \frac{1}{L_\mathcal{K}}\Big) \tilde{\varepsilon}_n,  \quad \text{where} \ N =  \Big\lceil T_1 (2 L_\mathcal{K} B_x)^{1/(1-\beta)} \Big\rceil
    \end{align}
    and $L_{\mathcal{K}} > 0$ is the Lipschitz constant of $f$ on $\mathcal{K}$. Letting $n \to \infty$ yields the desired convergence. Crucially, the explicit estimate above ensures that this convergence holds uniformly over any fixed bounded family of controls and initial conditions, with the corresponding dependence absorbed into the constants $c_0, c_1$ and the compact set $K$. This completes the proof.
\end{proof}

To reiterate why the $\mathcal{B}_{\hat{\psi}} $-norm is a natural approximation tool, we emphasize that the compact set $\mathcal{K}$ in the proof above arises only after first approximating $f$ globally over $\Lambda^{\suphol{\alpha}}_\beta$. To understand why such a global approximation is required, consider instead relying on the classical universality of signatures (Corollary \ref{cor_classical_universality}) to approximate $f$ uniformly on a fixed compact set of paths. The difficulty with this approach lies in the potential circular dependence between the compact set and the approximating solutions. Specifically, the relevant compact set must contain both the target solution $y$ and the approximations $y^{(n)}$, yet it cannot itself depend on the sequence $(y^{(n)})$. Without further assumptions, avoiding this circularity is far from straightforward.

In particular, this consideration is closely tied to the choice of activation $\sigma$. Indeed, if $\sigma$ is bounded, then — as clarified in the next result — we can deduce \textit{a priori} a compact set containing all paths $y$ and $y^{(n)}$ by standard arguments, without the need for weighted spaces. Such a conclusion is no longer available when $\sigma$ is unbounded.

\begin{corollary}
    Under the assumptions of Theorem \ref{prop_dynamic_universality_activations}, suppose in addition that the activation $\sigma$ is bounded. Then there exist linear functionals $(\bell_n) \subset T(\RR^{m+1})^{m\times d}$ and constants $c_0 \in \RR$, $c_1 > 0$ such that the corresponding approximating solutions $y^{(n)}$ of \eqref{eq_approximating_sig_cdes} converge in the $\beta$–Hölder topology to the unique solution of \eqref{eq_target_path_dependent_cde}. Moreover, the convergence is uniform over bounded sets of initial data as in \eqref{eq_fixed_bounded_initial_data_set}.
\end{corollary}
\begin{proof}
    The claim follows immediately from Theorem \ref{prop_dynamic_universality_activations}. Nevertheless, it is instructive to see that, in this bounded setting, the argument simplifies considerably, as weighted spaces are no longer required.

    First, applying the same notation and the cut-off/extension argument used in the proof of Theorem \ref{prop_dynamic_universality_activations}, we obtain constants $c_0 \in \RR$ and $c_1>0$ such that
    \[
    \sigma^{-1} \circ \big(c_1^{-1}\hat f^{ij} - c_0 c_1^{-1}\big)
    \]
    defines a bounded and continuous path functional. Recall that $\hat f^{ij}$ is itself bounded (by construction) and yields the same solution $y$ as the original functional in \eqref{eq_target_path_dependent_cde}. 

    Next, we derive an \textit{a priori} compact domain that contains both the solution $y$ of the target CDE and the solution $z$ of any Sig–CDE of the form \eqref{eq_approximating_sig_cdes}, independently of the choice of linear functional $\bell \in T(\mathbb{R}^{m+1})^{m\times d}$. Indeed, for all $T_0 \leq s < t \leq T_1$ and each component $i \in \{1,\hdots,m\}$,
    $$
        |y^i_{s,t}| \leq \sum_{j=1}^d \int_s^t |\hat{f}^{ij}(y|_{[0,u]})|\,|\dd x^j_u| \lesssim \|\hat{f}\|_\infty \|x\|_{1\text{-Höl};[0,T_1]} |t-s|,
    $$
    and so, up to a constant depending only on $d$,$m$ and $B_w$,
    $$
        \|y\|_{\alpha\text{-Höl};[0,T_1]} \lesssim B_x \big(1 + \|\hat{f}\|_\infty T_1^{1-\alpha}\big).
    $$
    
    Now, consider any $\bell \in T(\mathbb{R}^{m+1})^{m\times d}$, and let $z$ be the solution of the associated Sig–CDE. Using the boundedness of $\sigma$ and arguing as above,
    $$
    |z^i_{s,t}| \lesssim (|c_1| \lor |c_0|) \|\sigma\|_\infty \|x\|_{1\text{-Höl};[0,T_1]} |t-s|.
    $$
    Hence, up to a universal constant,
    $$
    \|z\|_{\alpha\text{-Höl};[0,T_1]} \lesssim B_x \big(1 + (|c_1| \lor |c_0|) \|\sigma\|_\infty T_1^{1-\alpha}\big).
    $$
    Importantly, we emphasize this bound holds for any solution of \eqref{eq_approximating_sig_cdes}, regardless of the particular choice of $\bell$; it is therefore independent of the approximating sequence to be introduced below.
    
    From these uniform bounds, we set
    $$
    r \coloneqq \max \big\{ (|c_1| \lor |c_0|) \|\sigma\|_\infty, \|\hat{f}\|_\infty \big\},
    \qquad
    R \propto B_x(1 + rT_1^{1-\alpha}),
    $$
    and we find that both trajectories $\{(t,y)\}_{t\in[0,T_1]}$ and $\{(t,z)\}_{t\in[0,T_1]}$ lie in the compact set $\mathcal{K} = [0,T_1] \times B^\alpha(R)$ (as before, $R$ is defined up to a constant dependent on $d,m$ and $B_w$). This compact domain is determined before any approximation of $f$ is performed.
    
    Having fixed $\mathcal{K}$, we use the Lipschitz assumption on $\sigma$ and Corollary \ref{cor_classical_universality} to approximate $f$ uniformly on $\mathcal{K}/_\sim$ by functionals of the form $c_1 \sigma \big(\langle \bell_n^{ij}, S(\, \widehat{\cdot} \,) \rangle\big) + c_0$. For each $n$, the corresponding Sig–CDE solution $\{(t,y^{(n)}) \}_{t\in [0,T_1]}$ remains in $\mathcal{K}$. Convergence of $y^{(n)} \to y$ in the $\beta-$Hölder topology then follows directly from Theorem \ref{prop_stability_pd_cdes}, as in the general case.
\end{proof}

We conclude this section by illustrating how weighted spaces allow us to obtain global approximation on the control level, even when $\sigma$ is not bounded. In particular, we show that, under mild additional assumptions, the convergence in Theorem \ref{prop_dynamic_universality_activations} holds for all controls $x$ simultaneously, i.e., the same sequence of signature linear functionals yields the approximation irrespective of $x$ once a suitable weighted control space is introduced. For this we require the following auxiliary result.

\begin{proposition}
\label{lem_weaker_continuity_controls}
    Let $y \equiv y(x)$ denote the solution to \eqref{eq_target_path_dependent_cde} driven by the control $x$ and assume $f$ is Lipschitz with respect to $d^\Lambda_{\beta\text{-Höl};[0,T_1]}$. Then there exists $\eta < 1$ such that the solution map 
    \[C^{1\text{-Höl}}_\eta([0,T_1],\RR^d) \ni x \mapsto y(x) \in C^{\beta\text{-Höl}}([0,T_1],\RR^m)\]
    is continuous. 
\end{proposition}
\begin{remark}
    We note that the estimates in Theorem \ref{prop_stability_pd_cdes} are not sufficient here, as they involve the Lipschitz seminorm on the control level, while we now require continuity with respect to the weaker $\eta-$Hölder seminorm. This naturally suggests using Young’s inequality \cite[Section 4.1]{friz2014course}, which in turn requires estimating the Hölder regularity of the map $u \mapsto f(y|_{[0,u]})$ for a generic path $y \in C^{\alpha\text{-Höl}}_\beta([0,T_1],\RR^m)$.
\end{remark}
\begin{proof}
    See Appendix \ref{Appendix_4}.    
\end{proof}

\begin{theorem}
\label{thm_global_approx_control_level}
    Under the assumptions of Theorem \ref{prop_dynamic_universality_activations}, suppose in addition that $f$ is Lipschitz with respect to $d^\Lambda_{\beta\text{-Höl};[0,T_1]}$. Then there exists a weighted space $\big( C^{1\text{-Höl}}_\eta([0,T_1],\RR^d), \psi_c\big)$ on which the approximation of Theorem \ref{prop_dynamic_universality_activations} holds globally. 
    
    Specifically, for every $\delta > 0$ there exist $\bell \in T(\RR^{m+1})^{m\times d}$ and constants $c_0\in \RR,c_1 >0$ such that the associated Sig-CDE solution $y^{\mathrm{Sig}}$ satisfies
    \[\sup_{x \in C^{1\text{-Höl}}_\eta([0,T_1],\RR^d)} \psi^{-1}_c(x) \, \| y(x) - y^{\mathrm{Sig}}(x)\|_{\beta\text{-Höl};[0,T_1]} < \delta.\]
    In particular, if both $f$ and $\sigma$ are bounded, one may take $\psi_c(x) = \eta_c(|x_0| + \|x\|_{1\text{-Höl};[0,T_1]})$ with $\eta_c : [0,\infty) \to (0,\infty)$ continuous, increasing, and satisfying any superlinear growth condition.
\end{theorem}
\begin{remark}
    The proof hinges on understanding how both the target solution $y$ and its approximation $y^{\mathrm{Sig}}$ depend on $\|x\|_{1\text{-Höl};[0,T_1]}$. Under the general linear growth assumption on $f$, the path $y^{\mathrm{Sig}}$ satisfies a Sig-CDE of the form \eqref{eq_approximating_sig_cdes}, where the coefficients $c_0$ and $c_1$ also depend on $\|x\|_{1\text{-Höl};[0,T_1]}$; this dependence disappears if $f$ is bounded. Once this relationship with the control norm is made explicit, the approximation $y^{\mathrm{Sig}}(x) \approx y(x)$ follows directly from Theorem \ref{prop_dynamic_universality_activations} for controls in a suitably chosen compact subset of $C^{1\text{-Höl}}_\eta([0,T_1],\RR^d)$, while outside this compact the approximation error can be absorbed through an appropriate admissible weight function $\psi_c$.
\end{remark}

\begin{proof}
    Fix $\delta > 0$ and set $B_x = k_0/\delta$ for some $k_0 > 0$, where $B_x$ plays the same role as in Theorem \ref{prop_dynamic_universality_activations}. By Lemma \ref{lem_pasting_argument_bound} (with $\gamma = \alpha$) and the ensuing discussion in Section \ref{sec_path_dependent_cdes}, there exist constants $k_1,k_2 \geq 0$ such that, for all $x \in C^{1\text{-Höl}}([0,T_1],\RR^d)$,\footnote{For notational simplicity, we only make explicit the dependence on $\|x\|_{1\text{-Höl};[0,T_1]}$.}
    \[\| y(x) \|_{\alpha\text{-Höl};[0,T_1]} \leq  k_1 \exp \Big( k_2 \|x\|_{1\text{-Höl};[0,T_1]}^{1/(1-\alpha)} \Big) \eqqcolon R_f(\|x\|_{1\text{-Höl};[0,T_1]}).\]
    
    Assume now that $\sigma^{-1}$ is defined over an interval $[a,b]$ with $a < b$ (allowing infinite endpoints). By the linear growth of $f$, the function in \eqref{eq_defining_constants_c0c1} is well-defined if and only if
    \begin{align}
    \label{eq_dependencies_c1_c0}
        c_1 \geq \frac{2}{b-a} B_f  \quad \text{and} \quad c_0 \in \Big[B_f - bc_1,-B_f-ac_1 \Big],
    \end{align}
    where $B_f \coloneqq C_f\big(1 + B_w + T_1^{\alpha-\beta} R_f(B_x) \big)$. When $a = -\infty$, the right endpoint  becomes vacuous; analogously, if $b = +\infty$, the left endpoint is automatically satisfied. In either case, \eqref{eq_dependencies_c1_c0} shows that there exist constants $c_0,c_1$ ensuring that \eqref{eq_defining_constants_c0c1} is well-defined and satisfying $|c_0|, |c_1| \leq k_3 \big(1 + R_f(B_x)\big)$, for some $k_3 > 0$ independent of $B_x$.

    For such (fixed) constants $c_0,c_1$ and any functional $\bell \in T(\RR^{m+1})$, we now deduce an \textit{a priori} bound for the unique solution $y^{\mathrm{Sig}}$ of the Sig-CDE \eqref{eq_approximating_sig_cdes}. Clearly, for all $T_0 \leq \sigma \leq  s \leq t \leq \tau \leq T_1$, 
    \[|y^{\mathrm{Sig}}_{s,t}| \leq \big( 1 + |c_1|\big) \int_s^t \sum_{i,j} \Big( \big|\sigma\big(\langle \bell^{ij}, S(\hat{y}^{\mathrm{Sig}}|_{[0,u]})\rangle \big) \big| + |c_0| \Big) \cdot |\dd x_u|.\]
    Hence, by the same reasoning as in \eqref{eq_non_uniform_linear_growth_activation} and using Lemma \ref{lem_aux_time_augmentation} to deal with time-augmentation, we obtain up to a harmless  constant (which depends inclusively on $\bell$),  
    \[|y^{\mathrm{Sig}}_{s,t}| \lesssim \big( 1 + |c_1|\big) \Big( 1 + |c_0| + T_1^{\alpha-\beta} \|y^{\mathrm{Sig}}_{\cdot \land t}\|_{\alpha\text{-Höl};[0,T_1]} \Big) \|x\|_{1\text{-Höl};[0,T_1]} |t-s|.\]
    Dividing by $|t-s|^\alpha$ on both sides and taking the supremum over $\sigma \leq  s \leq t \leq \tau $, leads to 
    \[\|y^{\mathrm{Sig}}\|_{\alpha\text{-Höl};[\sigma,\tau]} \lesssim \big( 1 + |c_1|\big) \Big( 1 + |c_0| + \|y^{\mathrm{Sig}}\|_{\alpha\text{-Höl};[0,\tau]} \Big) \|x\|_{1\text{-Höl};[0,T_1]} |\tau-\sigma|^{1-\alpha}, \]
    and so, by Lemma \ref{lem_pasting_argument_bound} with $\gamma = \alpha$, there exist constants $k_4,k_5 \geq 0$ such that
    \begin{align*}
        \| y^{\mathrm{Sig}}(x) \|_{\alpha\text{-Höl};[T_0,T_1]} &\leq k_4 \big(1 + |c_0|\big) \exp\Big( k_5 \big( (1 + |c_1|) \|x\|_{1\text{-Höl};[0,T_1]} \big)^{1/(1-\alpha)} \Big).
    \end{align*}
    for all $x\in C^{1\text{-Höl}}([0,T_1],\RR^d)$. In particular, by our choice of constants $c_0,c_1$, we obtain
    \[\| y^{\mathrm{Sig}}(x) \|_{\alpha\text{-Höl};[T_0,T_1]} \leq k_6 \big(1 + R_f(B_x)\big) \exp\Big(k_7 \big( (1 + R_f(B_x)) \|x\|_{1\text{-Höl};[0,T_1]} \big)^{1/(1-\alpha)}\Big),\]
    for some constants $k_6,k_7 \geq 0$. Denote the expression on the right by $R_{\mathrm{Sig}}(\|x\|_{1\text{-Höl};[0,T_1]})$.

    Next, define $\psi_c(x) \coloneqq \eta_c(|x_0| + \|x\|_{1\text{-Höl};[0,T_1]})$ where $\eta_c:[0,\infty) \to (0,\infty)$ is continuous, increasing, and satisfies 
    \[\eta_c(\|x\|_{1\text{-Höl};[0,T_1]}) > R_f(\|x\|_{1\text{-Höl};[0,T_1]}) \lor R_{\mathrm{Sig}}(\|x\|_{1\text{-Höl};[0,T_1]}).\] 
    It is straightforward to verify that $\psi_c$ is admissible in the sense of Section \ref{subsec_weighted_spaces_global_approximations}, and that $\eta_c$ can be chosen so that both $y(x)$ and $y^{\mathrm{Sig}}(x)$ satisfy \eqref{eq_vanishing_tail_condition}. In turn, Proposition \ref{lem_weaker_continuity_controls} and Lemma \ref{lem_characterising_B_psi}, guarantee that $y(x), y^{\mathrm{Sig}}(x) \in \cB_{\psi_c}$. Set 
    \begin{align*}
        M \coloneqq \Big( \inf_{x \in C^{1\text{-Höl}}_\eta} \psi_c(x) \Big)^{-1} , \quad R(B_x) \coloneqq R_f(B_x) \lor R_{\mathrm{Sig}}(B_x), 
    \end{align*}
    and define\footnote{Note that our notation here matches the one in the proof of Theorem \ref{prop_dynamic_universality_activations}. Namely, $R_f$ in Theorem \ref{prop_dynamic_universality_activations} can be taken to be $R_f(B_x)$, only now with the dependence on $B_x$ made explicit; similarly for $R_{\mathrm{Sig}}(B_x)$ and $R(B_x)$.} 
    \begin{align}
    \label{eq_epsilon_global_approx}
        \varepsilon \coloneqq \left[ M \exp \Big( \log(2) \Big\lceil T_1 (2 L_f B_x)^{1/(1-\beta)} \Big\rceil + \zeta R(B_x) \Big) \Big(1 + \frac{1}{L_f}\Big)\right]^{-1} \frac{\delta}{2},
    \end{align}
    where $L_f$ denotes the Lipschitz constant of $f$. By \eqref{eq_epsilon_global_approx} and Theorem \ref{prop_dynamic_universality_activations}  (namely, the estimate \eqref{eq_final_convergence_estimate}), there exists $\bell \in T(\RR^{m+1})^{m\times d}$ such that, for all $x$ satisfying \eqref{eq_fixed_bounded_initial_data_set}, we have 
    \[\|y(x) - y^{\mathrm{Sig}}(x)\|_{\beta\text{-Höl};[T_0,T_1]} \leq M^{-1} \frac{\delta}{2}.\]
    Observe that the set of controls in \eqref{eq_fixed_bounded_initial_data_set} is compact in the $\eta-$topology; denote it by $K_{B_x}$. Then, 
    \begin{align}
    \label{eq_inside_outside_compact_est}
        \|y - y^{\mathrm{Sig}}\|_{\cB_{\psi_c}} &= \sup_{x \in K_{B_x}} \frac{\| y(x) - y^{\mathrm{Sig}}(x) \|_{\beta\text{-Höl};[0,T_1]}}{\psi_c(x)} + \sup_{x \in K^c_{B_x}} \frac{\| y(x) - y^{\mathrm{Sig}}(x) \|_{\beta\text{-Höl};[0,T_1]}}{\psi_c(x)} \nonumber \\
        &\leq \frac{\delta}{2} + \sup_{x \in K^c_{B_x}} \frac{\| y(x)\|_{\beta\text{-Höl};[0,T_1]}}{\psi_c(x)} + \sup_{x \in K^c_{B_x}} \frac{\| y^{\mathrm{Sig}}(x) \|_{\beta\text{-Höl};[0,T_1]}}{\psi_c(x)},
    \end{align}
    and it remains to control the approximation error outside of $K_{B_x}$. This, however, is immediate from the choice of $\psi_c$: both $\| y(x)\|_{\beta\text{-Höl};[0,T_1]} \psi_c^{-1}(x)$ and $\| y^{\mathrm{Sig}}(x)\|_{\beta\text{-Höl};[0,T_1]} \psi_c^{-1}(x)$ decrease as $\|x\|_{1\text{-Höl};[0,T_1]} \to \infty$. Hence, 
     \begin{align*}
        \|y - y^{\mathrm{Sig}}\|_{\cB_{\psi_c}} \leq \frac{\delta}{2} + \frac{R_f(B_x)}{\eta_c(B_x)} + \frac{R_{\mathrm{Sig}}(B_x)}{\eta_c(B_x)} <  \delta,
    \end{align*}
   once $k_0 > 0$ is chosen appropriately. None of the constants $k_1,k_2,k_6,k_7$ depend on  $B_x$; thus $k_0$ may be fixed \textit{a posteriori} with no circularity. For example, replacing $k_1,k_2,k_6,k_7$ by $k\coloneqq \max\{k_1,k_2,k_6,k_7\}$ does not affect the argument. In which case, we may take
    \[\eta_c(t) = k \big(1 + R_f(t)\big) \exp\Big((k+1) \big( (1 + R_f(t))t \big)^{\frac{1}{1-\alpha}}\Big), \quad k_0 = k^2 + 1,\]
    so that, 
    \[\frac{R_f(B_x)}{\eta_c(B_x)} < \frac{R_{\mathrm{Sig}}(B_x)}{\eta_c(B_x)} = \exp\Big(-\big( \big(1 + R_f(B_x)\big) B_x\big)^{1/(1-\alpha)} \Big) < \frac{\delta}{4}.\]
    
    Finally, it is straightforward to check that if $f$ is bounded, then Lemma \ref{lem_pasting_argument_bound} is unnecessary and $R_f(\, \cdot \,)$ depends linearly on $\|x\|_{1\text{-Höl};[0,T_1]}$; the same holds true for $R_{\mathrm{Sig}}(\, \cdot \,)$ if $\sigma$ is bounded. In this bounded setting, the constants $c_0,c_1$ become independent of $x$, and $\eta_c$ may be chosen with any superlinear growth. Moreover, when either $a = -\infty$ or $b = + \infty$, the constraints  also simplify: one may take $c_1 = 1$, and $c_0$ only enters linearly in $R_{\mathrm{Sig}}(\, \cdot \,)$.  
\end{proof}

\subsection{Entire Functionals and Sig-CDEs}
\label{subsec_entire_functionals}

The previous section established that, under mild assumptions, Sig-CDEs form a universal class of equations within the broader framework of path-dependent CDEs. Moreover, this universality can already be realized with signature functionals of a particularly simple form: linear functionals composed with an activation function $\sigma$. In what follows, we specialize further by taking $\sigma$ to be a bounded entire function. In this way, we are led to the notion of \emph{bounded entire maps on group-like elements}, as introduced in \cite{cuchiero2023signature}, which provide a concrete and tractable universal model for path-dependent dynamics. 

To illustrate, let $\sigma: \CC \to \CC$ be an entire function such that $\sigma(\RR) \subset \RR$ and $\sigma$ is bounded on $\RR$. Its power series expansion reads
\begin{align}
    \label{eq_power_series_expansion_sigma}
    \sigma(z) = \sum_{k=0}^\infty c_k z^k, \quad \text{for all} \ z \in \CC. 
\end{align}
For any $\bell \in T(\RR^m)$, the shuffle property \eqref{eq_shuffle_property} gives
$$
    \sigma(\langle \bell, S(y)_{0,t} \rangle ) = \sum_{k=0}^\infty c_k (\langle \bell, S(y)_{0,t} \rangle)^k = \sum_{k=0}^\infty c_k \langle \bell^{\shuffle k} , S(y)_{0,t} \rangle,
$$
for all $y|_{[0,t]} \in \Lambda^{\suphol{\alpha}}_\beta ([0,T],\RR^m)$, where $\bell^{\shuffle k}$ denotes the $k$-fold shuffle product of $\bell$ with itself. Thus, $\sigma \circ \bell$ acts linearly on signatures and, as we now make precise, is an instance of an \emph{entire map on group-like elements}.

Recall the set $G(\RR^m)$ of group-like elements \eqref{eq_group_like_elements}. Following \cite{cuchiero2023global}, for $\mathbf{g} \in G(\RR^m)$ and $\bell \in T((\RR^m))$ define
$$
    |\bell|_\mathbf{g} \coloneqq \sum_{k=1}^\infty \Big| \sum_{|w| = k} \ell_w g_w \Big|,
$$
and $G(\mathbb{R}^m)^* \coloneqq \big\{ \bell \in T((\RR^m)) : |\bell|_\mathbf{g} < \infty \ \text{for all} \ \mathbf{g} \in G(\RR^m) \big\}$. For $\bell \in G(\RR^m)^*$ and $\mathbf{g} \in G(\RR^m)$, extend the canonical pairing by
$$
    \langle \bell, \mathbf{g} \rangle \coloneqq \lim_{N\to \infty} \sum_{n=0}^N \langle \pi_n(\bell) , \pi_n(\mathbf{g}) \rangle \equiv \sum_{n=0}^\infty \sum_{w \in \cW_n} \ell_w g_w.
$$
Then, the set of \emph{entire maps on group-like elements} \cite[Definition 4.5]{cuchiero2023signature} is 
$$
    \mathcal{G} \coloneqq \big\{ \bell : G(\RR^m)\to \RR : \mathbf{g} \mapsto \langle \bell, \mathbf{g} \rangle \ \text{and} \ \bell \in G(\RR^m)^* \big\}. 
$$

We now show that the composition of an entire function $\sigma$ with a map $\bell \in \mathcal{G}$ remains in $\mathcal{G}$. Consequently, Theorem \ref{prop_dynamic_universality_activations} can be instantiated for the subclass of bounded entire maps on group-like elements.

\begin{lemma}
\label{lem_entire_plus_entire_is_entire}
    Let $\bell: G(\RR^m) \to \RR$, $\mathbf{g} \mapsto \langle \bell, \mathbf{g} \rangle$ with $\bell \in G(\RR^m)^*$ be an entire map on group-like elements, and let $\sigma: \CC \to \CC$ be an entire function with power series expansion \eqref{eq_power_series_expansion_sigma} such that $\sigma(\RR) \subset \RR$. Then the composition
    $$
        (\sigma \circ \bell)(\mathbf{g}) = \sum_{n=0}^\infty c_n \langle \bell, \mathbf{g} \rangle^n
    $$
    is again an entire map on group-like elements. Moreover, there exists $\bell_\sigma \in G(\RR^m)^*$ such that $(\sigma \circ \bell)(\mathbf{g}) = \langle \bell_\sigma, \mathbf{g} \rangle$. Specifically,
    \begin{align}
        \label{eq_formal_series_entire_map}
        \bell_\sigma = \sum_{k=0}^\infty \sum_{w\in \cW_k} \left( \sum_{n=0}^\infty c_n (\ell^{\shuffle n})_w \right) e_w.
    \end{align}
\end{lemma}
\begin{proof}
    Set $\tilde{\bell} \coloneqq \bell - \ell_\varnothing e_\varnothing$ so that $\tilde{\ell}_\varnothing = 0$. Since $\mathbf{g} \in G(\RR^m)$ is group-like, we have $g_\varnothing = 1$, and hence $\langle \bell, \mathbf{g} \rangle = \ell_\varnothing + \langle \tilde{\bell} ,\mathbf{g} \rangle$. Define $\tau : \CC \to \CC$ by $\tau(z) \coloneqq \sigma(\ell_\varnothing + z)$. Then $\tau$ is entire, with power series expansion
    \[
    \tau(z) = \sum_{n=0}^\infty d_n z^n, \quad d_n = \frac{\sigma^{(n)}(\ell_\varnothing)}{n!}.
    \]
    Using the shuffle property \eqref{eq_shuffle_property}, we obtain
    \[
    (\sigma \circ \bell)(\mathbf{g}) = \tau(\langle \tilde{\bell}, \mathbf{g} \rangle ) = \sum_{n=0}^\infty d_n \langle \tilde{\bell} ,\mathbf{g} \rangle^n = \sum_{n=0}^\infty d_n \langle \tilde{\bell}^{\shuffle n}, \mathbf{g} \rangle.
    \]
    This series converges absolutely. Indeed, by \cite[Lemma 4.3 (ii)]{cuchiero2023signature}, $|\bell^{\shuffle n}|_\mathbf{g} \leq |\bell|_\mathbf{g}^n$, and therefore
    \[
    \sum_{n=0}^\infty |d_n| |\langle \tilde{\bell}^{\shuffle n}, \mathbf{g} \rangle| \leq  \sum_{n=0}^\infty |d_n| |\tilde{\bell}|^n_{\mathbf{g}} < \infty,
    \]
    since entire functions have power series that converge absolutely everywhere \cite[Theorem 4.2.5]{asmar2018complex}. Now, define the formal series 
    \[
    \tilde{\bell}_\sigma \coloneqq \sum_{k=0}^\infty \sum_{w \in \cW_k} \left( \sum_{n=0}^k d_n \big( \tilde{\ell}^{\shuffle n} \big)_w \right)e_w.
    \]
    Because $(\tilde{\ell}^{\shuffle n})_w = 0$ whenever $n > |w|$, we may write
    \[ \langle \tilde{\bell}_\sigma , \mathbf{g} \rangle = \sum_{n=0}^\infty d_n \sum_{k=0}^\infty \sum_{w\in \cW_k} \big(\tilde{\ell}^{\shuffle n} \big)_w g_w = \sum_{n=0}^\infty d_n \langle \tilde{\bell}^{\shuffle n}, \mathbf{g} \rangle = (\sigma \circ \bell) (\mathbf{g}).\]
    Finally, for each word $w$ of length $|w| =k$, we have
    \[
    \sum_{n=0}^\infty c_n (\ell^{\shuffle n})_w = \sum_{n=0}^\infty c_n \sum_{j=0}^k \binom{n}{j} \ell_\varnothing^{n-j} \big( \tilde{\ell}^{\shuffle j} \big)_w = \sum_{j=0}^k \left( \sum_{n=j}^\infty c_n \binom{n}{j} \ell_\varnothing^{n-j} \right) \big( \tilde{\ell}^{\shuffle j} \big)_w = \sum_{j=0}^k d_j \big( \tilde{\ell}^{\shuffle j} \big)_w.
    \]
    Thus $\tilde{\bell}_\sigma$ coincides with $\bell_\sigma$ in \eqref{eq_formal_series_entire_map} which proves the claim.
\end{proof}

\begin{corollary}
    In addition to the assumptions of Theorem \ref{prop_dynamic_universality_activations}, assume that $\sigma$ is an entire map such that $\sigma(\RR) \subset \RR$ and $\sigma$ is bounded on $\RR$. Then there exists a sequence $(\bell_n) \subset (G(\RR^{m+1})^*)^{m\times d}$ of bounded entire maps on group-like elements such that the solutions of
    \begin{align}
    \label{eq_approximating_sig_cdes_entire}
        y^{(n),i}_t 
        &= w^{i}_{T_0} 
        + \sum_{j=1}^{d} \int_{T_0}^t  \bell^{ij}_n \big( S(\hat{y}^{(n)}|_{[0,s]}) \big)  \dd x^j_s,  
        \quad y^{(n),i}|_{[0,T_0]} = w^i|_{[0,T_0]},
    \end{align}
    for $i \in \{1,\hdots,m\}$, converge in the $\beta$-Hölder topology to the solution of 
    \begin{align}
    \label{eq_target_path_dependent_cde_entire}
        y^i_t 
        &= w^i_{T_0} 
        + \sum_{j=1}^d \int_{T_0}^t f^{ij}(y|_{[0,s]}) \dd x^j_s,  
        \quad y^i|_{[0,T_0]} = w^i|_{[0,T_0]}.
    \end{align}
    As before, the convergence holds uniformly over the sets \eqref{eq_fixed_bounded_initial_data_set}.
\end{corollary}
\begin{remark}
    Strictly speaking, each $\bell^{ij}_n$ is in $T((\RR^{m+1}))$, and only the restriction $\bell^{ij}_n|_{G(\RR^{m+1})}$ defines an entire map on group-like elements. Recall that the well-posedness of \eqref{eq_approximating_sig_cdes_entire} requires the Lipschitz properties of the signature map and the functionals $\bell^{ij}_n$ on $\cT^{(p)}(\RR^{m+1})$ (or on any other (weighted) $\ell_p-$sum tensor space, by Corollary \ref{cor_signature_loc_lip}). Nevertheless, the vector fields in \eqref{eq_approximating_sig_cdes_entire} are only evaluated at group-like elements.
\end{remark}
\begin{proof}
    The claim follows directly from Theorem \ref{prop_dynamic_universality_activations} and Lemma \ref{lem_entire_plus_entire_is_entire}.
\end{proof}

\subsection{Lifted Sig-CDEs in Projective Limit Tensor Spaces}
\label{subsec_scaling_argument}

In this section, we revisit the well-posedness of Sig-CDEs from a different perspective. Instead of working directly with the path-dependent formulation \eqref{eq_sig_cde_integral_form}, we lift the dynamics to an infinite-dimensional classical CDE. This is natural, as signatures themselves satisfy a linear CDE:
if $y: [0,T_1] \to \RR^m$ solves \eqref{eq_sig_cde_integral_form} then $t \mapsto S_t \equiv S(y)_{T_0,t+T_0}$ satisfies 
    \begin{align}
    \label{eq_lifted_sig_cde}
        S_t = \mathbf{1} + \int_0^t S_u \otimes \big( F(\mathbf{g}_w \otimes S_u) \dd x_u \big),
    \end{align}
where $\mathbf{g}_w = S(w)_{0,T_0}$, and the equation is understood as a formal identity along signature paths. Conversely, any solution $S: [0,T_1 - T_0] \to \mathcal{T}^{(p)}(\RR^m)$ to \eqref{eq_lifted_sig_cde} induces a solution $t \mapsto (w \sqcup S^{(1)})_t$ of \eqref{eq_sig_cde_integral_form}.

Our aim is thus to  establish global existence and uniqueness for such lifted equations under conditions analogous to those used in the path-dependent setting, but formulated entirely at the level of tensor spaces. More concretely, we now consider CDEs of the form
\begin{align}
    \label{eq_lifted_sig_cdes_integral_form}
    S_t = S_0 + \int_0^t S_u \otimes F(S_u) \dd x_u = S_0 + \sum_{j=1}^d \int_0^t S_u \otimes F^j(S_u) \dd x^j_u,
\end{align}
under a Lipschitz condition on $\cT^{(p)}(\RR^m)$ together with a tensor-adapted homogeneous linear growth condition. We do not require $S$ to be a signature path,\footnote{In particular, $S_0$ does not have to be the unit tensor $\mathbf{1}$.} thereby obtaining a formulation that mirrors Proposition \ref{prop_well_posedness_SigCDEs} while remaining independent of the path-dependent framework. Lastly, since the extension is straightforward, we allow vector fields 
\[F : \cT^{(p)}(\RR^m) \to L\big(\RR^d,\mathfrak{t}^M(\RR^m) \big)\]
for some $M \geq 1$, rather than confining the components $F^j$ to the first level of $\mathfrak{t}^M(\RR^m)$ as in \eqref{eq_lifted_sig_cde}.

In what follows, we denote by $\cT^{(p)}_{\mathbf{1}}(\RR^m)$ the subset of tensors $\mathbf{a} \in \cT^{(p)}(\RR^m)$ with $\mathbf{a}^{(0)} = 1$. Similarly, $\mathbf{T}_{\mathbf{1}}(\RR^m)$ denotes the subset of the second projective limit \eqref{eq_second_projective_limit} introduced in Section \ref{subsec_limits_of_tensor_spaces} consisting of tensors satisfying $\mathbf{a}^{(0)} = 1$. The main result of this section is the following.

\begin{theorem}
\label{thm_scaling_argument}
    Fix $p>1$ and weights $(\lambda_k)$ such that $\lambda_k = o(k^q)$ for every $q>1$. Consider a map $F: \mathcal{T}^{(p)}_{\mathbf{1}}(\RR^m) \to L\big(\RR^d, \mathfrak{t}^M(\RR^m)\big)$ with $M \in \NN$. For each $j \in \{1,\dots,d\}$, assume that $F^j$ is locally Lipschitz: for every $\mathbf{a} \in \cT^{(p)}(\RR^m)$ there exists a neighborhood $U_{\mathbf{a}}$ of $\mathbf{a}$ and a constant $L_{\mathbf{a}}>0$ such that
    \begin{align}
        \label{eq_mixed_loc_lip}
        | F^j(\mathbf{a}_1) - F^j(\mathbf{a}_2) | \leq L_{\mathbf{a}} \| \mathbf{a}_1 - \mathbf{a}_2 \|_{(p)}, \qquad  \mathbf{a}_1,\mathbf{a}_2 \in U_{\mathbf{a}}.
    \end{align}
    Additionally, suppose that $F^j$ satisfies a modified homogeneous linear growth condition: for every $R>0$ there exists $C_F > 0$ such that, for all $\mathbf{a} \in T(\mathbb{R}^m) \subset \cT^{(p)}(\RR^m)$ with $\| \mathbf{a} - \mathbf{1} \|_{(p)} \leq R$, 
    \begin{align}
        \label{eq_generalized_hlg}
        \|\delta_{\alpha(1 + \lambda)^{-1}}F^j(\delta_\lambda \mathbf{a})\|_1 \leq C_F \, \alpha, \qquad \alpha, \lambda \geq 0.
    \end{align}
    Under these assumptions, for every $S_0 \in \mathbf{T}_{\mathbf{1}}(\RR^m)$ the CDE \eqref{eq_lifted_sig_cdes_integral_form} admits a unique global solution $S \in C^0\big([0,T],\cT^{(l)}_{\mathbf{1}}(\RR^m) \big)$ for every $l > 1$ and $T>0$. Moreover, $S_t \in \mathbf{T}_{\mathbf{1}}(\RR^m)$ for all $t \in [0,T]$, and the associated flow map $\mathbf{T}_{\mathbf{1}}(\RR^m) \times [0,T] \ni (S_0,t) \mapsto S_t \in \mathbf{T}_{\mathbf{1}}(\RR^m)$ is continuous. 
\end{theorem}

This problem is less straightforward than it may first appear. As noted in Section \ref{subsec_limits_of_tensor_spaces}, the space $\cT^{(p)}(\RR^m)$ is in general not a Banach algebra for $p>1$, so directly estimating the $\cT^{(p)}(\RR^m)-$norm of the tensor product in \eqref{eq_lifted_sig_cdes_integral_form} is not feasible. In particular, a Lipschitz condition on $F$ in $\cT^{(p)}(\RR^m)$ does not yield a Lipschitz condition in $\cT^{(p)}(\RR^m)$ for the full vector field $S_u\otimes F(S_u)$.

Additionally, the modified homogeneous linear growth condition \eqref{eq_generalized_hlg} allows $F$ to be unbounded, which lies outside the usual scope of classical CDE theory. Indeed, standard approaches work in a Banach algebra $\mathcal A$, where global existence follows from a linear growth condition $\| S_u\otimes F(S_u)\|_{\cA} \lesssim 1 + \| S_u \|_{\cA}$. Such an estimate is in turn guaranteed by assuming $F$ to be bounded, a restriction that our setting deliberately avoids.

\begin{remark}
    We note that \eqref{eq_generalized_hlg} is a variation of the homogeneous linear growth condition \eqref{eq_homogeneous_lin_growth}. If $F^j$ takes values in $\RR^m$, then
    \[
    \| \delta_{\alpha (1 + \lambda)^{-1}} F^j(\delta_\lambda \mathbf{a}) \|_1 = \alpha (1 + \lambda)^{-1} |F^j(\delta_\lambda \mathbf{a})|,
    \]
    so the resulting bound has the same form as \eqref{eq_homogeneous_lin_growth}. The key difference lies in the class of tensors considered: in \eqref{eq_generalized_hlg} the condition is imposed on a $\cT^{(p)}(\RR^m)-$neighborhood in the tensor algebra, whereas \eqref{eq_homogeneous_lin_growth} only concerns a neighborhood of signatures.
\end{remark}

As outlined below, the main difficulty in proving Theorem \ref{thm_scaling_argument} lies in establishing global existence, since local well-posedness follows readily from a truncation argument. To obtain global existence, we instead employ a \emph{scaling argument}. Briefly put, the dilation operator $\delta_\lambda$ (Section \ref{subsec_path_signatures_tensor_algebras}) is used to rescale the solution path into a bounded set. Intuitively, if the solution were to escape every bounded subset of $\cT^{(p)}(\RR^m)$ in finite time, keeping the rescaled path within a fixed bounded set would require increasingly smaller scaling factors. The key idea then is to show that such indefinite rescaling does not occur, and hence the solution cannot blow up in finite time.

To formalize this idea, fix $R>0$ and $p > 1$, and define the \emph{gauge}
\begin{align}
\label{eq_scaling_gauge}
    \lambda(\mathbf{a}) \coloneqq 1 \land \sup \big\{ \lambda \geq 0 : \| \delta_\lambda \mathbf{a}  - \mathbf{1} \|_{(p)} \leq R \big\},
\end{align}
for $\mathbf{a} \in \cT^{(p)}_{\mathbf{1}}(\RR^m)$, as the largest factor $0 < \lambda \leq 1$ for which $\delta_\lambda \mathbf{a}$ lies inside the ball of radius $R$ about $\mathbf{1}$ in $\cT^{(p)}(\RR^m)$. One easily checks that $\| \delta_{\lambda (\mathbf{a})} \mathbf{a} - \mathbf{1} \|_{(p)} \leq R$ and that the map $[0,1] \ni \lambda \mapsto \| \delta_{\lambda} \mathbf{a} - \mathbf{1} \|_{(p)}$ is continuous and increasing.

\begin{remark}
\label{rem_continuity_scaling_gauge}
    Importantly, $\lambda(\,\cdot\,)$ is continuous on $\cT^{(p)}_{\mathbf{1}}(\RR^m)$. To see this, define $\Phi(\lambda,\mathbf a)\coloneqq \|\delta_\lambda\mathbf a-\mathbf1\|_{(p)}$, which is jointly continuous on $[0,1]\times\cT^{(p)}_{\mathbf1}(\RR^m)$. Let $\mathbf a_n\to\mathbf a$ in $\cT^{(p)}_{\mathbf1}(\RR^m)$. If $\lambda(\mathbf a)<1$, then $\Phi(\lambda(\mathbf a),\mathbf a)=R$ and, by monotonicity, $\Phi(\lambda(\mathbf a)+\varepsilon,\mathbf a)>R$ for every $\varepsilon>0$. Joint continuity of $\Phi$ then implies $\Phi(\lambda(\mathbf a)+\varepsilon,\mathbf a_n)>R$ for all sufficiently large $n$, and hence $\lambda(\mathbf a_n)\le \lambda(\mathbf a)+\varepsilon$. Similarly, since $\Phi(\lambda(\mathbf a)-\varepsilon,\mathbf a)<R$, we obtain $\lambda(\mathbf a_n)\ge \lambda(\mathbf a)-\varepsilon$ for large $n$. Thus $\lambda(\mathbf a_n)\to\lambda(\mathbf a)$. The case $\lambda(\mathbf a)=1$ is analogous, using that $\Phi(1-\varepsilon,\mathbf a)<R$ for every $\varepsilon>0$.
\end{remark}

We prove two auxiliary lemmas before turning to the main result.

\begin{lemma}
\label{lem_aux_scaling_arg_1}
    Suppose condition \eqref{eq_mixed_loc_lip} holds and $S_0 \in \mathbf{T}_{\mathbf{1}}(\RR^m)$. For $N \ge 1$ and $j \in \{1,\dots,d\}$, define $F^{N,j} : \mathbf{1} + \mathfrak{t}^N(\RR^m) \to \mathfrak{t}^N(\RR^m)$ by $F^{N,j}(\mathbf{a}) = \pi_{0,N} \big( \iota_M F^j \big(\iota_N(\mathbf{a})\big) \big)$ and consider the system
    \begin{align}
    \label{eq_truncated_system_scaling}
        S^{(N)}_t = \pi_{0,N}(S_0) + \sum_{j=1}^d \int_0^t S^{(N)}_u \otimes_N F^{N,j}\!\big(S^{(N)}_u\big) \dd x^j_u .
    \end{align}
    Then, for every $N \ge 1$, the system admits a unique local solution on a maximal interval of existence $[0,T_N)$ for some $T_N>0$. In particular, $ S^{(N)} \in C^{1\text{-Höl}} \big([0,T_N), \mathbf{1} + \mathfrak{t}^N(\RR^m)\big)$.
\end{lemma}

\begin{remark}
    To simplify notation, we often omit the zero-padding maps $\iota_N : T^N(\RR^m) \to T(\RR^m)$, as the ambient tensor space will typically be clear from the context.
\end{remark}

\begin{proof}
    Since $T^N(\RR^m)$ is finite-dimensional, the norm induced by $\mathcal{T}^{(p)}(\RR^m)$ is equivalent to the Euclidean norm on $T^N(\RR^m)$. Consequently, each vector field $G^j(\mathbf{a}) = \mathbf{a} \otimes_N F^{N,j}(\mathbf{a})$ is locally Lipschitz on $T^N(\RR^m)$. Classical CDE theory therefore yields existence and uniqueness of a local solution on a maximal interval $[0,T_N)$ for some $T_N>0$.
\end{proof}

\begin{lemma}
\label{lem_aux_scaling_arg_2}
    Suppose the assumptions of Theorem \ref{thm_scaling_argument} hold. Then, for every $T>0$ and every $N\ge 1$, the solution $S^{(N)}$ extends to $[0,T]$. Moreover, for every $h > 1$, there exist constants $K_S,K_F>0$ such that 
    \begin{align}
        \sup_{N\ge 1}\sup_{t\in[0,T]}
        \|S^{(N)}_t\|_{(h)}
        \le K_S,
        \qquad
        \sup_{N\ge 1}\sup_{t\in[0,T]}
        \|F^{N,j}(S^{(N)}_t)\|_{(h)}
        \le K_F .
    \end{align}
\end{lemma}
\begin{proof}
    Fix $N\geq 1$ and write $\lambda_{N,t}\coloneqq \lambda(S^{(N)}_t)$ as in \eqref{eq_scaling_gauge}. For $t\in[0,T_N)$ and $\ell>0$, applying \eqref{eq_generalized_hlg} with $\lambda=1/\lambda_{N,t}$ and $\alpha = (1 + \lambda) \ell$ gives
    \begin{align}
        \label{eq_dilated_vector_field_control}
        \big \| \delta_\ell F^{N,j}\big(S^{(N)}_t \big) \big\|_1 \leq \big \| \delta_\ell F^j\big(S^{(N)}_t \big) \big\|_1 \leq C_F \ell \left( 1 + \frac{1}{\lambda_{N,t}}\right).
    \end{align}
    Next, applying $\delta_\ell$ to \eqref{eq_truncated_system_scaling} and using that $\delta_\ell$ is a homomorphism, we obtain
    \[
    \| \delta_\ell S^{(N)}_t \|_1 \lesssim \|\delta_\ell S_0 \|_1 + C_F \ell \int_0^t \|\delta_\ell S^{(N)}_u \|_1 \Big( 1 + \frac{1}{\lambda_{N,u}} \Big) |\dot{x}^j_u| \dd u,
    \]
    since $\big( T_1(\RR^m), \| \cdot \|_1 , \otimes \big)$ is a Banach algebra and truncation does not increase $\|\cdot\|_1$. Grönwall’s inequality then yields for $0\le t<T_N$
    \begin{align}
        \label{eq_first_gronwall_estimate}
        \| \delta_\ell S^{(N)}_t \|_1 \leq \|\delta_\ell S_0 \|_1 \exp \left( \ell C \int_0^t \Big( 1 + \frac{1}{\lambda_{N,u}} \Big) \dd u\right),
    \end{align}
    where $C = C_F \| x \|_{1\text{-Höl};[0,T]}$. Note that  $S^{(N)}$ is continuous on $[0,t]$ and $\lambda(\, \cdot \,)$ is continuous on $T^N(\RR^m)$, therefore the map $u \mapsto \lambda_{N,u}$ is continuous on $[0,t]$ and the integral above is well-defined. Set
    \[
    I_N(t)\coloneqq C\!\int_0^t\!\Big(1+\frac{1}{\lambda_{N,u}}\Big)\mathrm du,
    \qquad
    \psi_p(x) \coloneqq \frac{1}{Z_p} \sum_{k \geq 0} \exp(-k^p)  \|S_0\|_{1,\lambda_k} \big( \exp(\lambda_k x) - 1 \big).
    \]
    Observe that $\psi_p(x) < \infty$ for every $x\geq 0$ by the choice of $(\lambda_k)$ and the assumption on $S_0$. Indeed, up to a constant, $\psi_p(x)$ is bounded by $\sum_{k\geq 0} \exp(-k^p + k^r + \lambda_k x)$ for some $r<p$, which converges for all $x\geq 0$. Moreover, $\psi_p$ is continuous, strictly increasing, and $\psi_p(0) = 0$.

    Now, from the definition of $\| \cdot \|_{(p)}$ and \eqref{eq_first_gronwall_estimate} we obtain, for $0<\ell \leq 1$,
    \[
    \|\delta_\ell S^{(N)}_t-\mathbf1\|_{(p)} \leq \| S_0 - \mathbf{1} \|_{(p)} + \psi_p(\ell I_N(t))
    \]
    With this in mind, consider $R > \| S_0 - \mathbf{1} \|_{(p)}$ and take
    \[
    \ell = 1 \wedge \frac{\psi^{-1}_p\big(R -  \| S_0 - \mathbf{1} \|_{(p)} \big)}{I_N(t)}.
    \]
    Then $\|\delta_\ell S^{(N)}_t-\mathbf1\|_{(p)}\le R$, and so by definition of $\lambda_{N,t}$ we have $\lambda_{N,t}\ge\ell$. Consequently,
    \[
    \frac1{\lambda_{N,t}}
    \le 1 + \frac{C}{\psi^{-1}_p(R')}\int_0^t\Big(1+\frac{1}{\lambda_{N,u}}\Big)\dd u,
    \qquad t\in[0,T_N),
    \]
    with $R' \coloneqq R -  \| S_0 - \mathbf{1} \|_{(p)}$, and a second application of Grönwall’s inequality gives
    \[
    \frac{1}{\lambda_{N,t}}
    \leq \left( 1+ \frac{Ct}{\psi^{-1}_p(R')} \right) \exp \left(\frac{Ct}{\psi^{-1}_p(R')}\right),
    \qquad t\in[0,T_N),
    \]
    implying that $\lambda_{N,t}$ stays uniformly bounded away from $0$ on $[0,T]$, independently of $N$.
    
    Returning to \eqref{eq_first_gronwall_estimate}, letting $\ell = \lambda_k$ and summing with the weights $\exp(-k^h)$ gives finiteness of $\|S^{(N)}_t\|_{(h)}$ for any $h>1$ uniformly in $N$. In finite dimensions such \emph{a priori} bounds preclude blow-up, so the solution exists globally on $[0,T]$ for any $T>0$. Finally, the bound for $\|F^{N,j}(S^{(N)}_t)\|_{(h)}$ follows directly from \eqref{eq_dilated_vector_field_control} with $\ell=\lambda_k$ and the uniform control of $\lambda_{N,t}$.
\end{proof}

\begin{proof}[Proof of Theorem \ref{thm_scaling_argument}]
    Fix $r \in \mathbb{N}$, and set
    \[l_r \coloneqq 1 + \frac{1}{r}, \quad h_r \coloneqq 1 + \frac{1}{1+r},\]
    so that $1<h_r<l_r$. For each $N\geq 1$, let $S^{(N)}:[0,T] \to \mathcal{T}^{(h_r)}(\RR^m)$ be the global solution given by Lemmas \ref{lem_aux_scaling_arg_1} and \ref{lem_aux_scaling_arg_2}. We first show that $(S^{(N)})$ is relatively compact in $C^0\big([0,T],\mathcal{T}^{(l_r)}(\RR^m) \big)$ for every $r\geq 1$.

    For $0\leq s \leq t \leq T$, \eqref{eq_truncated_system_scaling} gives
    \[ 
    S^{(N)}_t - S^{(N)}_s = \sum_{j=1}^d \int_s^t \pi_{0,N} \big(S^{(N)}_u \otimes F^j(S^{(N)}_u) \big) \dd x^j_u. 
    \]
    Since $\|\pi_{0,N}(\mathbf a)\|_{(l_r)}\leq \|\mathbf a\|_{(l_r)}$ and the product 
    $\otimes:\mathcal{T}^{(h_r)}(\RR^m)\times\mathcal{T}^{(h_r)}(\RR^m)\to\mathcal{T}^{(l_r)}(\RR^m)$ is continuous, Lemma \ref{lem_aux_scaling_arg_2} yields up to a constant independent of $N$,
    \[
        \|S^{(N)}_t - S^{(N)}_s\|_{(l_r)} \lesssim \int_s^t \| S^{(N)}_u\|_{(h_r)} \| F^j(S^{(N)}_u)\|_{(h_r)} |\dot{x}^j_u| \dd u
        \lesssim K_SK_F \| x\|_{1\text{-Höl};[0,T]} |t-s|,
    \]
    Hence the family $\{S^{(N)}\}$ is uniformly Lipschitz in $\mathcal{T}^{(l_r)}(\RR^m)$, and in particular equicontinuous. Moreover, Lemma \ref{lem_aux_scaling_arg_2} shows that for each fixed $t\in[0,T]$ the set $\{S^{(N)}_t : N \geq 1\}$ is bounded in $\mathcal{T}^{(h_r)}(\RR^m)$, and therefore relatively compact in $\mathcal{T}^{(l_r)}(\RR^m)$ by Lemma \ref{lem_compact_emb_limit_spaces}. By Arzelà-Ascoli, $(S^{(N)})$ is relatively compact in $C^0\big([0,T],\mathcal{T}^{(l_r)}(\RR^m) \big)$ for every $r\geq 1$ \cite[Theorem 47.1]{munkres2000}.

    Then, applying a diagonal argument to the sequence of spaces $C^0\big([0,T],\mathcal{T}^{(l_r)}(\RR^m) \big)$, $r\geq 1$, we obtain a subsequence, denoted $(S^{(N_m)})$, and a path
    \[S^{[r]} \in C^0\big([0,T], \mathcal{T}^{(l_r)} (\mathbb{R}^m )\big),\]
    for each $r\geq 1$, such that
    \[\sup_{t\in [0,T]} \| S_t^{(N_m)} - S_t^{[r]} \|_{(l_r)} \longrightarrow 0 \quad \text{as} \ m\to \infty,\]
    for every $r\in \NN$. Crucially, these limits are compatible. Indeed, whenever $r' \geq r$, the canonical embedding $\cT^{(l_{r'})}(\RR^m) \hookrightarrow \cT^{(l_{r})}(\RR^m)$ is continuous, so the same subsequence converges in $C^0\big([0,T], \mathcal{T}^{(l_r)} (\mathbb{R}^m )\big)$ both to $S^{[r]}$ and to the image of $S^{[r']}$. By uniqueness of limit, these coincide. Therefore, by the universal property of the initial topology, there exists a unique 
    \[S \in C^0 \big( [0,T], \mathbf{T}_{\mathbf{1}}(\RR^m) \big)\]
    whose image in each $C^0\big([0,T], \mathcal{T}^{(l_r)} (\mathbb{R}^m )\big)$ is $S^{[r]}$. In particular, 
    \[\sup_{t\in [0,T]} \| S_t^{(N_m)} - S_t \|_{(l_r)} \longrightarrow 0 \quad \text{as} \ m\to \infty,\]
    for every $r \geq 1$. Now, fix arbitrary $l>1$ and choose $r$ sufficiently large so that $l_r < l$. As before, since the embedding $\cT^{(l_{r})}(\RR^m) \hookrightarrow \cT^{(l)}(\RR^m)$ is continuous, it follows that 
    \[\sup_{t\in [0,T]} \| S_t^{(N_m)} - S_t \|_{(l)} \longrightarrow 0.\]
    Thus, $(S^{(N_m)})$ converges to the same path $S$ in $C^0\big([0,T], \mathcal{T}^{(l)} (\mathbb{R}^m )\big)$ for arbitrary $l>1$.
    
    Fatou’s lemma with respect to the counting measure then implies that $S_t \in \mathcal{T}^{(h)}(\RR^m)$ for all $t$ and, in particular, $\sup_{t\in[0,T]}\|S_t\|_{(h)}\le K_S$ (with $K_S$ coming from Lemma \ref{lem_aux_scaling_arg_2}, dependent on $h$). It remains to show that $S$ is the unique solution to \eqref{eq_lifted_sig_cdes_integral_form}.

    Fix $n \in \NN$. For all $m$ sufficiently large so that $N_m\ge n$, applying $\pi_{0,n}$ to \eqref{eq_truncated_system_scaling} gives
    \begin{align}
        \label{eq_doubly_truncated_system}
        \pi_{0,n}(S^{(N_m)}_t) = \pi_{0,n}(S_0) + \sum_{j=1}^d \int_0^t \pi_{0,n}\big( S^{(N_m)}_u \otimes F^{j} \big(S^{(N_m)}_u\big) \big) \dd x^j_u.
    \end{align}
    Because $\pi_{0,n}: \mathcal T^{(l)}(\RR^m) \to T^n(\mathbb R^m)$ is continuous for any $l>1$, 
    $\pi_{0,n}(S^{(N_m)})\to\pi_{0,n}(S)$ uniformly on $[0,T]$. Furthermore, by \eqref{eq_mixed_loc_lip} and the uniform $\mathcal{T}^{(h)}(\RR^m)$–bounds on $S^{(N_m)}$ and $S$ from Lemma \ref{lem_aux_scaling_arg_2} (with $h<p$), each $F^j$ is Lipschitz on a common compact set containing all these paths. Hence, 
    \[ 
    \sup_{t \in [0,T]} | F^j(S^{(N_m)}_t) - F^j(S_t) | \lesssim \sup_{t \in [0,T]} \| S^{(N_m)}_t - S_t \|_{(p)} \longrightarrow 0,
    \]
    The integrands in \eqref{eq_doubly_truncated_system} thus converge uniformly in $C^0\big([0,T],\cT^{(l)}(\RR^m) \big)$ for $l>1$, and dominated convergence yields 
    \[
    \pi_{0,n}(S_t) = \pi_{0,n}(S_0) + \sum_{j=1}^d \int_0^t \pi_{0,n}\big( S_u \otimes F^{j} (S_u) \big) \dd x^j_u.
    \]
    Since $n$ was arbitrary, this identity holds in the full tensor algebra, and hence in $\mathcal T^{(l)}(\RR^m)$. Therefore, $S$ is a continuous $\mathcal T^{(l)}$–valued solution of \eqref{eq_lifted_sig_cdes_integral_form}.

     Regarding uniqueness, let $S,S'\in C^0\big([0,T],\mathcal{T}^{(l)}(\RR^m)\big)$ for all $l>1$ be solutions of \eqref{eq_lifted_sig_cdes_integral_form} with initial conditions $S_0,S'_0 \in \mathbf{T}_{\mathbf{1}}(\RR^m)$. By Remark \ref{rem_continuity_scaling_gauge} and compactness, it follows that 
     \[
     \sup_{t\in[0,T]}\left(\frac{1}{\lambda(S_t)}+\frac{1}{\lambda(S'_t)}\right)<\infty.
     \]
    Next, set $D_k(t) \coloneqq \|\delta_{\lambda_k}(S_t-S'_t)\|_1$. Subtracting the integral equations for $S$ and $S'$, applying $\delta_{\lambda_k}$, and using the Banach–algebra property of $T_{1,\lambda_k}(\RR^m)$ yields
    \[
    D_k(t)\leq D_k(0)
    +\sum_{j=1}^d \int_0^t \Big(
    D_k(u) \|F^j(S_u) \|_{1,\lambda_k}
    +\| S'_u\|_{1,\lambda_k} \| F^j(S_u)-F^j(S'_u) \|_{1,\lambda_k}
    \Big)|\dot x_u^j| \dd u.
    \]
    From \eqref{eq_dilated_vector_field_control} and the uniform bound on $\lambda(S_u)^{-1}$ we obtain
    \begin{align}
    \label{eq_uniqueness_aux_1}
        \| F^j(S_u)\|_{1,\lambda_k}\leq C_F \lambda_k \Big(1+\lambda(S_u)^{-1}\Big)\leq c_1 \lambda_k,
    \end{align}
    for some $c_1>0$, and similarly for $S'$. Moreover, combining \eqref{eq_uniqueness_aux_1} with the same Grönwall argument used in \eqref{eq_first_gronwall_estimate} gives
    \begin{align}
    \label{eq_uniqueness_aux_2}
        \sup_{u\in[0,T]} \| S'_u\|_{1,\lambda_k} \leq \| S_0'\|_{1,\lambda_k} \exp(c_2 \lambda_k)
    \end{align}
    for a constant $c_2 >0$ independent of $k$. Finally, since $K\coloneqq S([0,T])\cup S'([0,T])$ is compact in $\mathcal T^{(p)}(\RR^m)$, condition \eqref{eq_mixed_loc_lip} implies the existence of $L_K>0$ such that
    \begin{align}
    \label{eq_uniqueness_aux_3}
       \|F^j(S_u)-F^j(S'_u)\|_{1,\lambda_k} \lesssim (1 + \lambda_k^M) |F^j(S_u)-F^j(S'_u)| \leq L_K \|S_u-S'_u\|_{(p)}.
    \end{align}
    Inserting \eqref{eq_uniqueness_aux_1}–\eqref{eq_uniqueness_aux_3} into the inequality for $D_k$ yields
    \[
    D_k(t) \lesssim D_k(0) + \sum_{j=1}^d \int_0^t c_1 \lambda_k D_k(u) + \|  S_0' \|_{1,\lambda_k} \exp(c_2\lambda_k)  (1 + \lambda_k^M) \| S_u - S_u'\|_{(p)} |\dot{x}^j_u| \dd u,
    \]
    which, by Grönwall's inequality, implies
    \[
    D_k(t) \lesssim  \exp(c_1 \lambda_k t) D_k(0) + \int_0^t \exp( (c_1 t + c_2 + M) \lambda_k ) \| S_0'\|_{1,\lambda_k}   \| S_u - S_u'\|_{(p)} \dd u
    \]
    Now, summing over $k$ with weights $\exp(-k^p)/Z_p$ and using that $S_0,S_0' \in \cT^{(r)}(\RR^m)$ for some $r<p$ gives
    \begin{align}
        \label{eq_final_gronwall_scaling}
        \| S_t - S'_t \|_{(p)} \lesssim C_0 \| S_0 - S'_0 \|_{(r)} + \int_0^t C_1 \| S_u - S'_u \|_{(p)}  \dd u
    \end{align}
    where $C_0,C_1<\infty$ because the series
    \[
    \sum_{k\geq 0} \exp\big(-k^p + k^r + c_1 \lambda_k t\big), \qquad
    \sum_{k\geq 0} \exp\big(-k^p + k^r + (c_1 t + c_2 + M)\lambda_k \big)
    \]
    are finite. If $S_0=S'_0$, Grönwall’s inequality applied to the last estimate yields $S=S'$ on $[0,T]$, proving uniqueness.

    Finally, the flow property on $\mathbf{T}_{\mathbf{1}}(\RR^m)$ follows from the continuity of the inclusions $\mathbf{T}_{\mathbf{1}}(\RR^m) \hookrightarrow \mathcal{T}_{\mathbf{1}}^{(l_r)}(\RR^m)$ and \eqref{eq_final_gronwall_scaling}. Indeed, for each $r \ge 1$, the map
    \[
    \mathbf{T}_{\mathbf{1}}(\RR^m) \times [0,T] \ni (S_0,t) \mapsto S_t \in \mathcal{T}^{(l_r)}(\RR^m)
    \]
    is continuous: for $l_r \geq p$ continuity follows directly from \eqref{eq_final_gronwall_scaling}, and for $1 < l_r < p$ continuity follows by repeating the last application of Grönwall with $\|\cdot \|_{(l_r)}$ instead of $\| \cdot \|_{(p)}$, together with our assumption on the $(\lambda_k)$. Since Lemma \ref{lem_aux_scaling_arg_2} ensures that $S_t \in \mathbf{T}_{\mathbf{1}}(\RR^m)$ for all $t \in [0,T]$, the desired continuity with values in $\mathbf{T}_{\mathbf{1}}(\RR^m)$ follows from the universal property of the initial topology.   
\end{proof}

\subsection{Sig-CDEs on Step\texorpdfstring{$-N$}{N} Groups}
\label{subsec_not_so_esoteric}

In Section \ref{subsec_scaling_argument}, we reformulated Sig-CDEs as $\mathcal{T}^{(p)}(\RR^m)$–valued CDEs, thereby removing the path dependence of \eqref{eq_sig_cde_integral_form}. The corresponding well-posedness conditions were seen to agree with those of Proposition \ref{prop_well_posedness_SigCDEs}. This formulation, however, relies on a vector space and does not yet exploit the intrinsic group structure of signatures. We now adopt a genuinely group-theoretic perspective, viewing truncated Sig-CDEs as
$(\mathbf{1}+\mathfrak{t}^N(\RR^m))-$valued equations, and revisit existence and uniqueness in this intrinsic setting.

\begin{remark}
    By \emph{intrinsic} we mean we will study well-posedness using the metric structure of the group itself — specifically via \emph{homogeneous gauges} (Section \ref{subsec_path_signatures_tensor_algebras}) — rather than norms on an ambient linear space. Although we could equivalently work on $G^N(\RR^m)$ or, more generally, on any homogeneous group, we adopt the model $\mathbf{1}+\mathfrak{t}^N(\RR^m)$ to stay aligned with the tensor coordinates standard in the signature literature and consistent with our earlier tensor space formulation.
\end{remark}

Concretely, we now consider $(\mathbf{1} + \mathfrak{t}^N(\RR^m))$–valued CDEs of the form
\begin{align}
\label{eq_generalized_lifted_Sig_CDE}
\dd S_t = \sum_{j=1}^d S_t \otimes_N F^j(S_t) \dd x^j_t, \quad S_0 \in \mathbf{1} + \mathfrak{t}^N(\RR^m),
\end{align}
where $F^j : \mathbf{1} + \mathfrak{t}^N(\RR^m) \to \mathfrak{t}^N(\RR^m)$ takes values in the Lie algebra $\mathfrak{t}^N(\RR^m)$ and, as before, $x : [0,T] \to \RR^d$ is a Lipschitz control. Similarly to \eqref{eq_lifted_sig_cdes_integral_form},  this provides a natural extension of a truncated lifted Sig-CDE, in which the image of the vector fields $F^j$ was restricted to the first level of $\mathfrak{t}^N(\RR^m)$, corresponding to $\RR^m$.

\begin{remark}
One may ask whether it remains appropriate to refer to
\eqref{eq_generalized_lifted_Sig_CDE} as a Sig-CDE, since in general a solution $S$ need not coincide with the signature of an $m-$dimensional path. Nevertheless, following \cite{bellingeri2022smooth}, it is now standard terminology to regard differentiable paths with values in $G^N(\RR^m)$ as \emph{smooth rough paths}, and their minimal extensions as the corresponding signatures. This covers the case in which $F$ takes values in $\mathfrak{g}^N(\RR^m)$, and we present a recent application of such dynamics in \Cref{exmpl:efm} below. If instead $F$ takes values in $\mathfrak{t}^N(\RR^m)$ outside the free Lie algebra, then the resulting object no longer satisfies the shuffle relations and hence is not a signature in the usual sense. Even so, such \emph{free developments} remain natural objects and appear, for example, in the dynamics of expected signatures associated with certain stochastic processes \cite{friz2025expected}.
\end{remark}

\begin{example}\label{exmpl:efm}
    In \cite{jaber2025exponentially}, the authors consider modified signatures with \emph{exponentially fading memory}.
    For a given path $z:[0,T]\to\RR^m$ (here deterministic and Lipschitz), these are defined as the solution to
    \begin{equation}
    \label{eq:EFM}
        \dd S_t = -\Lambda(S_t) \dd t + S_t \otimes \dd z_t,
        \qquad
        S_0 = 1 \in \mathbf{1}+\mathfrak{t}^N(\RR^m),
    \end{equation}
    (see \cite[Proposition 3.6]{jaber2025exponentially}), where
    \[
        \Lambda: \mathbf{1}+\mathfrak{t}^N(\RR^m) \to \mathfrak{t}^N(\RR^m),
        \qquad
        \mathbf{g} \mapsto \frac{\dd}{\dd t}\,\delta_{e^{\lambda t}} \mathbf{g}\Big\vert_{t=0}.
    \]
    This fits into our current setting by noting that
    \[
        F^0(\mathbf{g})
        \coloneqq
        \mathbf{g}^{-1}\Lambda(\mathbf{g})
        =
        \frac{\dd}{\dd h}\Bigl(\mathbf{g}^{-1}\delta_{e^{\lambda h}} \mathbf{g} \Bigr)\Big\vert_{h=0},
        \qquad
        \mathbf{g} \in \mathbf{1}+\mathfrak{t}^N(\RR^m),
    \]
    defines a smooth map $F^0:\mathbf{1}+\mathfrak{t}^N(\RR^m)\to\mathfrak{t}^N(\RR^m)$. Indeed, $h \mapsto \mathbf{g}^{-1} \delta_{e^{\lambda h}} \mathbf{g}$ is a smooth curve in $\mathbf{1}+\mathfrak{t}^N(\RR^m)$ starting at the identity, so its derivative at $h=0$ belongs to the tangent Lie algebra. In particular, if $\mathbf{g} \in G^N(\RR^m)$, then $F^0(\mathbf{g}) \in\mathfrak{g}^N(\RR^m)$. Setting $x_t = (t,z_t)$, and $F^i \equiv e_i$ for $i=1,\hdots,m$, equation \eqref{eq:EFM} takes the form \eqref{eq_generalized_lifted_Sig_CDE}, and by Proposition \ref{prop_local_wellposedness_Sig_CDEs} below, its unique solution remains in $G^N(\RR^m)$.
\end{example}

We observe that \eqref{eq_generalized_lifted_Sig_CDE} defines a CDE on the smooth manifold $\mathbf{1} + \mathfrak{t}^N(\RR^m)$; see Remark \ref{rem_geometric_form_Sig_CDEs} below. Thus, classical results on CDEs on manifolds ensure that such equations are well-posed. In particular, if each $F^j$ is smooth, then a unique local solution exists, and global existence follows under standard growth conditions — see \cite[Theorem D.6]{lee2003smooth} and \cite[Theorem 2.15]{agrachev2019comprehensive}.

That said, by exploiting the global diffeomorphism $\log_{\otimes_N} : \mathbf{1} + \mathfrak{t}^N(\RR^m) \to \mathfrak{t}^N(\RR^m)$, one may equivalently retain the control-theoretic formulation of \eqref{eq_generalized_lifted_Sig_CDE}, which is standard in the literature on signatures. Indeed, setting $\Omega_t \coloneqq \log_{\otimes_N}(S_t)$, it follows from \cite[Lemma 7.23]{Friz2010} that
\begin{align}
\label{eq_log_transform_Sig_CDE}
\dd \Omega_t
= \left( \sum_{k=0}^{N-1} \frac{(-1)^k B_k}{k!} \operatorname{ad}_{\Omega_t}^k \right)
F(\exp_{\otimes_N}(\Omega_t)) \dd x_t,
\quad
\Omega_0 = \log_{\otimes_N}(S_0),
\end{align}
where $B_k \in \{1, - 1/2, \hdots\}$ are the Bernoulli numbers and $\operatorname{ad}_{\Omega_t}: \mathfrak{t}^N(\RR^m) \to \mathfrak{t}^N(\RR^m)$ denotes the usual adjoint operator. Based on this representation, we say that $S$ solves \eqref{eq_generalized_lifted_Sig_CDE} if and only if $\Omega$ solves the above equation, allowing us to study well-posedness entirely within the classical framework of CDEs in finite-dimensional normed spaces. 

\begin{remark}
    Since the approximating equations \eqref{eq_approximating_sig_cdes} already involve truncated signatures, this finite-dimensional perspective arises naturally. Moreover, as noted in \cite{le2021space} and \cite[Section 8.3]{schmeding2022introduction}, the group of full (non-truncated) signatures may fail to possess a Lie or even topological group structure, which further justifies focusing on the truncated case for this intrinsic setting.
\end{remark}

\begin{proposition} 
\label{prop_local_wellposedness_Sig_CDEs}
    Fix $T>0$ and equip $\mathbf{1} + \mathfrak{t}^N(\RR^m)$ with the Euclidean norm on $T^N(\RR^m)$. Consider $F : \mathbf{1}+\mathfrak{t}^N(\RR^m) \to L(\RR^d,\mathfrak{t}^N(\RR^m))$. If $F$ is continuous, then there exists $t_{\max}\in(0,T]$ and at least one solution $S:[0,t_{\max}) \to \mathbf{1}+\mathfrak{t}^N(\RR^m)$ to \eqref{eq_generalized_lifted_Sig_CDE}. If $F$ is locally Lipschitz, this solution is unique; and if $F$ is smooth, then the solution map $S_0 \mapsto S$ is also smooth. Moreover, if each $F^j$ takes values in $\mathfrak{g}^N(\RR^m)$, then the solution $S$ takes values in $G^N(\RR^m)$ provided it starts in $G^N(\RR^m)$. 
\end{proposition}
\begin{proof}
    The maps $\exp_{\otimes_N}$ and $\operatorname{ad}$ are smooth, since both are polynomials in tensor coordinates.\footnote{Equivalently, the coordinate representation of the exponential map becomes the identity if we use exponential coordinates, i.e. if we identify $\mathbf{1} + \mathfrak{t}^N(\RR^m)$ with its Lie algebra $\mathfrak{t}^N(\RR^m)$ via $\exp_{\otimes N}$.} Consequently, local existence and uniqueness follow directly from \cite[Theorem 3.4]{Friz2010} and \cite[Corollary 3.9]{Friz2010}. The smoothness of the solution map follows from \cite[Theorem 4.4]{Friz2010}.
\end{proof}

\begin{remark}
\label{rem_geometric_form_Sig_CDEs}
    For completeness, we note that \eqref{eq_generalized_lifted_Sig_CDE} can also be written in a more geometric form. Let $L_{\mathbf{a}} : \mathbf{1} + \mathfrak{t}^N(\RR^m) \to \mathbf{1} + \mathfrak{t}^N(\RR^m)$ denote left-multiplication by $\mathbf{a}$, i.e. $\mathbf{b} \mapsto \mathbf{a} \otimes_N \mathbf{b}$. Then \eqref{eq_generalized_lifted_Sig_CDE} can be expressed as 
    \begin{align}
    \label{eq_geometric_formulation_Sig_CDEs}
        \dot{S}_t = (L_{S_t})_* \big(F(S_t) \dot{x}_t\big) 
    = \sum_{j=1}^d (L_{S_t})_* \big(F^j(S_t)\big) \, \dot{x}^j_t,
    \end{align}
    where $(L_{S_t})_* : \mathfrak{t}^N(\RR^m) \to T_{S_t} (\mathbf{1} + \mathfrak{t}^N(\RR^m))$ denotes the differential of $L_{S_t}$, mapping elements in the Lie algebra into the tangent space $T_{S_t} (\mathbf{1} + \mathfrak{t}^N(\RR^m))$. Recall that the derivative $\dot{x}$ exists almost everywhere since $x \in C^{1\text{-Höl}}([0,T_1],\RR^d)$ (see \cite[Proposition 1.32]{Friz2010}). A closer inspection of the aforementioned results for ODEs on manifolds reveals that they impose the same local well-posedness conditions on $F$ as those stated in Proposition \ref{prop_local_wellposedness_Sig_CDEs}.
\end{remark}

\begin{remark}
    Smoothness of the solution map $S_0 \mapsto S$ becomes especially relevant in view of Theorem \ref{prop_dynamic_universality_activations}. Indeed, Proposition \ref{prop_local_wellposedness_Sig_CDEs} implies that the solution map associated with the (finite-dimensional) lift of \eqref{eq_approximating_sig_cdes} depends smoothly on the parameters $c_0, c_1$, and on each $\bell_n^{ij}$ for all $i \in \{1,\hdots,m\}$ and $j \in \{1,\hdots,d\}$. 
\end{remark}

We now examine well-posedness under \emph{intrinsic} conditions. The ball–box estimate \cite[Proposition 7.49]{Friz2010} implies that the gauge $\|\cdot\|_h$ (Section \ref{subsec_path_signatures_tensor_algebras}) induces the same topology as the standard manifold structure on $\mathbf{1}+\mathfrak{t}^N(\RR^m)$.\footnote{That is, the Euclidean topology induced through $\exp_{\otimes N}$, or, equivalently, by restricting the ambient tensor space metric (cf. \cite[Remark 7.31]{Friz2010}).} Consequently, continuity of
$F : \mathbf{1}+\mathfrak{t}^N(\RR^m)\to L(\RR^d,\mathfrak{t}^N(\RR^m))$
still guarantees local existence, exactly as in Proposition \ref{prop_local_wellposedness_Sig_CDEs}.

Intrinsic regularity, however, is considerably weaker. Assume
\[
|F(\mathbf{a})-F(\mathbf{b})| \lesssim \|\mathbf{a}-\mathbf{b}\|_h
\quad \text{and} \quad
|F(\mathbf{a})|\lesssim 1+\|\mathbf{a}\|_h.
\]
Because the metric induced by $\|\cdot\|_h$ is only $1/N-$Hölder equivalent to the Euclidean distance, such an intrinsic Lipschitz condition translates into substantially lower regularity from the Euclidean perspective. Classical uniqueness results can therefore not be applied in their usual form. Similarly, since locally $|\mathbf{a}| \lesssim \|\mathbf{a}\|_h$ the growth bound above is weaker than the standard Euclidean linear growth assumption, and so the usual criteria for global existence no longer apply directly.

To the best of our knowledge, \cite{magnani2018lipschitz} is the only work that studies ODE well-posedness on homogeneous groups under explicitly intrinsic assumptions. Notably, the authors exhibit a counterexample to local uniqueness under conditions of the type above — an issue to which we return later. With this in mind, we then turn our attention to global existence.

\begin{remark}
    There is good initial reason to expect workable intrinsic conditions for global existence. A naive application of classical global existence results requires strong assumptions on $F$ to compensate for the polynomial growth of the adjoint operator in \eqref{eq_log_transform_Sig_CDE}. By contrast, Proposition \ref{prop_homogeneous_linear_growth} shows that, at least for maps $F$ depending only on the first $N$ signature levels, mere linear growth in the Carnot–Carathéodory norm $\|\cdot\|_{cc}$ (see \cite[Theorem 7.32]{Friz2010}) already suffices. Indeed, if \eqref{eq_homogeneous_lin_growth} holds, then for every $\mathbf{g} \in G^N(\RR^m)$ we may set $\lambda = \|\mathbf{g}\|_{cc}/R$, where $R>0$ depends only on $r$ and $T$, and choose $R$ so that $\delta_{\lambda^{-1}} \mathbf{g} \in S(B^\beta(r))$. It follows that
    \begin{align}
        \label{eq_hlg_equivalent_form}
        |F(\mathbf{g})| = |F(\delta_\lambda\delta_{\lambda^{-1}}\mathbf{g})| \leq C_F\big(1+ \|\mathbf{g}\|_{cc}\, R^{-1}\big)
        \lesssim 1+\|\mathbf{g}\|_{cc},
    \end{align}
    establishing linear growth with respect to the Carnot-Carathéodory norm. Conversely, suppose \eqref{eq_hlg_equivalent_form} holds and fix $r > 0$, together with $\mathbf{g}_0 \in S(B^\beta(r))$. If $F$ depends only on the first $N$ levels, then for every $\lambda > 0$, 
    \[|F(\delta_\lambda \mathbf{g}_0)| = |F(\pi_{0,N} \delta_\lambda \mathbf{g}_0)| \lesssim 1 + \|\pi_{0,N} \delta_\lambda \mathbf{g}_0\|_{cc} \lesssim 1 + \lambda,\]
    where the last estimate follows from the equivalence of homogeneous gauges \cite[Theorem 7.44]{Friz2010}, and the factorial decay estimate \eqref{eq_factorial_decay_signatures}. In particular, the implicit constant depends on $r, N$ and $\beta$. This shows that condition \eqref{eq_homogeneous_lin_growth} is satisfied, suggesting that global existence can be recovered from the corresponding path-level result.
\end{remark}

For $n \in \{1,\hdots, N\}$ and $\mathbf{a} \in \mathbf{1} + \mathfrak{t}^N(\RR^m)$, let $F_{[n]}(\mathbf{a})$ denote the level$-n$ block of $F(\mathbf{a})$; that is, the submatrix consisting of the rows indexed by words of length $n$, i.e. elements of $\cW_n$ (see Section \ref{subsec_path_signatures_tensor_algebras}). As before, we write $F_{[n]}^j$ for the $j-$th column of $F_{[n]}$, and $F^{ij}_{[n]}$ for its $(i,j)-$entry. 

We call $F : \mathbf{1} + \mathfrak{t}^N(\RR^m) \to L(\RR^d,\mathfrak{t}^N(\RR^m))$ \emph{level-wise triangular} if, for every $j \in \{1,\hdots, d\}$ and $n \in \{1,\hdots,N\}$, the map $F_{[n]}^j : \mathbf{1} + \mathfrak{t}^N(\RR^m) \to \mathfrak{t}^N(\RR^m)$ depends only on $\pi_{1,n}(\mathbf{a})$, i.e. $F_{[n]}^j(\mathbf{a}) \equiv F_{[n]}^j\big(\mathbf{a}^{(1)},\hdots, \mathbf{a}^{(n)}\big)$.

\begin{theorem}
\label{prop_intrinsic_global_existence}
    Fix $T>0$ and let $F: \mathbf{1} + \mathfrak{t}^N(\RR^m) \to L(\RR^d,\mathfrak{t}^N(\RR^m))$ be continuous. If $F$ satisfies the homogeneous linear growth condition
    \begin{align}
    \label{eq_strong_homogeneous_linear_growth}
        \| F^{j}(\mathbf{a}) \|_h \leq C_F \big(1 + \| \mathbf{a} \|_h \big), \quad \text{for all} \ j \in \{1,\hdots,d\},
    \end{align}
    and some constant $C_F > 0$, then the CDE \eqref{eq_generalized_lifted_Sig_CDE} admits a global solution. Alternatively, if $F$ is level-wise triangular and there exists $C_F > 0$ such that
    \begin{align}
    \label{eq_weak_homogeneous_linear_growth}
        |F_{[n]}^{j}(\mathbf{a})| \leq C_F \big(1 + \| \pi_{0,n}(\mathbf{a}) \|_h \big), \quad \text{for all} \ j \in \{1,\hdots,d\},
    \end{align}
    and $n \in \{1,\hdots, N\}$, then \eqref{eq_generalized_lifted_Sig_CDE} again admits a global solution. 
\end{theorem}
\begin{proof}
    Since all homogeneous gauges are equivalent \cite[Theorem 7.44]{Friz2010}, we follow \cite{magnani2018lipschitz} and fix
    $$\| \ba \|_h \coloneqq \| \log_{\otimes N} \ba \|_h \coloneqq \left( \sum_{i=1}^\mu (\log_{\otimes N} \ba)_i^{(2N)!/d_i} \right)^{1/(2N)!},$$
    where $\mu \in \NN$ is chosen so that $\mathfrak{t}^N(\RR^m) \cong \RR^\mu$; the index $i \in \{1,\hdots,\mu\}$ labels the graded (word) basis $\big\{ \cW_n : n = 1,\hdots,N \big\}$; and $d_i \in \{1,\hdots,N\}$ denotes the word length of the basis element corresponding to $i$.\footnote{For instance, consider the word basis $\{1,2,11,12,21,22\}$ of $\mathfrak{t}^2(\RR^2)$. Assigning the labels $\{1,2,3,4,5,6\}$ to these basis elements in this order, we have $d_1=d_2=1$, while $d_i=2$ for all $i\geq 3$.} Because \eqref{eq_generalized_lifted_Sig_CDE} admits a global solution if and only if its linearized form \eqref{eq_log_transform_Sig_CDE} does, we prove the result for the latter. For $\Omega \in \mathfrak{t}^N(\RR^m)$, define
    $$H(\Omega) \coloneqq \sum_{k=0}^{N-1} \frac{(-1)^k B_k}{k!} \operatorname{ad}_{\Omega_t}^k(\,\cdot \,) \quad \text{and} \quad \mathcal{H}(\Omega) \coloneqq H(\Omega) \circ F\big(\exp(\Omega)\big).$$
    The map $\mathcal{H}(\Omega)$ is thus a composition of linear operators and hence belongs to  $L(\RR^d,\mathfrak{t}^N(\RR^m))$. More importantly, by \cite[Proposition 1.26]{gb28hardy}, for all $i,k \in \{1,\cdots,\mu\}$ there exists a polynomial $p^i_k : \mathfrak{t}^N(\RR^m) \to \RR$ such that
    \begin{align}
    \label{eq_adjoint_matrix_entries}
       p^i_k (\Omega) =  H^{ik}(\Omega) \quad \text{and} \quad |p^i_k(\Omega)| \lesssim \| \exp_{\otimes N}(\Omega) \|_h^{d_i - d_k},
    \end{align}
    for all $\Omega \in \mathfrak{t}^N(\RR^m)$. In particular, $p^i_k = \delta^i_k$ if $d_i \leq d_k$. Put differently, the $(i,k)-$entry of $H(\Omega)$ is a polynomial in the coordinates of $\Omega$ satisfying the growth bound above, and it vanishes whenever the degree of $k$ exceeds that of $i$. Consequently, the matrix $H(\Omega)$ is lower triangular, making $H$ level-wise triangular in our terminology.

    Let $\Omega: [0,t_{\max}) \to \mathfrak{t}^N(\RR^m)$ be a solution to \eqref{eq_log_transform_Sig_CDE}. Then, for all $t\in [0,t_{\max})$,
    \begin{align*}
        |\Omega_t| \leq |\Omega_0| + \int_0^t |\mathcal{H}(\Omega_u)|\, |\dd x_u| \lesssim |\Omega_0| + \int_0^t \sum_{i=1}^\mu \sum_{k=1}^d |\mathcal{H}^{ik}(\Omega_u)| \, |\dd x_u|.
    \end{align*}
    Furthermore, using \eqref{eq_adjoint_matrix_entries} and \eqref{eq_strong_homogeneous_linear_growth}, 
    \begin{align*}
        |\mathcal{H}^{ik}(\Omega_u)| &= \Big| \sum_{l=1}^\mu H^{il}(\Omega_u) F^{lk}\big(\exp_{\otimes N}(\Omega_u) \big) \Big| \lesssim \sum_{l=1}^\mu \| \exp_{\otimes N}(\Omega_u) \|_h^{d_i - d_l} |F^{lk}\big(\exp_{\otimes N}(\Omega_u) \big)| \\
        &\leq \sum_{l=1}^\mu \| \exp_{\otimes N}(\Omega_u) \|_h^{d_i - d_l} \| F^k\big(\exp_{\otimes N}(\Omega_u)\big) \|_h^{d_l}  \lesssim \sum_{l=1}^\mu \| \Omega_u \|_h^{d_i - d_l} \big(1 + \| \Omega_u \|_h \big)^{d_l}.
    \end{align*}
    Therefore, up to a universal constant dependent on $N$,
    \begin{align*}
        |\Omega_t| \leq |\Omega_0| + \int_0^t \sum_{i=1}^\mu \sum_{k=1}^d \big(1 + \| \Omega_u \|_h^{d_i} \big) \, |\dd x_u| \lesssim |\Omega_0| + \int_0^t \big(1 + \| \Omega_u \|_h^N \big) \, |\dd x_u|,
    \end{align*}
    and the first claim follows from Grönwall's inequality \cite[Lemma 3.2]{Friz2010} together with the ball-box estimate \cite[Proposition 7.49]{Friz2010} — note that $[0,t] \ni u \mapsto \exp(\Omega_u)$ is a bounded set in $\mathbf{1} + \mathfrak{t}^N(\RR^m)$.

    To prove the second part, assume that $F$ is level-wise triangular and that \eqref{eq_weak_homogeneous_linear_growth} holds in place of \eqref{eq_strong_homogeneous_linear_growth}. We proceed by induction on the levels of $\Omega$. For each $n \in \{1,\hdots, N\}$,
    \begin{align*}
        \dd \Omega_t^{(n)} = \mathcal{H}_{[n]}(\Omega_t) \dd x_t, \quad \Omega_0^{(n)} = \pi_{1,n}(\Omega_0)
    \end{align*}
    As such, for $n = 1$, we have
    \begin{align*}
        |\Omega_t^{(1)}| \leq |\Omega_0^{(1)}| + \int_0^t |\mathcal{H}_{[1]}(\Omega_u)| \, | \dd x_u | = |\Omega_0^{(1)}| + \int_0^t |H_{[1]}(\Omega_u) \, F(\exp_{\otimes N}(\Omega_u))| \, | \dd x_u |. 
    \end{align*}
    Since $d_i = 1$ for all $i \in \{1,\hdots,m\}$, it follows that $H^{ik}(\Omega_u) = \delta^i_k$ for all $k \in \{1,\hdots,\mu\}$, and thus
    \begin{align*}
        H_{[1]}(\Omega_u) \, F(\exp_{\otimes N}(\Omega_u)) = F_{[1]}(\exp_{\otimes N}(\Omega_u)).
    \end{align*}
    Hence, by our hypothesis \eqref{eq_weak_homogeneous_linear_growth},
    \begin{align*}
        |\Omega_t^{(1)}| \leq |\Omega_0^{(1)}| + \int_0^t | F_{[1]}(\exp_{\otimes N}(\Omega_u))| \, | \dd x_u | \lesssim |\Omega_0^{(1)}| + \int_0^t \big( 1 + |\Omega_u^{(1)}| \big) \, |\dd x_u|,
    \end{align*}
    where we used the fact that $\| \pi_{1} \exp_{\otimes N}(\Omega_u)\|_h$ coincides with the Euclidean $(2N)!-$norm of $\Omega^{(1)}_u$, and so, up to a constant depending on $N$, $\|\pi_1 \exp_{\otimes N}(\Omega_u)\|_h \lesssim |\Omega_u^{(1)}|$. By Grönwall’s inequality, $\Omega^{(1)}$ is globally defined, establishing the base case.

    Next, assume $\Omega^{(k)}$ exists globally for all $k < n < N$. We show that $\Omega^{(n)}$ does as well. Consider the matrix $H_{[n]}(\Omega_u) F(\exp_{\otimes N}(\Omega_u))$, whose $(i,j)$-entry is given by
    \begin{align*}
        \sum_{k=1}^\mu H^{ik}(\Omega_u) F^{kj}(\exp_{\otimes N}(\Omega_u)) = \sum_{k=1}^{\mu_{n-1}} p^i_k(\Omega_u) F^{kj}(\exp_{\otimes N}(\Omega_u)) + F^{ij}(\exp_{\otimes N}(\Omega_u)),
    \end{align*}
    where $\mu_{n-1}$ denotes the dimension of $\mathfrak{t}^{n-1}(\RR^m)$. Because both $F$ and $H$ are level-wise triangular, the summation term on the right depends only on $\big(\Omega_u^{(1)}, \hdots, \Omega_u^{(n-1)}\big)$, which are globally defined by the induction hypothesis. We may therefore write
    \begin{align*}
        H_{[n]}(\Omega_u) F(\exp_{\otimes N}(\Omega_u)) = Q_u + F_{[n]}(\exp_{\otimes N}(\Omega_u)),
    \end{align*}
    where $Q_u$ depends solely on $\big(\Omega_u^{(1)},\hdots,\Omega_u^{(n-1)}\big)$. Consequently, by \eqref{eq_weak_homogeneous_linear_growth},
    \begin{align*}
        |\Omega_t^{(n)}| \lesssim |\Omega_0^{(n)}| + \sup_{u \in [0,T]}|Q_u| \int_0^t |\dd x_u| +  \int_0^t \big(1 + \| \pi_{0,n} \exp_{\otimes N}(\Omega_u) \|_h \big) \, |\dd x_u|.
    \end{align*}
    Since $\| \pi_{0,n} \exp_{\otimes N}(\Omega_u) \|_h \lesssim |\Omega_u|^{1/n}$ (on bounded sets), the claim follows from the Bihari–LaSalle inequality (non-linear Grönwall lemma) \cite{lasalle1949uniqueness,bihari1956generalization} with modulus $\omega(u) = u^{1/n}$. 
\end{proof}

Several observations are now in order. 

\begin{remark}
\label{rem_hml_special_case}
    To begin with, we emphasize that Theorem \ref{prop_intrinsic_global_existence} extends to any homogeneous group in the sense of Section \ref{subsec_path_signatures_tensor_algebras} (see also \cite[Section 2.1]{magnani2018lipschitz}). In particular, it applies to $G^N(\RR^m)$, making \eqref{eq_homogeneous_lin_growth} a special case of \eqref{eq_strong_homogeneous_linear_growth} when considering maps $F$ that depend only on the first $N$ levels. Indeed, as observed earlier, \eqref{eq_homogeneous_lin_growth} is equivalent to \eqref{eq_hlg_equivalent_form}, and when each $F^j$ takes values in $\RR^m$, the homogeneous gauge $\|F^j(\mathbf{g})\|_h$ reduces to the standard Euclidean norm. This translates the homogeneous linear growth condition \eqref{eq_homogeneous_lin_growth} from the path level to the present setting.
\end{remark}

\begin{remark}
    Additionally, Theorem \ref{prop_intrinsic_global_existence} may be viewed as a generalization of \cite[Theorem 4.1]{magnani2018lipschitz}. From \eqref{eq_geometric_formulation_Sig_CDEs} we have
    \[\dot{S}_t = \sum_{i=1}^\mu a_i(t,S_t) X_i(S_t), \quad \text{where} \quad a_i(t,S_t) \coloneqq \sum_{j=1}^d F^{ij}(S_t) \dd x_t^j\]
    and $X_i(S_t) \coloneqq (L_{S_t})_*(e_i)$ denotes the unique left-invariant vector field satisfying $X_i(\mathbf{1}) = e_i$. It follows immediately that the assumptions in \cite{magnani2018lipschitz} imply our condition \eqref{eq_strong_homogeneous_linear_growth}. The presence of $(1+\|\mathbf{g}\|_h)$ instead of $\|\mathbf{g}\|_h$ is immaterial, and replacing a constant $C_F$ by an integrable function does not affect the argument. Finally, the uniqueness of equilibrium points established in \cite{magnani2018lipschitz} appears simply as a corollary of the underlying \textit{a priori} estimate ensuring global existence. 
\end{remark}

\begin{remark}
    Beyond uniqueness at equilibrium points, \cite{magnani2018lipschitz} establishes uniqueness when the vector field lies in the module generated by an involutive subalgebra; see \cite[Theorem 4.3]{magnani2018lipschitz}. For the homogeneous groups relevant here, namely $\mathbf{1}+\mathfrak{t}^N(\mathbb{R}^m)$ and $G^N(\mathbb{R}^m)$, this condition is extremely restrictive, as the only involutive subalgebras of $\mathbb{R}^m$ are one-dimensional subspaces. In effect, the vector field must point in a single fixed direction. More importantly, however, \cite{magnani2018lipschitz} shows that local uniqueness generally fails under Lipschitz assumptions with respect to a homogeneous gauge, even for \emph{horizontal vector fields} (i.e.~taking values in $L(\mathbb{R}^d,\mathbb{R}^m)$, rather than $L(\mathbb{R}^d,\mathfrak{t}^N(\mathbb{R}^m))$ as in our setting).
\end{remark}

Taken together, the negative uniqueness result of \cite{magnani2018lipschitz} and our Theorem \ref{prop_intrinsic_global_existence} indicate the following distinction. The path-level linear growth condition has a natural \emph{intrinsic} analogue on the group (Proposition \ref{prop_homogeneous_linear_growth} and Remark \ref{rem_hml_special_case}), whereas the continuity condition required for well-posedness (Proposition \ref{prop_well_posedness_SigCDEs}) does not: intrinsic Lipschitz regularity is too weak.

We therefore conclude by showing that, in the finite-dimensional setting considered here, local Lipschitz continuity on paths is equivalent to \emph{extrinsic} Lipschitz continuity on the group, i.e. with respect to a Euclidean metric via tensor coordinates.

\begin{proposition}
    Fix $1/2 < \beta < \alpha \leq 1$. Consider a map $F: T^N(\RR^m) \to L(\RR^d,\RR^m)$. Then $F \circ S^N : \Lambda^{\alpha\text{-Höl}}_\beta([0,T],\RR^m) \to L(\RR^d,\RR^m)$ is locally Lipschitz with respect to $d_{\beta\text{-Höl}}^\Lambda$ if and only if $F|_{G^N(\RR^m)}$ is locally Lipschitz with respect to the Euclidean metric $|\cdot|$ on $T^N(\RR^m)$.
\end{proposition}
\begin{proof}
    If $F|_{G^N(\RR^m)}$ is locally Euclidean-Lipschitz, then Proposition \ref{lem_continuity_sig_map} applied to the continuity of the signature map immediately yields that $F \circ S^N$ is locally Lipschitz with respect to $d^\Lambda_{\beta\text{-Höl}}$. 
    
    For the converse, fix $\mathbf{g} \in G^N(\mathbb{R}^m)$ and choose $\mathbf{g}_1,\mathbf{g}_2 \in G^N(\mathbb{R}^m)$ in a sufficiently small neighbourhood of $\mathbf{g}$. By \cite[Proposition 7.64]{Friz2010}, and using that the argument of this result can be rephrased in terms of the $1-$Hölder seminorm (rather than $1-$variation), there exist Lipschitz paths
    \[y_1,y_2\in C^{1\text{-Höl}}([0,T],\mathbb{R}^m)\subset \Lambda^{\alpha\text{-Höl}}_\beta([0,T],\mathbb{R}^m)\]
    such that $S^N(y_i)_{0,T} = \mathbf{g}_i$ and\footnote{Both paths $y_1$ and $y_2$ can be chosen to start at $0$.}
    \[ \|y_1-y_2\|_{1\text{-Höl};[0,T]} \leq C_{\mathbf{g}} \, |\mathbf{g}_1-\mathbf{g}_2| \]
    for some constant $C_{\mathbf{g}}>0$. By the assumed local Lipschitz continuity of $F\circ S^N$,
    \[|F(\mathbf{g}_1)-F(\mathbf{g}_2)| = \big|F(S^N(y_1)_{0,T}) - F(S^N(y_2)_{0,T})\big|
    \lesssim \|y_1-y_2\|_{\beta\text{-Höl};[0,T]},\]
    where we suppress the dependence on $y$ such that $\mathbf{g} = S(y)_{0,T}$. Since $\beta<1$, we have
    \[\|y_1-y_2\|_{\beta\text{-Höl};[0,T]}
    \lesssim T^{1-\beta} \|y_1-y_2\|_{1\text{-Höl};[0,T]} \lesssim C_{\mathbf{g}} \, |\mathbf{g}_1-\mathbf{g}_2|.\]
    Thus $F|_{G^N(\mathbb{R}^m)}$ is locally Euclidean–Lipschitz at $\mathbf{g}$, completing the proof. 
\end{proof}

\newpage
\appendix

\section{Auxiliary Results}
\label{Apendix}

\subsection{Section \ref{sec_preliminaries}}
\label{Appendix_2}

\begin{proof}[Proof of Lemma \ref{lem_characterising_B_psi}]
    We adapt the argument in \cite[Theorem 2.7]{dorsek2010semigroup} to the current Banach-valued setting. Fix an arbitrary $\varepsilon > 0$. Using \eqref{eq_vanishing_tail_condition}, choose $R>0$ such that 
    \[
    \sup_{x \notin K_R} \frac{\|f(x)\|_Y}{\psi(x)} < \varepsilon.
    \]
    Since $f$ is continuous over $X$ by assumption, $M_R \coloneqq \sup_{x \in K_R} \|f(x)\|_Y < \infty$. Choose $N \geq M_R$ and define 
    \[
    T_N(y) = 
    \begin{cases}
        y, \quad &\| y \|_Y \leq N, \\
        N y/\|y\|_Y, \quad &\| y \|_Y >N,
    \end{cases} 
    \qquad g_N = T_N \circ f.
    \]
    Note that $T_N : Y \to Y$ is continuous, and hence $g_N$ is a continuous, bounded function over $X$. Moreover,  we have that $\|y - T_N(y) \|_Y = \max\{0,\|y \|_Y - N\}$. Consequently,
    \[
    \frac{\|f(x) - g_N(x)\|_Y}{\psi(x)} = \frac{\max\{0,\|f(x)\|_Y - N\}}{\psi(x)} \leq \frac{\|f(x)\|_Y}{\psi(x)},
    \]
    and since $g_N(x) = f(x)$ for every $x \in K_R$, 
    \[
    \|f - g_N \|_{\cB_\psi} =  
    \sup_{x\notin K_R} \frac{\|f(x) - g_N(x)\|_Y}{\psi(x)} \leq \sup_{x \notin K_R} \frac{\| f(x) \|_Y}{\psi(x)} < \varepsilon.
    \]
    This proves that $f \in \cB_\psi(X,Y)$. Finally, for the converse direction, note that the tail condition \eqref{eq_vanishing_tail_condition} follows immediately from $f \in \cB_\psi(X,Y)$; see \cite[Lemma 2.7 (i)]{cuchiero2023global}. If in addition $\psi$ is locally bounded (that is, for every $x\in X$ there is a neighborhood $U \ni x$ and $M<\infty$ with $\sup_{U}\psi\leq M)$), then every $f\in\mathcal B_\psi(X,Y)$ is continuous: approximate $f$ by continuous bounded functions $g_n$ in the $\cB_{\psi}$ norm, and on any neighborhood where $\psi \leq M$ observe that
    \[\|g_n-f\|_\infty\leq M\|g_n-f\|_{\mathcal B_\psi}\to 0.\]
    Thus $f$ is a uniform limit of continuous functions on that neighborhood, so continuous there. Since the point was arbitrary, $f$ is continuous on $X$.
\end{proof}

\begin{proof}[Proof of Proposition \ref{lem_projective_limit_topo}]
    Let $\tau'$ be a topology turning $\varprojlim T_{1,\lambda_k}(\RR^m)$ into a Fréchet space and making all inclusions $\iota_k$ continuous, and let $\tau$ denote the locally convex topology introduced above.

    By construction, $\tau$ is the coarsest topology under which the maps $\iota_k$ are continuous, so we immediately have $\tau \subseteq \tau'$. To establish the reverse inclusion, consider the identity map
    $$
        \text{id}: \big(\varprojlim T_{1,\lambda_k}(\RR^m), \tau \big) \to \big(\varprojlim T_{1,\lambda_k}(\RR^m), \tau' \big).
    $$
    If this map is continuous, the claim follows. Take a sequence $(x_n) \subset \varprojlim T_{1,\lambda_k}(\RR^m)$ such that 
    $$
        x_n \overset{\tau}{ \longrightarrow} x \quad \text{and} \quad x_n \overset{\tau'}{ \longrightarrow} y.
    $$
    
    Since $\tau'$ is finer than $\tau$, the convergence $x_n \to y$ in $\tau'$ implies convergence $x_n \to y$ in $\tau$. But $\big(\varprojlim T_{1,\lambda_k}(\RR^m), \tau\big)$ is Hausdorff (as it is completely metrizable), so necessarily $x = y$. This shows that the graph of $\text{id}$ is closed, and the conclusion follows from the closed graph theorem \cite[$\S$ 15, Theorem 12.3]{kothe1969topological}.
\end{proof}

\begin{proof}[Proof of Lemma \ref{lem_compact_emb_limit_spaces}]
    By Tonelli's theorem we may rewrite, for $r\in\{p,q\}$,
    \[
    \|\ba\|_{(r)}=\sum_{n\ge0}\rho^{(r)}_n|\ba^{(n)}|,
    \qquad 
    \rho^{(r)}_n:=\frac{1}{Z_r}\sum_{k\ge0} \exp(-k^r) \lambda_k^{n}.
    \]
    We claim that $\rho^{(q)}_n/\rho^{(p)}_n\to0$ as $n\to\infty$. Indeed, since $q>p$, for every $\varepsilon>0$ there exists $K\in\NN$ such that
    \[
    \exp(-k^q) \le \varepsilon \exp(-k^p),\qquad k \ge K.
    \]
    Let $L:=\max_{0\le k<K}\lambda_k$ and choose $k_0\ge K$ with $\lambda_{k_0}>L$. Then
    \[
    \rho^{(q)}_n
    \le \frac{1}{Z_q}\Big(C\,L^n+\varepsilon\sum_{k\ge K} \exp(-k^p) \lambda_k^{n} \Big),
    \qquad C \coloneqq \sum_{k=0}^{K-1} \exp(-k^q) < Z_q,
    \]
    while
    \[
    \rho^{(p)}_n \geq \frac{1}{Z_p}\sum_{k\ge K} \exp(-k^p) \lambda_k^{n} \geq \frac{1}{Z_p} \exp(-k_0^p) \lambda_{k_0}^{n}.
    \]
    Dividing the bounds yields
    \[
    \frac{\rho^{(q)}_n}{\rho^{(p)}_n}
    \leq \frac{Z_p}{Z_q}\Big(C \exp(k_0^p) \Big(\frac{L}{\lambda_{k_0}}\Big)^n+\varepsilon\Big),
    \]
    and since $L/\lambda_{k_0}<1$ we obtain 
    \[
    \limsup_{n\to\infty} \rho^{(q)}_n/\rho^{(p)}_n \leq (Z_p/Z_q) \varepsilon.
    \]
    As $\varepsilon>0$ was arbitrary, it follows that $\rho^{(q)}_n/\rho^{(p)}_n \to 0$. Now, let $\mathbf{a} \in \cT^{(p)}(\RR^m)$. Then, 
    \[
    \|\ba-\pi_{0,N}\ba\|_{(q)}=\sum_{n>N}\rho^{(q)}_n|\ba^{(n)}|
    \le \Big(\sup_{n>N} \rho^{(q)}_n/\rho^{(p)}_n\Big)\sum_{n>N}\rho^{(p)}_n|\ba^{(n)}|
    \le \Big(\sup_{n>N} \rho^{(q)}_n/\rho^{(p)}_n\Big) \|\ba\|_{(p)},
    \]
    and taking the supremum over $\|\ba\|_{(p)}\le R$ gives \eqref{eq_tail_estimate}. As for compactness, let $(\ba_m)_{m\geq 1}$ be a bounded sequence in $\cT^{(p)}(\RR^m)$ and fix $N \geq 1$. Then $(\pi_{0,N} \ba_m)$ is bounded in the finite-dimensional space $T^N(\RR^m)$, and therefore admits a convergent subsequence in $\cT^{(q)}(\RR^m)$. A diagonal argument yields a subsequence $(\ba_{m_k})$ such that $(\pi_{0,N} \ba_{m_k})$ converges in $\cT^{(q)}(\RR^m)$ for every $N \geq 1$. Using the tail estimate \eqref{eq_tail_estimate}, we then conclude that this subsequence is Cauchy in $\cT^{(q)}(\RR^m)$, and hence convergent.
\end{proof}

\begin{proof}[Proof of Lemma \ref{lem_continuity_tensor_product}]
    Consider any sequence of positive scalars $(\lambda_k)$ and fix $1<p<q$. Since each space $T_{1,\lambda_k}(\RR^m)$ is a Banach algebra, we have
    \[
    \|\mathbf a \otimes\mathbf b\|_{(q)}
    = \frac{1}{Z_q}\sum_{k\ge0} \exp(-k^q)\|\mathbf a\otimes\mathbf b\|_{1,\lambda_k}
    \le \frac{1}{Z_q}\sum_{k\ge0} \exp(-k^q) \|\mathbf a\|_{1,\lambda_k} \|\mathbf b\|_{1,\lambda_k}.
    \]
    Using the relation between the norms $\| \cdot \|_{1,\lambda_k}$ and $\| \cdot \|_{(p)}$, this yields
    \[
    \|\mathbf a\otimes\mathbf b\|_{(q)}
    \le \frac{Z_p^2}{Z_q}\sum_{k\ge0} \exp(-k^q+2k^p) \|\mathbf a\|_{(p)} \|\mathbf b\|_{(p)},
    \qquad \mathbf a,\mathbf b\in\cT^{(p)}(\RR^m).
    \]
    Since $q>p$, the series $\sum_{k\ge0} \exp(-k^q+2k^p)$ converges, which proves the claim.
\end{proof}

\subsection{Section \ref{sec_path_dependent_cdes}}
\label{Appendix_3}

\begin{lemma}
\label{lem_aux_hale}
    Consider $\alpha \in \left( \frac{1}{2},1\right]$ and $\beta < \alpha$. Let $f: \Lambda^{\alpha\text{-Höl}}_\beta([0,T_1],\RR^m) \to L(\RR^d,\RR^m)$ be a continuous, non-anticipative functional, and for $\delta>0$ let $W \subset [0,T_1-\delta] \times C_\beta^{\alpha\text{-Höl}}([0,T_1],\RR^m)$ be a compact set. Then there exists a neighborhood $V$ of $W$ such that the restriction $f|_{V/_\sim}$ is bounded. More precisely, there exists $M>0$ such that
    \begin{equation}
    \label{eq_lem_aux_hale_bound}
        |f(t,z_{\cdot \land t})| < M, \ \text{for all} \ (t,z_{\cdot \land t}) \in V/_\sim \subset \Lambda^{\alpha\text{-Höl}}_\beta([0,T_1],\RR^m). 
    \end{equation}
    Additionally, there exist constants $0 < \tau < \delta$ and $R>0$ such that, for any $(T_0, w_{\cdot \land T_0}) \in W/_\sim$, we have
    \[
        \left(T_0 + t, (w \sqcup z)_{\cdot \land (T_0 + t)} \right) \in V/_\sim, \ \text{for all} \ t \in [0,\tau] \ \text{and} \ z \in \mathcal{D}(\tau,R),
    \]
    where
    \[
        \mathcal{D}(\tau,R) := \left\{z \in C^{\alpha\text{-Höl}}([0,\tau],\RR^m) : z_0 = 0,\ \|z\|_{\alpha\text{-Höl};[0,\tau]} \leq R \right\}. 
    \]
\end{lemma}
\begin{proof}
    Since $W \subset [0,T_1-\delta] \times C_\beta^{\alpha\text{-Höl}}([0,T_1],\RR^m)$ is compact, the quotient $W/_\sim$ is compact in the topology induced by the metric $d^\Lambda_{\beta\text{-Höl};[0,T_1]}$. Moreover, the continuity of the non-anticipative functional $f: \Lambda^{\alpha\text{-Höl}}_\beta([0,T_1],\RR^m) \to L(\RR^d,\RR^m)$ implies the existence of a constant $M > 0$ such that
    \[
    \sup_{(T_0, w_{\cdot \land T_0}) \in W/_\sim} |f(T_0, w_{\cdot \land T_0})| < M,
    \]
    By continuity again, for each $(T_0, w) \in W$, there exist $\tilde{\tau}_{(T_0,w)} > 0$ and $\tilde{R}_{(T_0,w)} > 0$ such that
    \begin{align}
    \label{eq_aux_lemma_1}
        \big|f(T_0 + t, (w + v)_{\cdot \land (T_0 + t)})\big| < M, \ \text{for all} \ t \in [0,\tilde{\tau}_{(T_0,w)}] \ \text{and} \ v \in B_{\tilde{R}_{(T_0,w)}},
    \end{align}
    where
    \[
        B_{\tilde{R}_{(T_0,w)}} \coloneqq \left\{ v \in C^{\alpha\text{-Höl}}_\beta([0,T_1], \RR^m) : |v_0| + \|v\|_{\beta\text{-Höl};[0,T_1]} \leq \tilde{R}_{(T_0,w)} \right\}.
    \]
    Using the compactness of $W$, we may extract finitely many such neighborhoods and choose constants $\tilde{\tau} > 0$ and $\tilde{R} > 0$ for which inequality \eqref{eq_aux_lemma_1} holds uniformly over all $(T_0,w) \in W$. Define 
    $$
    V \coloneqq \Big\{ (T_0 + t, w + v) : (T_0, w) \in W,\ t \in [0,\tilde{\tau}],\ \text{and} \ v \in B_{\tilde{R}} \Big\}.
    $$
    Then $f|_{V/_\sim}$ is bounded by $M$, which establishes the first part of the lemma.

    We now prove the second claim. Specifically, this amounts to showing that for suitable choices of $\tau,R > 0$, we have the identity
    $$(w \sqcup z)_{\cdot \land (T_0 + t)} = (w + \tilde{z})_{\cdot \land (T_0 + t)}, \ \text{for} \ t \in [0,\tau]  \ \text{and} \ z \in \cD(\tau,R),$$
    where $\tilde{z} \in C^{\alpha\text{-Höl}}_\beta([0,T_1],\RR^m)$ satisfies $|\tilde{z}_0| + \| \tilde{z} \|_{\beta\text{-Höl};[0,T_1]} \leq \tilde{R}$. To verify this, take $R < \tilde{R}$ and $\tau < \tilde{\tau}$. Then, for each $z \in \cD(\tau,R)$, define $\tilde{z}$ by
    $$
        \tilde{z}_t = 
        \begin{cases}
            0, \quad &\text{for} \ t \in [0,T_0] \\
            z_{t - T_0} - (w_t - w_{T_0}), \quad &\text{for} \ t \in [T_0,T_0 + \tau] \\
            z_\tau - (w_{T_0 + \tau} - w_{T_0}), \quad &\text{for} \ t\in [T_0+\tau,T_1]
        \end{cases}.
    $$
    Clearly, $(w \sqcup z)_{\cdot \land (T_0 + t)} = (w + \tilde{z})_{\cdot \land (T_0 + t)}$ for $t\in [0,\tau]$. We now show that $|\tilde{z}_0| + \| \tilde{z} \|_{\beta\text{-Höl};[0,T_1]} \leq \tilde{R}$ for an appropriately chosen $\tau$. Indeed, observe that
    \begin{align*}
        \| \tilde{z} \|_{\beta\text{-Höl};[0,T_1]} = \| \tilde{z} \|_{\beta\text{-Höl};[T_0,T_0 + \tau]} \leq& \| z\|_{\beta\text{-Höl};[0,\tau]} + \| w \|_{\beta\text{-Höl};[T_0,T_0+\tau]} \\ 
        \leq& \tau^{\alpha-\beta} \Big( R +  \| w \|_{\alpha\text{-Höl};[T_0,T_0 + \tau]} \Big),
    \end{align*}
    which can be made strictly less than $\tilde{R}$ by choosing $\tau$ sufficiently small. Moreover, the choice of such a $\tau$ can be made uniformly over $(T_0,w) \in W$ by a compactness argument. Hence, $|\tilde{z}_0| + \| \tilde{z} \|_{\beta\text{-Höl};[0,T_1]} \leq \tilde{R}$, and thus $\big(T_0 +t, (w \sqcup z)_{\cdot \land (T_0 + t)} \big) \in V/_\sim$ for $t\in [0,\tau]$ and $z \in \cD(\tau,R)$.
\end{proof}

\begin{remark}
\label{rem_function_neighbourhood}
    In the spirit of \cite[Lemma 2.2]{hale2013introduction}, note that Lemma \ref{lem_aux_hale} in fact holds in greater generality: namely, the upper bound \eqref{eq_lem_aux_hale_bound} remains valid on a neighborhood $U \subset C^0_b(V/_\sim, L(\RR^d,\RR^m))$ of $f|_{V/_\sim}$. Indeed, by continuity of $f$, there exists $\varepsilon > 0$ such that 
    $$
      |f(t,z_{\cdot \land t})| < M - \varepsilon, \quad \text{for all} \ (t,z_{\cdot \land t}) \in V/_\sim.
    $$
    Consequently, for every $\tilde{f} \in C^0_b(V/_\sim, L(\RR^d,\RR^m))$ satisfying $\| f|_{V/_\sim} - \tilde{f} \|_\infty < \varepsilon$, the estimate \eqref{eq_lem_aux_hale_bound} continues to hold with $f$ replaced by $\tilde{f}$.
\end{remark}

\subsection{Section \ref{sec_sig_cdes}}
\label{Appendix_4}

\begin{lemma}
\label{lem_aux_time_augmentation}
    The time-augmentation map 
    $$\widehat{\cdot} : \Lambda^\suphol{\alpha}_\beta([0,T_1],\RR^m) \to \widehat{\Lambda}^\suphol{\alpha}_\beta([0,T_1],\RR^{m+1}), \quad z|_{[0,t]} \mapsto (\operatorname{id}|_{[0,t]}, z|_{[0,t]})$$  
    is a homeomorphism. Consequently, for any weight function $\hat{\psi}$ on $\widehat{\Lambda}^\suphol{\alpha}_\beta([0,T_1],\RR^{m+1})$, the function $\psi(z|_{[0,t]}) \coloneqq \hat{\psi}(\hat{z}|_{[0,t]})$ is an admissible function on $\Lambda^\suphol{\alpha}_\beta([0,T_1],\RR^m)$. Moreover, for every $f\in \cB_{\psi}(\Lambda^\suphol{\alpha}_\beta([0,T_1],\RR^m))$, the lifted function $\hat{f}(\hat{z}|_{[0,t]}) \coloneqq f(z|_{[0,t]})$ satisfies
    $$\| f \|_{\cB_{\psi}}  = \| \hat{f} \|_{\cB_{\hat{\psi}}}.$$
    That is, the lifting operator $f \mapsto \hat{f}$ is an isometry between $\cB_{\psi}$ and the subspace of $\cB_{\hat{\psi}}$ consisting of functionals that ignore the time component. 
\end{lemma}
\begin{proof}
    The fact that $\ \widehat{\cdot} \ $ is a bijection with inverse $\hat{z}|_{[0,t]} \mapsto z|_{[0,t]}$ is immediate. To check that it is a homeomorphism, observe that there exists a constant $C_m > 0$ (depending only on $m$) such that 
    $$ \| \hat{x}|_{[0,t]} - \tilde{\hat{y}}|_{[0,t]} \|_{\beta\text{-Höl}} \leq C_m \Big( |t-s|^{1-\beta} + \| x|_{[0,t]} - \tilde{y}|_{[0,t]} \|_{\beta\text{-Höl}} \Big),$$
    for all $\hat{x}|_{[0,t]}, \hat{y}|_{[0,s]} \in \widehat{\Lambda}^\suphol{\alpha}_\beta([0,T_1],\RR^{m+1})$ with $s \leq t$. Additionally, one trivially has 
    $$\| x|_{[0,t]} - \tilde{y}|_{[0,t]} \|_{\beta\text{-Höl}} \leq  \| \hat{x}|_{[0,t]} - \tilde{\hat{y}}|_{[0,t]} \|_{\beta\text{-Höl}}.$$
    Hence, $d^\Lambda_{\beta\text{-Höl}} \big( x|_{[0,t]} , y|_{[0,s]} \big) \to 0$ if and only if $d^\Lambda_{\beta\text{-Höl}} \big( \hat{x}|_{[0,t]} , \hat{y}|_{[0,s]} \big) \to 0$. The rest of the claim follows straightforwardly.
\end{proof}

\begin{proof}[Proof of Proposition \ref{lem_weaker_continuity_controls}]
     Fix some $\eta \in (\max\{\alpha,1-\alpha+\beta\},1)$. Let $(x^{(n)}) \subset C^{1\text{-Höl}}_\eta([0,T_1],\RR^d)$ be a sequence such that $\|x - x^{(n)} \|_{\eta\text{-Höl};[0,T_1]}\to 0$. Denote by $y_x,(y_{x^{(n)}})$ the unique solutions of  \eqref{eq_target_path_dependent_cde} driven by $x$ and $(x^{(n)})$, respectively. We aim to prove that 
    \begin{equation*}
        \| y_x - y_{x^{(n)}}\|_{\beta \text{-Höl};[0,T_1]}\to 0.
    \end{equation*}
  
    Since $f$ is globally Lipschitz with respect to $d^\Lambda_{\beta\text{-Höl};[0,T_1]}$ (see Section \ref{sec_preliminaries}), we have
    \begin{align*}
        |f(y_x|_{[0,t]}) - f(y_x|_{[0,s]})| \lesssim   |t-s| + d_{\beta\text{-Höl}}(y_x|_{[0,t]}, \tilde{y}_x|_{[0,t]}) \leq |t-s| + \|y_x|_{[s,t]}\|_{\alpha\text{-Höl}}|t-s|^{\alpha-\beta},
    \end{align*}
    for all $0 \leq s \leq t \leq T_1$, showing that $u \mapsto f(y_x|_{[0,u]})$ is $(\alpha-\beta)-$Hölder continuous. In particular,
    \begin{equation}
    \label{eq_f_holder_bound}
        \| f(y_x|_{[0,\cdot]}) \|_{(\alpha-\beta)\text{-Höl};[s,t]} \lesssim 1 + \| y_x \|_{\alpha\text{-Höl};[s,t]},
    \end{equation}
    and the same holds true if we replace $y_x$ with $y_{x^{(n)}}$. Next, we derive a uniform $\alpha-$Hölder bound for all solutions. Because $\eta + \alpha - \beta > 1$, Young’s inequality \cite[Equation (4.3)]{friz2014course} gives
    \begin{align*}
        \big| (y_{x^{(n)}})_{s,t} \big|& \lesssim \big| f(y_{x^{(n)}}|_{[0,s]}) \big| \|x^{(n)}\|_{\eta\text{-Höl};[s,t]}|t-s|^\eta \\
        &+ \| f(y_{x^{(n)}}|_{[0,\cdot]}) \|_{(\alpha-\beta)\text{-Höl};[s,t]} \|x^{(n)}\|_{\eta\text{-Höl};[s,t]} |t-s|^{\eta+\alpha-\beta},
    \end{align*}
    up to a constant dependent on $\eta,\alpha,\beta$, for all $\sigma \leq s < t \leq \tau$ and $[\sigma,\tau] \subset [T_0,T_1]$. Moreover, by the linear growth of $f$ and estimate \eqref{eq_f_holder_bound}, we see that
    \[
        \big| (y_{x^{(n)}})_{s,t} \big| \lesssim  \big( 1 + \| y_{x^{(n)}} \|_{\alpha\text{-Höl};[0,s]} \big) \|x^{(n)}\|_{\eta\text{-Höl};[s,t]}|t-s|^\eta,
    \]
    and so, dividing by $|t-s|^\alpha$ on both sides and taking the supremum over $\sigma\leq s < t \leq \tau$, yields
    \[
        \| y_{x^{(n)}} \|_{\alpha\text{-Höl};[\sigma,\tau]} \lesssim  \big( 1 + \| y_{x^{(n)}} \|_{\alpha\text{-Höl};[0,\tau]} \big) \|x^{(n)}\|_{\eta\text{-Höl};[0,T_1]}|\tau-\sigma|^{\eta-\alpha}.
    \]
    Crucially, since $\sup_n \|x^{(n)}\|_{\eta\text{-Höl};[0,T_1]} < \infty$, the estimate above can be put into the form \eqref{eq_pasting_argument_lem_hyp} with constants $C_0,C_1$ independent of $n$, and hence Lemma \ref{lem_pasting_argument_bound} guarantees the existence of a constant $K > 0$ such that\footnote{In the present setting, we apply Lemma \ref{lem_pasting_argument_bound} with $\gamma=\alpha$ and we note that the exponent $1-\alpha$ can be replaced by $\eta - \alpha$ which is positive by assumption.}
    \[
    \sup_n \| y_{x^{(n)}} \|_{\alpha\text{-Höl};[0,T_1]} < K.
    \]

    Finally, by the compact embedding of Hölder spaces, every subsequence of $(y_{x^{(n)}})$ has a further subsequence, still denoted by $(y_{x^{(n)}})$, such that $\| y_{x^{(n)}} - \tilde{y} \|_{\beta\text{-Höl};[0,T_1]} \to 0$ for some $\alpha-$Hölder continuous path $\tilde y$. By the Lipschitz assumption on $f$, we have
    \[
    \| f(y_{x^{(n)}}|_{[0,\cdot]}) - f(\tilde y|_{[0,\cdot]}) \|_{\infty;[0,T_1]} \lesssim \| y_{x^{(n)}} - \tilde{y} \|_{\beta\text{-Höl};[0,T_1]} \to 0,
    \]
    and using \eqref{eq_f_holder_bound}, we see that $\sup_n \|f(y_{x^{(n)}}|_{[0,\cdot]}) \|_{(\alpha-\beta)\text{-Höl};[0,T_1]} \lesssim 1 + K < \infty$. Therefore, by interpolation of Hölder spaces
    \[
    \|f(y_{x^{(n)}}|_{[0,\cdot]}) - f(\tilde y|_{[0,\cdot]}) \|_{\mu\text{-Höl};[0,T_1]} \lesssim (1+K)^{\mu/(\alpha-\beta)}  \| f(y_{x^{(n)}}|_{[0,\cdot]}) - f(\tilde y|_{[0,\cdot]}) \|_{\infty;[0,T_1]}^{1-\mu/(\alpha-\beta)} \to 0
    \]
    for some $1 - \eta < \mu < \alpha - \beta$; note that $\eta + \alpha - \beta > 1$. By continuity of Young integration (see \cite[Proposition 6.11]{Friz2010}), and since $\mu + \eta > 1$, it then follows that
    \[
    \int_{T_0}^t f(y_{x^{(n)}}|_{[0,s]}) \dd x^{(n)}_s \longrightarrow \int_{T_0}^t f(\tilde{y}|_{[0,s]}) \dd x_s,
    \]
    for every $t \in [T_0,T_1]$. Proposition \ref{prop_uniqueness} implies that $\tilde{y} = y_x$, and given that every subsequence of $(y_{x^{(n)}})$ has a further subsequence converging to $y_x$ we conclude that $\|y_x - y_{x^{(n)}}\|_{\beta\text{-Höl};[0,T_1]} \to 0$.
\end{proof}

\end{document}